\documentclass[11pt]{article}

\usepackage{verbatim,latexsym,amsfonts,amsmath,amssymb,graphicx,fancyhdr,hyperref,bbm
}
\usepackage{appendix,latexsym,amsfonts,amsmath,amssymb,graphicx,hyperref,amsthm,soul,verbatim,authblk}
\usepackage[framemethod=tikz]{mdframed}

\newcommand{\EE}{\mathbb{ E}}
\newcommand{\PP}{\mathbb{P}}

\newcommand{\R}{\mathbb{R}}
\newcommand{\FF}{\mathbb{F}}
\newcommand{\C}{\mathbb{C}}
\newcommand{\Q}{\mathbb{Q}}

\newcommand{\HH}{\mathbb{H}}
\newcommand{\N}{\mathbb{N}}

\newcommand{\Z}{\mathbb{Z}}

\newcommand{\A}{\mathbb{A}}

\newcommand{\pa}{\partial}

\newcommand{\F}{{\cal F}}

\newcommand{\ind}{\mathbbm{1}}

\def\eps{\varepsilon}
\def\til{\widetilde}
\def\ha{\widehat}
\def\sem{\setminus}
\def\lin{\overline}
\def\ulin{\underline}

\def\Up{\Upsilon}

 \DeclareMathOperator{\Cont}{Cont}
 \DeclareMathOperator{\supp}{supp}
\DeclareMathOperator{\rad}{rad}
 \DeclareMathOperator{\diam}{diam}
\DeclareMathOperator{\dist}{dist} 
 \DeclareMathOperator{\id}{id}
\DeclareMathOperator{\Imm}{Im }

 \DeclareMathOperator{\doub}{doub}
 \DeclareMathOperator{\aaa}{a}
 
\DeclareMathOperator{\ii}{i} \DeclareMathOperator{\oo}{o}

\theoremstyle{plain}
\newtheorem{Theorem}{Theorem}[section]
\newtheorem{Lemma}[Theorem]{Lemma}
\newtheorem{Corollary}[Theorem]{Corollary}
\newtheorem{Proposition}[Theorem]{Proposition}

\theoremstyle{definition}
\newtheorem{Definition}[Theorem]{Definition}
\newtheorem{Remark}[Theorem]{Remark}
\newtheorem{Example}[Theorem]{Example}
\numberwithin{equation}{section}
\newcommand{\BGE}{\begin{equation}}
\newcommand{\BGEN}{\begin{equation*}}
\newcommand{\EDE}{\end{equation}}
\newcommand{\EDEN}{\end{equation*}}

\def\st{\stackrel}

\begin{document}
\title{Boundary Green's functions and Minkowski content  measure
of multi-force-point SLE$_\kappa(\ulin\rho)$}
\author{Dapeng Zhan
}
\affil{Michigan State University}
\date{\today}
\maketitle

\begin{abstract}
  We consider a transient chordal  SLE$_\kappa(\rho_1,\dots,\rho_m)$ curve $\eta$ in $\HH$  from $w$ to $\infty$ with force points $ v_1> \cdots >v_m$ in $(-\infty,w^-]$, which intersects and is not boundary-filling on $(-\infty,v_m)$. The main result is that there is an atomless locally finite Borel measure $\mu_\eta$ on $\eta\cap (-\infty,v_m]$ such that for any $v<v_m$, the $d$-dimensional Minkowski content of $\eta\cap [v,v_m]$ exists and equals $\mu_\eta [v,v_m]$, where $d=\frac{(\sum \rho_j+4)(\kappa-4-2\sum \rho_j)}{2\kappa}$ is the Hausdorff dimension of $\eta\cap [v,v_m]$. In the case that all $\rho_j=0$, this measure agrees with the covariant measure derived in \cite{cov} for chordal SLE$_\kappa$ up to a multiplicative constant.
  We call such measure a Minkowski content measure, extend it to a class of subsets of $\R^n$, and prove that they satisfy conformal covariance.
  To construct the Minkowski content measure on $\eta\cap [v,v_m]$, we follow the standard approach to derive the existence and estimates of the one- and two-point boundary Green's functions of $\eta$  on $(-\infty,v_m)$, which are the limits of the rescaled probability that $\eta$ passes through small discs or open real intervals centered at  points on $(-\infty,v_m)$.
\end{abstract}

\tableofcontents

\section{Introduction}
The Schramm-Loewner evolution (SLE)  introduced by Oded Schramm   is a
one-parameter ($\kappa\in(0,\infty)$) family of random fractal curves growing in plane domains, which satisfy conformal invariance and domain Markov property. Due to its close relation with  two-dimensional statistical lattice models, Brownian motion, Gaussian free field and Liouville quantum gravity, SLE has  received a lot of attention over the past two decades. We refer the reader to Lawler's textbook \cite{Law-SLE} for basic properties of SLE.

Hausdorff dimension and Minkowski content are central objects in the study of the fractal properties of SLE. It is known that the Hausdorff dimension of an SLE$_\kappa$ curve is $2\wedge(1+\frac\kappa 8)$ (\cite{RS,Bf}), and the corresponding Minkowski content of any bounded SLE$_\kappa$ arc exists (\cite{LR}). The Hausdorff dimension of the intersection of an SLE$_\kappa$ curve with the domain boundary is $0\vee(1\wedge (2-\frac 8\kappa))$  (\cite{dim-real}), and the corresponding Minkowski content of the intersection restricted to any compact boundary arc also exists (\cite{Mink-real}).

The Minkowski content of SLE agrees with the natural parametrization of SLE developed in \cite{LS,LZ}, which was constructed  using the Doob-Meyer decomposition.  {Alberts and Sheffield} used that idea to construct a conformally covariant measure on the intersection of a {\it chordal} SLE$_\kappa$ (one important type of SLE) curve with $\R$, and conjectured that their measure has close relation with the Minkowski content of the intersection (\cite{cov}).

In this paper, we work on chordal SLE$_\kappa(\ulin\rho)$, which are natural variants of chordal SLE$_\kappa$. The growth of a chordal SLE$_\kappa(\ulin\rho)$ curve  is affected by the force values $\ulin\rho=(\rho_1,\dots,\rho_m)$ and force points $(v_1,\dots,v_m)$ in addition to the initial point and target point of the curve. One-force-point SLE$_\kappa(\rho)$ was introduced in \cite{LSW-8/3} for the construction of conformal restriction measures. Multi-force-point SLE$_\kappa(\ulin\rho)$ was extensively studied in \cite{MS1} in the framework of imaginary geometry. Using imaginary geometry, the authors of \cite{MW} proved that for a (one-force-point) chordal SLE$_\kappa(\rho)$ curve $\eta$ in $\HH$ started from $w$ with force point $v<w$, the Hausdorff dimension of $\eta\cap (-\infty,v)$ is $\frac{(\rho +4)(\kappa-4-2\rho )}{2\kappa}$, when  $\rho\in((-2)\vee (\frac\kappa 2-4),\frac\kappa 2-2)$.

The goal of this paper is to prove the existence of the Minkowski content of the intersection of a multi-force-point chordal SLE$_\kappa(\ulin\rho)$ curve $\eta$ with some boundary arc of the domain. We assume that the curve $\eta$ grows in $\HH=\{z\in\C:\Imm z>0\}$ from a point $w\in\R$ to $\infty$, and the force points lie on the same side of $w$, say $v_m<\cdots<v_1<w$. Moreover, we assume that the force values are such that a.s.\ $\eta$ intersects but does not contain $(-\infty,v_m)$. By the work of \cite{MW},  the Hausdorff dimension of $\eta\cap(-\infty,v_m]$ is $d:=\frac{(\rho_\Sigma +4)(\kappa-4-2\rho_\Sigma )}{2\kappa}$, where $\rho_\Sigma:=\sum_{j=1}^m\rho_j$.
We are interested in the $d$-dimensional Minkowski content of $\eta\cap I$ for intervals $I\subset (-\infty,v_m]$.

To establish the Minkowski content, we follow the standard approach in   \cite{LR} and \cite{Mink-real} to study the existence of the one-point  Green's function
$$G_1(z_1):=\lim_{r_1\to 0^+} r_1^{d-1} \PP[\dist(z_1,\eta\cap \R)\le r_1]$$
and the two-point  Green's function
$$G_2(z_1,z_2):=\lim_{r_1,r_2\to 0^+} r_1^{d-1} r_2^{d-1}\PP[\dist(z_j,\eta\cap \R)\le r_j,j=1,2],$$
for $z_1\ne z_2\in (-\infty,v_m)$, and derive some estimates.

The first result on SLE Green's function in the literature (cf.\ \cite{Law-Green}) is about the one-point Green's function for a chordal SLE at an interior point, but with the Euclidean distance replaced by the conformal radius.  The two-point conformal radius Green's function for a chordal SLE at interior points was proved to exist in \cite{LW}. The existence of the Eucliean distance version (the original definition) of these Green's functions was proved in \cite{LR}. The one- and two-point Green's functions for a chordal SLE at boundary points were derived in \cite{Mink-real}.

We use the ideas in \cite{Mink-real} to study the boundary Green's functions of a chordal SLE$_\kappa(\ulin\rho)$ curve $\eta$.
Due to the force points, the proofs here involve more technical details than those in \cite{Mink-real}. We prove the existence of the one-point and two-point Green's functions for $\eta$ on $(-\infty,v_m)$ as well as some upper bounds and convergence rates. We also obtain the exact formula of the one-point Green's function up to a multiplicative constant.

Using the   estimates of these Green's functions, we then prove that, for any bounded measurable set $S\subset (-\infty,v_m]$, the limit
$\lim_{r\to 0^+} r^{1-d} \lambda(\{u\in S: \dist(u,\eta\cap\R)\le r\}) $
converges  both a.s.\ and in $L^2$, where $\lambda$ is the Lebesgue measure on $\R$. We let $\xi_S$ denote the limit and prove that whenever $\pa S$ has zero Minkowski content, $\xi_S$ is the  Minkowski content of $\eta\cap {S}$.

Furthermore, we use these $\xi_S$ to prove that there is a random measure $\mu_\eta$ supported by $\eta\cap (-\infty,v_m]$ such that for any compact interval $I\subset (-\infty, v_m]$, $\mu_\eta(I)=\xi_I$. We call such $\mu_\eta$ the Minkowski content measure on $\eta\cap (-\infty,v_m]$. The Minkowski content measures could be defined on a more general family of subsets of Euclidean spaces. We prove that they satisfy the  conformal covariance and restriction property, and use these properties to show that the increasing function $\mu_\eta(\eta[0,t])$ could act as one component in the Doob-Meyer decomposition of some supermartingale. When all $\rho_j=0$, the SLE$_\kappa(\ulin\rho)$ reduces to the standard chordal SLE$_\kappa$. The above statement shows that the $\mu_\eta$ in this special case agrees with the covariant measure developed in \cite{cov} for chordal SLE$_\kappa$, and so proves the conjecture there.

The rest of the paper is organized as follows. In Section \ref{Section 2} we review chordal Loewner equations, SLE$_\kappa(\ulin\rho)$ processes, extremal length, and  a family of Bessel-like diffusion processes, which will be needed later in this paper. In Section \ref{Section 3} we derive some one-point estimates of the  SLE$_\kappa(\ulin\rho)$  curve $\eta$, i.e., some upper bounds of the probability that $\eta$ gets near a marked point in $(-\infty,v_m]$. In Sections \ref{Section 4} and \ref{Section 5} we obtain the existence and some estimates of the one- and two-point Green's functions for $\eta$ on $(-\infty,v_m)$. In Section \ref{Section 6.1} we define the notion of Minkowski content measure, and develop a framework of this new concept. Finally, in Section \ref{Section 6.2} we   use the Green's functions from Sections \ref{Section 4} and \ref{Section 5} to construct the Minkowski content measure on $\eta\cap (-\infty,v_m]$ and prove the conjecture of \cite{cov}.

\section{Preliminaries} \label{Section 2}
\subsection{Chordal Loewner equation}\label{Section 2.1}
A bounded  relatively closed subset $K$ of $\HH$ is called an $\HH$-hull if  $\HH\sem K$ is a simply connected domain.
 For an $\HH$-hull $K$, there is a unique conformal map $g_K$ from $\HH\sem K$ onto $\HH$ such that $g_K(z)=z+ O(\frac 1{z })$ as $z\to \infty$.
If $K=\emptyset$, $g_K=\id$. Suppose $K\ne\emptyset$. Let $K^{\doub}=\lin K\cup\{\lin z:z\in {\lin K}\}$, where $\lin K$ is the closure of $K$, and $\lin z$ is the  complex conjugate of $z$. By Schwarz reflection principle, $g_K$ extends {conformally} to $\C\sem K^{\doub}$. {Let $S_K=\C\sem g_K(\C\sem K^{\doub})$. Then $S_K$ is a compact subset of $\R$.} Let $a_K=\min(\lin{K}\cap \R)$, $b_K=\max(\lin{K}\cap\R)$,  $c_K=\min S_K$, $d_K=\max S_K$. Then  $g_K$ maps  $\C\sem (K^{\doub}\cup [a_K,b_K])$ conformally onto $\C\sem [c_K,d_K]$   {(cf.\ \cite[Section 5.2]{LERW})}.

Let $W \in C([0,T),\R)$ for some $T\in(0,\infty]$. The chordal Loewner equation driven by $W$  is
$$\pa_t g_t(z)=  \frac{2}{g_t(z)-W(t)},\quad 0\le t<T;\quad g_0(z)=z.$$
For every $z\in\C$, let $\tau^*_z$ be the first time that the solution $t\mapsto g_t(z)$ blows up; if such time does not exist, then set $\tau^*_z=\infty$.
For  $t\in[0,T)$, let $K_t=\{z\in\HH:\tau^*_z\le t\}$. It turns out that for each $t\ge 0$, $K_t$ is an $\HH$-hull, $K_0=\emptyset$, $K_t\ne\emptyset$ for $t>0$, and $g_t$ agrees with $g_{K_t}$. We call $g_t$ and $K_t$ the chordal Loewner maps and hulls, respectively, driven by $W$.  We write $g^W_t$ and $K^W_t$ to emphasize the dependence of $g_t$ and $K_t$ on $W$, when necessary. Let  $c^W_t=c_{K^W_t}$ and $d^W_t=d_{K^W_t}$ for $t>0$; and $c^W_0=d^W_0=W(0)$.

\begin{Definition}
For $w\in\R$, let $\R_{w}$ be the set $\R\sem\{w\}\cup \{w^+,w^-\}$ {, where $w^+$ and $w^-$ are two symbols lying outside $\R$, which satisfy the inequalities $w^-<w<w^+$,  $w^+<x$, and $w^->y$ for any $x,y\in\R$ with $x>w>y$}.
Let $W,g^W_t,\tau^*_z,c^W_t,d^W_t$ be as above. The modified chordal Loewner maps $\ha g^W_t$ driven by $W$ are functions defined on $\R_{W(0)}$ for $0\le t<T$ such that
$$\ha g^W_t(v)=\left\{\begin{array}{ll} g^W_t(v), &\mbox{if }0\le t<\tau^*_v; \\ d^W_t, & \mbox{if }t\ge \tau^*_v \mbox{ and }v\ge W(0)^+;\\ c^W_t , & \mbox{if }t\ge \tau^*_v\mbox{ and }v\le W(0)^-.
\end{array}
\right.$$
Here $\tau^*_{W(0)^\pm}$ are understood as $0$, and so $\ha g^W_t(W(0)^+)=d^W_t$ and $\ha g^W_t(W(0)^-)=c^W_t$.
\label{Def-Rw}
\end{Definition}

 {We remark that the $\ha g^W_t$ agrees with the $g^{w}_{K}$ in \cite[Definition 2.11]{Two-Green-boundary} when $w=W(0)$ and $K=K^W_t$. It is useful to describe the force point processes of SLE$_\kappa(\ulin\rho)$ with boundary force points. See the next subsection.}

For $z_0\in \C$ and $S\subset\C$, let $\rad_{z_0}(S)=\sup\{|z_0-z|:z\in S\cup \{z_0\}\}$. The following proposition is well known (cf.\ \cite{Law-SLE}).

\begin{Proposition}
  Let $K_t$ and $g_t$, $0\le t\le \delta$, be chordal Loewner hulls and maps driven by $W$. Suppose $\rad_{ {W(0)}}( K_\delta)\le r$ for some $x_0\in\R$ and $r>0$. Then we have
  \begin{enumerate}
  \item [(i)] $|  W(\delta) - W(0)|\le 2r$;
  \item [(ii)] $|g_{\delta}(z)-z|\le 3r$ for any $z\in\lin\HH\sem \lin { { K_\delta}}$;
  \item [(iii)] $|\ln (g_\delta'(z))|\le 5 (\frac r{|z-x_0|})^2$ for any $z\in\lin\HH$ with $|z-x_0|\ge 10 r$.
\end{enumerate}
\label{basic-chordal}
\end{Proposition}

We use $f_t=f^W_t$ to denote the continuation of $(g^W_t)^{-1}$ from $\HH$ to $\lin\HH$, when the continuation exists. If for every $t\in[0,T)$,  $f^W_t$ exists, and $\eta(t):=f^W_t({W(t)})$ is continuous on $[0,T)$, then we say that $\eta$ is the chordal Loewner curve driven by $W$. Such $\eta$ may not exist in general; and when it exists, $\eta$ is a curve in $\lin\HH$ started from $W(0)$, and $\HH\sem K_t$ is the unbounded connected component of $\HH\sem \eta[0,t]$ for all $t$.

\subsection{Chordal SLE$_\kappa({\protect\ulin\rho})$ processes}\label{Section 2.2}
When the driving function $W$ equals $\sqrt\kappa B(t)$, where $\kappa>0$ and $B$ is a standard Brownian motion, the chordal Loewner curve driven by $W$ is known to exist (\cite{RS}) and is called a chordal SLE$_\kappa$ curve in $\HH$ from $0$ to $\infty$.
The chordal SLE$_\kappa(\ulin\rho)$ processes  are natural variants of chordal SLE$_\kappa$, where one keeps track of additional marked points.  We now review the definition and properties of chordal SLE$_\kappa(\ulin \rho)$ developed in \cite{MS1} that will be needed in this paper.

Let $\kappa>0$, $m\in\N$, $\rho_1,\dots,\rho_m\in\R$, $w\in\R$.   Let $v_1,\dots,v_m\in \R_w=\R\sem\{w\}\cup\{w^+,w^-\}$.  We require that for $\sigma\in\{+,-\}$, $\sum_{j:v_j=w^\sigma}\rho_j>-2$. We write $\ulin\rho$ and $\ulin v$ respectively for $(\rho_1,\dots,\rho_m)$ and $(v_1,\dots,v_m)$.
The (chordal) SLE$_\kappa(\ulin{ {\rho}})$ process in $\HH$ started from $w$ with force points $\ulin v$ is the chordal Loewner process driven by the function $W(t)$, $0\le t<T$, which   solves the SDE
$$dW(t)=\sqrt\kappa dB(t)+\sum_{j=1}^m\ind_{\{W(t)\ne \ha g^W_t(v_j)\}} \frac{\rho_j}{W(t)- \ha g^W_t(v_j)}\,dt, \quad  W(0)=w,$$
where $B(t)$ is a standard Brownian motion, and $\ha g^W_t$ are as in Definition \ref{Def-Rw}.

The SDE should be understood as an integral equation, i.e., $W-w-\sqrt\kappa B$ is absolutely continuous with the derivative being the summation in the SDE. The solution exists uniquely up to the first time $T$ (called a continuation threshold) that $ \sum_{j: g^W_t(v_j)=W(t)} \rho_j\le -2$.
If there does not exist a continuation threshold, then the lifetime is $\infty$.

We call $V_j(t):=\ha g^W_t(v_j)$ the force point process started from $v_j$. It is absolutely continuous with initial value $V_j(0)=v_j$ and derivative $\ind_{\{W(t)\ne {\color{blue} \ha {\color{black}{g}}}^W_t(v_j)\}}\frac 2{V_j(t)-W(t)}$. In fact, the set $\{t:W(t)={\color{blue} \ha{\color{black} g}}^W_t(v_j)\}$ has Lebesgue measure zero. So we will omit the factors $\ind_{\{W(t)\ne {\color{blue} \ha{\color{black} g}}^W_t(v_j)\}}$. Thus, $W$ and $V_1,\dots,V_m$ together solve the following system of SDE-ODE:
\BGE dW=\sqrt\kappa dB+\sum_{j=1}^m \frac{\rho_j dt}{W-V_j},\quad W_j(0)=w_j;\label{dW}\EDE
\BGE dV_j=\frac{2 dt}{V_j-W},\quad V_j(0)=v_j,\quad 1\le j\le m.\label{dVU}\EDE
Here we omit ``$(t)$'' for simplicity.
Moreover, $V_j\ge W$ if $v_j\ge w^+$; and $V_j\le W$ if $v_j\le w^-$. If $V_j(t_0)=V_k(t_0)$ at some point $t_0$, then $V_j(t)=V_k(t)$ for all $t\ge t_0$. If $v_j\ge v_k$, then $V_j\ge v_k$. If $v_j=v_k$ for some $j\ne k$, then $V_j=V_k$, and so $v_j$ and $v_k$ may be treated as a single force point with force value $\rho_j+\rho_k$.

An SLE$_\kappa(\ulin\rho)$ process a.s.\ generates a chordal Loewner curve $\eta$, which satisfies the following DMP  {(domain Markov property)}. Let $\F=(\F_t)_{t\ge 0}$ be the complete filtration generated by $\eta$. We write $\ulin V$ for $(V_1,\dots,V_m)$.  If $\tau$ an $\F$-stopping time, then on the event $\{\tau<T\}$, there is a curve $\eta^\tau$ in $\lin\HH$ such that $\eta(\tau+\cdot)=f_\tau(\eta^\tau)$, and the law of $\eta^\tau$ conditionally on $\F_\tau$ is an SLE$_\kappa(\ulin\rho)$ curve   started from $W(\tau)$ with force points $\ulin V(\tau)$, where when $V_j(\tau)=W(\tau)$, then $V_j(\tau)$ as a force point is treated as $W(\tau)^+$ (resp.\ $W(\tau)^-$) if $v_j\ge w^+$ (resp.\ $v_j\le w^-$).

\subsection{Extremal length} \label{Section 2.3}
Extremal length is a nonnegative quantity $\lambda(\Gamma)$ associated with a family of (piecewise $C^1$) curves  (cf.\ \cite[Definition 4-11]{Ahl}).
The extremal distance between $E$ and $F$ in $\Omega$, denoted by $d_\Omega(E,F)$ , is the extremal length of the family of curves in $\Omega$ connecting $E$ and $F$.

 Extremal length satisfies conformal invariance (cf.\ \cite[Section 4-1]{Ahl}) and comparison principle (\cite[Theorems 4.1 and 4.2]{Ahl}).  {The conformal invariance means that} if $\Gamma$ is a family of curves in a domain $D$, and $f$ is a conformal map on $D$, then $\lambda(f(\Gamma))=\lambda(\Gamma)$.  {The comparison principle means that} if every curve in $\Gamma$ contains a curve in $\Gamma'$, then $\lambda(\Gamma)\ge \lambda(\Gamma{ '})$. If every curve in $\Gamma$ contains a curve in $\Gamma_1$ and a curve in $\Gamma_2$, and $\Gamma_1$ and $\Gamma_2$ are respectively contained in $\Omega_1$ and $\Omega_2$, which are disjoint open sets, then $\lambda(\Gamma)\ge \lambda(\Gamma_1)+\lambda(\Gamma_2)$.

Suppose $R$ is a rectangle of dimension $a\times b$, and $I_1,I_2$ are sides of $R$ with length $a$. Then $d_R(I_1,I_2)=b/a$ (cf.\ \cite[Section 4-2]{Ahl}). By conformal invariance, we get the following result.

\begin{Example}
  For $x\in\R$ and $0<r_1<r_2$, let $A_{\HH}(x,r_1,r_2)$ denote the semi-annulus $\{z\in\HH:r_1<|z-x|<r_2\}$. Let $I_j=\{z\in\HH:|z-x|=r_j\}$, $j=1,2$. Then $d_{A_{\HH}(x,r_1,r_2)}(I_1,I_2)=\frac 1\pi\ln(\frac{r_2}{r_1})$.
\end{Example}

\begin{Proposition}
  Let $S_1$ and $S_2$ be a disjoint pair of connected nonempty closed subsets of $\lin\HH$ that intersect $\R$. Then we have
  $$\prod_{j=1}^2 \Big(\frac{\diam(S_j)}{\dist(S_1,S_2)}\wedge 1\Big)\le 144 e^{-\pi d_{\HH}(S_1,S_2)}.$$
  If $S_2$ is unbounded, then for any $z_0\in S_1$,
  $$\frac{\rad_{z_0}(S_1)}{\dist(z_0,S_2)}\wedge 1 \le 144 e^{-\pi d_{\HH}(S_1,S_2)}.$$
  \label{extrem-prop}
\end{Proposition}
\begin{proof}
  The first statement in the special case that $S_1$ and $S_2$ are bounded is exactly \cite[Lemma 2.9]{existence}. The general case follows from the special case  by replacing $S_j$ with $S_j\cap \{|z|\le n\}$ and letting $n\to \infty$. The second statement follows immediately from the first one.
\end{proof}

\subsection{Transition density of some diffusion processes} \label{Section 2.4}
In this subsection we review some results about the transition density of a family of Bessel-like diffusion processes, {which will be used in the proofs of Theorems \ref{Thm-est1} and \ref{m=1}. See Remarks \ref{Thm-est1-rem} and \ref{m=1-rem}.}
We first briefly review Jacobi polynomials (cf.\ \cite[Chapter 18]{NIST:DLMF}). For indices $\alpha_+,\alpha_->-1$,
Jacobi polynomials $P_n^{(\alpha_+,\alpha_-)}(x)$,  $n\in\N\cup\{0\}$, are a class of orthogonal polynomials w.r.t.\ the weight $w(s):=(1-s)^{\alpha_+}(1+s)^{\alpha_-}$. They satisfy
\BGE P_0^{(\alpha_+,\alpha_-)}\equiv 1;\label{P0}\EDE
\BGE \int_{-1}^1 w(s) P_n^{(\alpha_+,\alpha_-)}(s) P_m^{(\alpha_+,\alpha_-)}(s)  ds =0,\quad \mbox{if }n\ne m;\label{orthogonal}\EDE
\BGE \int_{-1}^1 w(s) P_n^{(\alpha_+,\alpha_-)}(s)^2 ds=\frac{2^{\alpha_++\alpha_-+1}\Gamma(n+\alpha_++1)\Gamma(n+\alpha_-+1)} {n!(2n+\alpha_++\alpha_-+1)\Gamma(n+\alpha_++\alpha_-+1)}=:{\cal A}_n^{(\alpha_+,\alpha_-)};\label{Jacobi-norm}\EDE
\BGE \sup_{x\in[-1,1]} |P_n^{(\alpha_+,\alpha_-)}(x)|\le \frac{\Gamma((\alpha_+\vee \alpha_-)+n+1)}{n!\Gamma((\alpha_+\vee \alpha_-)+1)} ,\quad \mbox{if }\alpha_+\vee \alpha_->-\frac 12.\label{Jacobi-bound}\EDE
$$\frac d{dx} P_n^{(\alpha_+,\alpha_-)}(x)=\frac {n+\alpha_++\alpha_-+1} 2 P_{n-1}^{(\alpha_++1,\alpha_-+1)}(x),\quad \mbox{if }n\ge 1;$$
$$P_n^{(\alpha_+,\alpha_-)}(1)=\frac{\Gamma(\alpha_++n+1)}{n!\Gamma(\alpha_++1)}.$$
When $n\ge 1$, using the last two displayed formulas, and applying (\ref{Jacobi-bound})  to $P_{n-1}^{(\alpha_++1,\alpha_-+1)}$, we find that for all $\alpha_+,\alpha_->-1$,
\BGE \sup_{x\in[-1,1]} |P_n^{(\alpha_+,\alpha_-)}(x)|\le \frac{\Gamma(\alpha_++n+1)}{n!\Gamma(\alpha_++1)} +   \frac{n(n+\alpha_++\alpha_-+1)\Gamma((\alpha_+\vee \alpha_-)+n+1)}{n!\Gamma((\alpha_+\vee \alpha_-)+2)}.\label{Jacobi-bound-general} \EDE
This inequality trivially holds for $n=0$.

\begin{Proposition}
  Let $\delta_+,\delta_-\in\R$. Let $B$  be a standard Brownian motion. Suppose $Z(t)$, $0\le t<T$, is a diffusion process that lies on an interval $I$ and satisfies the SDE
  \BGE dZ=\sqrt{1-Z^2}dB-\frac{\delta_+}4(Z+1)dt-\frac{\delta_-}4(Z-1)dt.\label{Bessel-SDE}\EDE
We consider the following four cases:
\begin{itemize}
  \item [(i)] $\delta_+>0,\delta_->0$, $I=[-1,1]$, and $T=\infty$;
  \item [(ii)] $\delta_+<2,\delta_-<2$,  $I=(-1,1)$, and $\lim_{t\uparrow T} Z(t)\in \{-1,1\}$;
  \item [(iii)] $\delta_+>0$, $\delta_-<2$,   $I= (-1,1]$, and $\lim_{t\uparrow T} Z(t)=-1$;
  \item [(iv)] $\delta_->0$, $\delta_+<2$, $I= [-1,1)$, and $\lim_{t\uparrow T} Z(t)=1$.
\end{itemize}
Then in each of the four cases, $Z$ has a transition density $p_t(x,y)$ in the sense that
  if $Z$ starts from $x\in I$,  then for any $t>0$ and any measurable set $A\subset I$,
$$\PP_x[ \{T>t\}\cap \{Z(t)\in A\}]=\int_A p_t(x,y)dy.$$
Moreover, $p_t(x,y)$ can be expressed   by
   \BGE p_t(x,y)= f(x)g(y)\sum_{n=0}^\infty \frac{P_n^{(\alpha_+,\alpha_-)}(x) P_n ^{(\alpha_+,\alpha_-)}(y)}{{\cal A}^{(\alpha_+,\alpha_-)}_n}  e^{-\beta_n t}.\label{Jacobi-series}\EDE
   where $P_n^{(\alpha_+,\alpha_-)}$ are Jacobi polynomials with indices $(\alpha_+,\alpha_-)$, ${\cal A}^{(\alpha_+,\alpha_-)}_n $ are normalization constants given by (\ref{Jacobi-norm}), and  $\alpha_+,\alpha_-,f,g,\beta_n$, $n\ge 0$, are given by the following formulas depending on cases.
   \begin{itemize}
    \item  In Case (i), $\alpha_+ =\frac{\delta_+}2-1$, $\alpha_- =\frac{\delta_-}2-1$,  $f(x)= 1$, $g(y)=(1-y)^{\alpha_+}(1+y)^{\alpha_-}$,  and $\beta_n=\frac 1 2 n(n+1+\alpha_++\alpha_-)$.
    \item In Case (ii),  $\alpha_+=1-\frac{\delta_+}2$, $\alpha_-=1-\frac{\delta_-}2$, $f(x)=(1-x)^{\alpha_+}(1+x)^{\alpha_-}$, $g(y)= 1$, and $\beta_n=\frac 1 2 (n+1)(n+\alpha_++\alpha_-)$.
     \item In Case (iii),
         $\alpha_+=\frac{\delta_+}2-1$, $\alpha_-=1-\frac{\delta_-}2$, $f(x)=(1+x)^{\alpha_-} $, $g(y)=(1-y)^{\alpha_+}$, and $\beta_n=\frac 1 2 (n+1+\alpha_+)(n+\alpha_-)$.
     \item In Case (iv), $\alpha_-=\frac{\delta_-}2-1$, $\alpha_+=1-\frac{\delta_+}2$, $f(x)=(1-x)^{\alpha_+}$, $g(y)=(1+y)^{\alpha_-}$, and $\beta_n=\frac 1 2 (n+1+\alpha_-)(n+\alpha_+)$.
  \end{itemize}
\label{Bessel-transition}
\end{Proposition}

\begin{Remark}
  Cases (i) and (ii) of the proposition in the special case $\delta_+=\delta_-$ were proven in \cite[Appendix B]{tip}. The general case was mentioned in the remark after \cite[Corollary 8.5]{tip}. We state the proposition here for future reference, but omit the proof since it is similar to the proofs of the special cases.  {In all cases, there is a function $h(x)$ given by $h(x)=(1-x)^{\frac{\delta_+}2-1}(1+x)^{\frac{\delta_-}2-1}$ such that $h(x)p_t(x,y)=h(y) p_t(y,x)$.}
\end{Remark}

In each case of Proposition \ref{Bessel-transition}, we define a probability density $p_\infty$ by
\BGE {\cal Z}=\int_{-1}^1 g(x)dx,\quad p_\infty(x)= \frac{\ind_{[-1,1]}(x) g(x)} {\cal Z}.\label{invariant-def}\EDE
Note that ${\cal Z}\in(0,\infty)$. Using (\ref{P0},\ref{orthogonal},\ref{Jacobi-bound-general})  {and that $f(x) g(x)=w(x)$}, we get
\BGE \int_{-1}^1 p_\infty(x) p_t(x,y) d {x}=p_\infty(y) e^{-\beta_0 t},\quad y\in[-1,1],\quad t>0.\label{invariant=}\EDE
In Case (i), since $\beta_0=0$, $ p_\infty$ is the invariant density of $Z$. In Cases (ii-iv), since $\beta_0>0$, $ p_\infty$ is the quasi-invariant density of $Z$ with decay rate $\beta_0$. Because of (\ref{P0}), the first term ($n=0$) in the series (\ref{Jacobi-series}) is $ {\cal Z} f(x) p_\infty (y)e^{-\beta_0 t}/{\cal A}_0^{(\alpha_+,\alpha_-)}$. Since $\beta_n\ge \beta_1>\beta_0$ for $n\ge 1$, using (\ref{P0},\ref{Jacobi-norm},\ref{Jacobi-bound-general}) and Stirling's formula, we find that there are constants $C,L\in(0,\infty)$ depending only on $\alpha_+,\alpha_-$ such that
\BGE |p_t(x,y)- {\cal Z} f(x)  p_\infty(y)e^{-\beta_0 t}/{\cal A}_0^{(\alpha_+,\alpha_-)}|\le  C f(x)g(y) e^{-\beta_1 t}, \quad\mbox{if } t>L.\label{transition-rate}\EDE
This formula describes the convergence rate of the  transition density to the invariant/quasi-invariant density.
A simpler argument shows that  there are constants $C,L\in(0,\infty)$ depending only on $\alpha_+,\alpha_-$, such that
\BGE |p_t(x,y)|\le C f(x) g(y) e^{-\beta_0 t},\quad \mbox{if }t\ge L.\label{upper-bound-pt}\EDE
In Cases (ii)-(iv), we have $\beta_0>0$. So
\BGE \PP_x[T>t]=\int_{-1}^1 p_t(x,y)dy\le C {\cal Z} f(x) e^{-\beta_0 t}\le C {\cal Z} e^{-\beta_0 t},\quad \mbox{if }t\ge L.\label{T>t}\EDE
This shows that the lifetime of $Z$ in Cases (ii)-(iv) has an exponential tail.

If $X$ is a diffusion process that lies on $[0,1]$ and satisfies the SDE
  \BGE dX=\sqrt{X(1-X)}dB-\frac{\delta_+}4 Xdt-\frac{\delta_-}4(X-1)dt\label{Bessel-SDE-X}\EDE
for some standard Brownian motion $B$, then $Z:=2X-1$ satisfies SDE (\ref{Bessel-SDE}), and we may use Proposition \ref{Bessel-transition} to get a transition density of $X$.

\begin{Lemma}
  Let $B$ be a standard Brownian motion. Let $\delta_-,\delta_+\ge 2$. Suppose a continuous semimartingale $X(t)$, $0\le t<\infty$, lie {s} on $(0,1)$, and satisfies (\ref{Bessel-SDE-X}), and another continuous semimartingale $\til X(t)$, $0\le t<T$, lie {s} on $(0,1]$, and satisfies the SDE
 \BGE d\til X=\sqrt{\til X(1-\til X)} dB-Pdt-\frac{\delta_-}4 (\til X-1) dt.\label{Bessel-SDE-til-X}\EDE
 Assume that    $P\le \frac{\delta_+}4\til X$,  $X(0)\le \til X(0)$, and the $B$ in (\ref{Bessel-SDE-X}) and (\ref{Bessel-SDE-til-X}) are the same Brownian motion. Then a.s.\ $X(t)\le \til X(t)$ for all $t\in [0,T)$.
  \label{dominated}
\end{Lemma}
\begin{proof}
  Let $f(x)=\arccos(1-2x)$, which is strictly increasing and maps $[0,1]$ onto $[0,\pi]$. Let $\Theta=f(X)\in (0,\pi)$ and $\til \Theta=f(\til X)\in (0,\pi]$.
   It suffices to show that a.s.\ $\Theta(t)\le \til \Theta(t)$ for all $t\in[ 0,T)$. Since $f$ is $C^2$ on $(0,1)$,   by It\^o's formula (cf.\ \cite{RY}) $\Theta$ satisfies the following SDE on $[0,\infty)$:
  $$d\Theta=dB+\frac{\delta_--1}4\cot_2(\Theta)dt-\frac{\delta_+-1}4 \tan_2(\Theta)dt,$$
  where $\cot_2:=\cot(\cdot/2)$ and $\tan_2:=\tan(\cdot/2)$. Thus,  $\Theta-B$ is $C^1$ with derivative  $\frac{\delta_--1}4\cot_2(\Theta)-\frac{\delta_+-1}4 \tan_2(\Theta)$.

  The above argument does not apply to $\til\Theta$ since $f$ is not $C^2$ at   $1$. Let $$\Psi=\ind_{\til X\in (0,1)}\cdot \Big(\frac{P}{\sqrt{\til X(1-\til X)}}-\frac 14 \tan_2(\til\Theta)\Big).$$
Since  $P\le \frac{\delta_+}4\til X$, we have $\Psi\le \frac{\delta_+-1 }4\tan_2(\til \Theta)$. For a (random) interval $I$, we say that $\til \Theta$ is {\it nice} on $I$ if $\til\Theta-B$ is $C^1$ on $I$ with derivative $\frac{\delta_--1}4\cot_2(\til\Theta)-\Psi$.
   Suppose $\tau$ is any finite stopping time. Let $T_1(\tau)=\inf(\{t\ge \tau: \til X(t)=1\}\cup \{\infty\})$. Then $\til X\in(0,1)$ on $[\tau,T_1(\tau))$. Since $f$ is $C^2$ on $(0,1)$, we may apply It\^o's formula to conclude that $\til\Theta$ satisfies the  SDE
  $$d\til\Theta=dB+\frac{\delta_--1}4\cot_2(\til\Theta)dt -\Psi dt$$
on $[\tau,T_1(\tau))$. Thus, a.s.\ $\til\Theta$ is nice on $[\tau,T_1(\tau))$.

  Let $\sigma_1$ and $\sigma_2$ be two random times taking values in $[0,\infty]$, which may not be stopping times. Let $E$ be any event on which $\sigma_1<\sigma_2$ and $\til X$ stays in $(0,1)$ on $(\sigma_1,\sigma_2)$. We claim that, a.s.\ on the event $E$, $\til\Theta$ is nice on $(\sigma_1,\sigma_2)$. First,  the claim holds if on the event $E$, $\sigma_1 $ and $\sigma_2$ are deterministic times $t_1<t_2$, and $\til X(t_1)\in (0,1)$. To see this, we apply the result of the last paragraph to $\sigma_1=t_1$ and use and the fact that $E\subset\{T_1(t_1)>t_2\}=\{[t_1,t_2]\subset [t_1,T_1(t_1))\}$.
Second,  the claim holds if on the event $E$, $\sigma_1$ and $\sigma_2$ take values in $\Q_+$, and $\til X(\sigma_1)\in (0,1)$. This is true because  for any $p<q\in \Q_+$, a.s.\ on the event $ E\cap \{\sigma_1=p\}\cap \{\sigma_2=q\}$, $\til\Theta$ is nice on $(\sigma_1,\sigma_2)$. Finally, for the general case, we construct two sequences of $\Q_+$-valued random times $(\sigma_1^{(n)})$ and  $(\sigma_2^{(n)})$ such that  $\sigma_1^{(n)}\downarrow \sigma_1$ and $\sigma_2^{(n)}\uparrow \sigma_2$. Let $E_n=E\cap \{\sigma_1^{(n)}<\sigma_2^{(n)}\}$, $n\in\N$. Then $E_n\uparrow E$. We have  {shown} that, for each $n\in\N$, a.s.\ on $E_n$, $\til\Theta$ is nice on  $(\sigma_1^{(n)},\sigma_2^{(n)})$.
 Since  $ (\sigma_1^{(n)},\sigma_2^{(n)})\uparrow (\sigma_1,\sigma_2)$ and $E_n\uparrow E$, the claim is proved.

Since $X(0)\le \til X(0)$, $\Theta(0)\le \til\Theta(0)$.   Now we prove that a.s.\ $\Theta \le \til \Theta $ on $[0,T)$. Let $$E=\{\exists t\in [0,T):\Theta(t)>\til\Theta(t)\},\quad E_1=\{\exists t\in [0,T):\til\Theta(t)=0\},\quad E_2=E\sem E_1.$$  Define random times $\sigma_1$ and $\sigma_2$  as follows. On $E_1$, let $\sigma_2=\inf\{t\ge 0: \til\Theta(t)=0\}$; on  $E_2$, let $$\sigma_2=\inf\{t\ge 0: \til \Theta(t)-\Theta(t)= \inf \{\til\Theta(s)-\Theta(s):0\le s<T\}/2\};$$  on $E=E_1\cup E_2$, let $\sigma_1=\max\{t\in [0,\sigma_2]: \til\Theta(t)\ge \Theta(t)\}$; and on   $E^c$, let $\sigma_1=0$ and $\sigma_2=\infty$.  Then on the event $E$, $\sigma_1<\sigma_2<\infty$, $\til\Theta(\sigma_1)=\Theta(\sigma_1)$, and $\til\Theta\in (0,\Theta)$ on $(\sigma_1,\sigma_2]$. By the claim, a.s.\ on the event $E$,
  $$(\til\Theta(\sigma_2)-B(\sigma_2))-(\til\Theta(\sigma_1)-B(\sigma_1))=\int_{\sigma_1}^{\sigma_2} \Big(\frac{\delta_--1}4\cot_2(\til\Theta(t))-\Psi(t)\Big)dt.$$
The fact that  $\Theta-B$ is $C^1$ on $[0,\infty)$ with derivative being $\frac{\delta_--1}4\cot_2(\Theta)-\frac{\delta_+-1}4 \tan_2(\Theta)$ implies
  $$(\Theta(\sigma_2)-B(\sigma_2))-(\Theta(\sigma_1)-B(\sigma_1))=\int_{\sigma_1}^{\sigma_2} \Big(\frac{\delta_--1}4\cot_2(\Theta(t))-\frac{\delta_+-1}4 \tan_2(\Theta(t))\Big)dt.$$
  Combining the above two displayed formulas and using that $\til\Theta(\sigma_1)=\Theta(\sigma_1)$, we get a.s.\ on $E$,
$$ \til\Theta(\sigma_2)-\Theta(\sigma_2)=\int_{\sigma_1}^{\sigma_2} \Big(\frac{\delta_--1}4(\cot_2(\til\Theta(t))-\cot_2(\Theta(t)))-\Psi(t)+\frac{\delta_+-1}4 \tan_2(\Theta(t))\Big)dt>0.$$
Here we use the facts that $0<\til\Theta<\Theta<\pi$ on  $(\sigma_1,\sigma_2)$, $\delta_->1$, and $\Psi\le \frac{\delta_+-1}4\tan_2(\til \Theta)$. However, since $\til\Theta(\sigma_2)<\Theta(\sigma_2)$,  $E$ must be a null event. So  a.s.\ $\Theta \le \til \Theta $ on $[0,T)$.
  \end{proof}

\section{One-point Estimate} \label{Section 3}
Let $\kappa>0$, $m\in\N$, $\ulin\rho=( \rho_1,\dots,\rho_m) \in\R^n $,  $w\in\R$, and $\ulin v=(v_1,\dots,v_m)\in\R_w^m$ (Definition \ref{Def-Rw}). Assume that $v_m\le v_{m-1}\le \cdots \le v_1\le w^-$.  Let $\eta$ be an SLE$_\kappa(\ulin\rho)$ curve in $\HH$ started from $w$ with force points $\ulin v$.
Let $\F$ be the complete filtration generated by $\eta$.
Let $W$ be the driving process for $\eta$, and let $V_j$ be the force point process   started from $v_j$, $1\le j\le m$. We write $\ulin V$ for $(V_1,\dots,V_m)$.
Then $W\ge V_1\ge\cdots\ge V_m$, and they satisfy the $\F$-adapted SDE/ODE (\ref{dW},\ref{dVU}).
Let $g_t$ and $K_t$ be the chordal Loewner maps and hulls generated by $\eta$.
For $u\in\R\sem \{w\}$, let
\BGE U(t)=g_t(u),\quad D_u(t)=g_t'(u),\quad 0\le t<\tau_u^*.\label{U-Du}\EDE Then $U$ and $D_u$ satisfy the following ODE on $[0,\tau_u^*)$:
\BGE dU=\frac{2 dt}{U-W},\quad \frac{d D_u}{D_u}=\frac{-2dt}{(U-W)^2}.\label{dUD}\EDE
The purpose of this section is to prove the following theorem.

\begin{Theorem}
Suppose that $\ulin\rho$ satisfies that $\rho_\Sigma:=\sum_{k=1}^m \rho_k >\frac\kappa 2-4$ and $\sigma_j:=2+\sum_{k=1}^j \rho_k\in(0, \sigma^*]$, $0\le j\le m-1$, for some $\sigma^*\ge 2$.  Then there is a constant $C\in(0,\infty)$ depending only on $\kappa,\rho_\Sigma,\sigma^*$ such that the following holds. Suppose $v_m<w^-$. Let  $u=v_m$ and  $v_0=w^-$.  Let $r\in (0,|w-u|)$ and $j_0=\max \{0\le j\le m: v_j\ge u+r\} $.  Let $\gamma=\frac{2\rho_\Sigma+8-\kappa}{2\kappa}$.Then
    \BGE \PP[\dist(\eta,u)\le r]\le C  \Big(\frac{r}{|v_{j_0}-u|}\Big)^{ {\gamma\sigma_{j_0}}} \prod_{j=0}^{j_0-1}\Big(\frac{|v_{j+1}-u|}{|v_j-u|}\Big) ^{{\gamma\sigma_{j}}}.\label{est1}\EDE
\label{Thm-est1}
\end{Theorem}
 {\begin{Remark}
  The theorem gives an upper bound of the probability that an SLE$_\kappa(\ulin\rho)$ curve $\eta$ gets near a boundary point $u$. This is usually the first step to prove the existence of $1$-point Green's function. The $u$ is assumed to be $v_m$, i.e., the force point farthest from the initial point. But we can also use the theorem to treat the case that   $u\in (-\infty,v_m)$ by including $u$ as a force point of $\eta$ with force value $0$ without changing the law of $\eta$. To prove the theorem, we replace the Euclidean distance by some conformal radius using Koebe's $1/4$ theorem, reparametrize the curve such that the conformal radius decreases exponentially, obtain a diffusion process using the force point processes, and finally compare the process with some radial Bessel process to bound the probability that its lifetime is greater than a given number. \label{Thm-est1-rem}
\end{Remark}}
\begin{proof}[ {Proof of Theorem \ref{Thm-est1}}]
 Throughout the proof, a constant is a number in $(0,\infty)$ depending only on $\kappa,\rho_\Sigma,\sigma^*$, whose value may vary. We write $X\lesssim Y$ if there is a constant $C$ such that $X\le CY$, and write $X\asymp Y$ if $X\lesssim Y$ and $Y\lesssim X$.

By adding force points with force value $0$, we may assume that there is $j_0\in \{1,\dots,m-1\}$ such that $v_{j_0}-u=r$, and  that $v_0=w^-$ is a force point. Let $V_0(t)$ be the force point process started from $v_0$, which also satisfies (\ref{dVU}). By merging force points, we assume that $v_{m-1}>v_m$. Let $U$ and $D_u$ be defined by (\ref{U-Du}).
Then $W\ge V_0\ge \cdots \ge V_{m-1}>V_m=U$ on $[0,\tau_u^*)$.  Since the part of $\eta$ after $\tau_u^*$ is disconnected from $u$ by $\eta[0,\tau_u^*)$, $\dist(\eta,u)=\dist(\eta[0,\tau_u^*),u)$.
So (\ref{est1}) becomes
  \BGE \PP[\dist(\eta[0,\tau_u^*),u)\le r]\lesssim \prod_{j=0}^{j_0-1}\Big(\frac{|v_{j+1}-u|}{|v_j-u|}\Big) ^{\gamma \sigma_j}.\label{est1''}\EDE

For $0\le t<\tau_u^*$, $u$ lies in  $\C\sem (K_t^{\doub}\cup [w,\infty))$; $g_t$ maps $\C\sem (K_t^{\doub}\cup [w,\infty))$  conformally onto $\C\sem [V_0(t),\infty)$, and sends $u$ to $U(t)$.
Define $ H_j$ and $Q_j$  on $[0,\tau_u^*)$ by
\BGE   H_j=\frac{V_j-U}{V_0-U},\quad Q_j=\frac{W-V_0}{W-V_j}.\label{DHQ}\EDE
Here if $\tau^*_{v_j}$ happens before $\tau^*_u$, then we define $Q_j$ to be constant $1$ on $[\tau^*_{v_j},\tau^*_u)$. Then $0=H_m<H_{m-1}\le \cdots \le H_0=1$ and $0\le Q_m\le\cdots\le Q_0=1$. By (\ref{dVU}), $H_j$ satisfies the ODE
\BGE \frac{d H_j}{H_j}=\frac{2(V_0-V_j)dt}{(W-U)(W-V_0)(W-V_j)}.\quad 0\le j\le m-1.\label{dH}\EDE
Define  $R=1-Q_m=\frac{V_0 -U }{W -U }\in(0,1]$ on  $[0,\tau_u^*)$.
Then $R$ starts from $1$ (because $V_0(0)=W(0)=w$), and
by (\ref{dW},\ref{dVU}) satisfies the SDE:
\BGE dR=R\Big[-\frac{\sqrt\kappa dB}{W-U}+\frac{(\kappa-2-\rho_m)dt}{(W-U)^2} -\frac{2dt}{(W-U)(W-V_{0})} -\sum_{j=0}^{m-1} \frac{\rho_j dt }{(W-U)(W-V_j)} \Big] .\label{dR}\EDE
 Define functions $\Up_j =\frac{|V_j -U |}{D_u }$ on $[0,\tau_u^*)$.
 By Koebe $1/4$ theorem,
\BGE \Up_j(t)\asymp |v_j-u|\wedge \dist(\eta[0,t],u).\label{Upj}\EDE
 Define $\Phi =-\frac \kappa2 \ln\Big(\frac{\Upsilon_0 }{|w-u|}\Big)$ on $[0,\tau_u^*)$.
Then $\Phi$ star{  t}s from $0$, and by (\ref{dVU},\ref{dUD})  satisfies the  ODE
\BGE d\Phi(t)=\frac{\kappa Rdt}{(W-U)(W-V_{0})}>0\quad\mbox{on }[0,\tau_u^*).\label{dPhi}\EDE
So $\Phi$ is  strictly increasing on $[0,\tau_u^*)$.
Since $v_0=w=\eta(0)$, by (\ref{Upj}) we have \BGE \dist(u,\eta[0,t])  \asymp \Up_0(t)=|w-u| e^{-\frac 2\kappa \Phi(t)}, \quad 0\le t<\tau_u^*\label{dist(u,eta)}\EDE

Let $\ha \tau_u^*=\sup \Phi[0,\tau_u^*)$.
 Define $\ha W,\ha U,\ha R,\ha D_u, \ha V_j,\ha Q_j,\ha H_j$  on $[0,\ha \tau_u^*)$, to be respectively the time-changes of $W,U, R, D_u, V_j,Q_j,H_j$ via $\Phi$. For example, $\ha R(t)=R(\Phi^{-1}(t))$, $0\le t<\ha \tau_u^*$. By (\ref{dH},\ref{dPhi}), $\ha H_j$ satisfies the following ODE on $[0,\ha \tau_u^*)$:
\BGE \frac{d\ha H_j}{\ha H_j}=\frac 2\kappa \cdot \frac 1{\ha R}\cdot \frac{\ha V_0-\ha V_j}{\ha W-\ha V_j}dt=\frac 2\kappa\cdot \frac{1-\ha Q_j}{\ha R}dt,\quad 0\le j\le m-1. \label{dhaH}\EDE
By (\ref{Upj},\ref{dist(u,eta)}), we have
\BGE \ha H_j(t)  =\frac{\Up_j(\Phi^{-1}(t))}{\Up_0(\Phi^{-1}(t))}\asymp \frac{|v_j-u|\wedge (|w-u|e^{-\frac 2\kappa t})}{|w-u| e^{-\frac 2\kappa t}}.
\label{Koebe-H-hat}\EDE
From (\ref{dR},\ref{dPhi}), for some standard Brownian motion $\ha B$, $\ha R$ satisfies the following SDE:
 \BGE d\ha R=\sqrt{\ha R(1-\ha R)}d\ha B+\frac{\kappa-2-\rho_m}{\kappa}(1-\ha R)dt -\frac 2\kappa dt-\sum_{j=0}^{m-1}  \frac{\rho_j}\kappa \ha Q_j dt. \label{dhaR}\EDE
By It\^o's formula, for any $c\in\R$,
$$\frac{d\ha R^c}{\ha R^c}=c \sqrt{\frac{1-\ha R}{\ha R}}d\ha B+\Big(c\frac{\kappa -2-\rho_m}\kappa+\frac {c(c-1)}2 \Big)\frac{1-\ha R}{\ha R} dt-\frac{2c}{\kappa \ha R} dt  - \frac c{\ha R}\sum_{j=0}^{m-1}  \frac{\rho_j}\kappa \ha Q_j dt$$
$$=c \sqrt{\frac{1-\ha R}{\ha R}}d\ha B+\Big(c\frac{\kappa -4-\rho_\Sigma}\kappa+\frac {c(c-1)}2 \Big)\frac{dt}{\ha R} -\Big(c\frac{\kappa -2-\rho_m}\kappa+\frac {c(c-1)}2 \Big)dt + c\sum_{j=0}^{m-1}  \frac{\rho_j}\kappa \frac{1- \ha Q_j}{\ha R}dt.$$
Let $\ha E(t)= e^{-\frac 2 \kappa t}$.
By taking $c=2\gamma$ in the above SDE and using (\ref{dhaH}), we find that
$$\ha M:=\ha E^{-\gamma\sigma_{m-1}} \ha R^{2\gamma}\prod_{j=0}^{m-1} \ha H_j^{-\gamma\rho_j}$$ is a positive local martingale on $[0,\ha\tau^*_u)$  and satisfies the SDE
\BGE \frac{d \ha M}{\ha M}
=\frac{2\rho_\Sigma+8-\kappa}{\kappa} \frac{1-\ha R}{\sqrt{\ha R(1-\ha R)}} d\ha B.\label{dhaM}\EDE
Let
$\ha H_{j,j+1}=\frac{\ha H_{j+1}}{\ha H_j}=\frac{\ha V_{j+1}-\ha U}{\ha V_j-\ha U}\in(0,1]$, $0\le j\le m-2$. By (\ref{dhaH}), each $\ha H_{j,j+1}$ satisfies $\frac{d\ha H_{j,j+1}}{\ha H_{j,j+1}}=\frac 2\kappa \frac{\ha Q_j-\ha Q_{j+1}}{\ha R}\ge 0$, and so is increasing.
Since $\sigma_j-\sigma_{j-1}=\rho_j$ and $\ha H_0=1$, we have
\BGE \ha M=\ha E^{-\gamma\sigma_{m-1}} \ha R^{2\gamma}  \prod_{j=0}^{m-2}\ha H_{j,j+1}^{\gamma\sigma_j-\gamma\sigma_{m-1}}. \label{haM-re}\EDE

We will apply Girsanov theorem. Let $\PP$ denote the current measure. Fix  $n\in\N$.  Let $\tau_n$ be the maximum element in $[0,\ha \tau_u^*]$ such that $|\ln(\frac{\ha M(t)}{\ha M(0)})|\le n$ on $[0,\tau_n)$. Then $\ha M(\cdot\wedge \tau_n)$ is a uniformly bounded martingale on $[0,\infty)$. Here when $\tau_n=\ha \tau_u^*$ and $t\ge \tau_n$, $\ha M(t\wedge \tau_n)$ is understood as $\lim_{t\uparrow \ha \tau_u^*} \ha M(t)$, which a.s. converges on the event $\{\tau_n=\ha \tau_u^*\}$. So we get $\EE[\frac{\ha M(\tau_n)}{\ha M(0)}]=1$. Now we define a new probability measure $\til\PP_n$ by $\frac{d\til\PP_n}{d\PP}=\frac{\ha M(\tau_n)}{\ha M(0)}$.
By Girsanov theorem and (\ref{dhaR},\ref{dhaM}), under the new measure $\til\PP_n$, $\ha {  R}$ satisfies the following SDE on $[0,\tau_n)$:
$$ d\ha R=\sqrt{\ha R(1-\ha R)}d\til B -\frac 2\kappa dt-\sum_{j=0}^{m-1}  \frac{\rho_j}\kappa \ha Q_j dt-\frac{2\rho_\Sigma +6-\rho_m}{\kappa}(\ha R-1)dt, $$
where $\til B$ is a  standard Brownian motion under $\til\PP_n$. Since $\sigma_j-\sigma_{j-1}=\rho_j$, $\sigma_0=2$, $\ha Q_0=1$, and $\ha Q_{m-1}=1-\ha R$, we may rewrite the above SDE as
 \BGE d\ha R=\sqrt{\ha R(1-\ha R)}d\til B -\sum_{j=0}^{m-2}  \frac{\sigma_j}\kappa (\ha Q_j-\ha Q_{j+1}) dt-\frac{\rho_\Sigma +4}{\kappa}(\ha R-1)dt. \label{dtilR}\EDE
Since $\ha Q_j\ge \ha Q_{j+1}$ and $\sigma_j\le \sigma^*$,  we get
$\sum_{j=0}^{m-2}  \frac{\sigma_j}\kappa (\ha Q_j-\ha Q_{j+1}) \le \frac{\sigma^*}\kappa (1-\ha Q_{m-1})\le  \Big(\frac{\sigma^*}\kappa\vee \frac 12\Big) \ha R$.
Thus, under $\til\PP$, $\ha R$ up to $\tau_n$ satisfies the property of $\til X$ in  Lemma \ref{dominated} with $\delta_-:=\frac 4\kappa(\rho_\Sigma +4)>2$ and $\delta_+:=\frac 4\kappa\sigma^* \vee 2\ge 2$. By Lemma \ref{dominated}, under $\til\PP_N$, up to $\tau_n$, $\ha R$ stochastically dominates the process $X$, which starts from $1$, and satisfies the SDE (\ref{Bessel-SDE-X}).
 By Proposition \ref{Bessel-transition} (i), $X$ has a transition density $p_t(x,y)$, which by (\ref{upper-bound-pt}) satisfies
 \BGE p_t(x,y)\lesssim y^{\frac{\delta_-}2-1}=y^{2\gamma}.\label{ptupper}\EDE

 Let $L>0$ be such that $|w-u|e^{-\frac 2\kappa L}=\frac r4$. By (\ref{dist(u,eta)}), we get (\ref{est1''}) if
 \BGE \PP[\ha \tau_u^*>L]\lesssim \prod_{j=0}^{j_0-1} \Big(\frac{|v_{j+1}-u|}{|v_j-u|}\Big)^{\gamma \sigma_j}. \label{est1'''}\EDE
 Since $\frac{d\til\PP_n}{d\PP}=\frac{\ha M(\tau_n)}{\ha M(0)}$ and $\ha M(\cdot\wedge \tau_n)$ is a uniformly bounded martingale under $\PP$, we get
$$ \PP[L<\tau_n\wedge \ha \tau_u^*]=\til\EE_n [\ind_{\{L<\tau_n\wedge \ha \tau_u^*\}}  {\ha M(0)}/{\ha M(L\wedge \tau_n\wedge L)}  ]=\til\EE_n [\ind_{\{L<\tau_n\wedge \ha \tau_u^*\}} {\ha M(0)}/{\ha M(L )} ].$$ Define  $\ha N(t)=\prod_{j=0}^{m-2}\Big( \frac{\ha H_{j,j+1}(t)}{\ha H_{j,j+1}(0)} \Big)^{ \gamma\sigma_{m-1}-\gamma\sigma_j} $.
Since $\ha E(L)=e^{-\frac 2\kappa L}\asymp \frac r{|w-u|}$, by (\ref{haM-re}),
$$  \frac{\ha M(0)}{\ha M(L )}=\ha E(L)^{\gamma\sigma_{m-1}} \ha R(L)^{-2 \gamma} \ha N(L)\asymp \Big(\frac r {|w-u|}\Big) ^{\gamma\sigma_{m-1}} \ha R(L)^{-2 \gamma} \ha N(L).$$
Since $0<\ha H_{j,j+1}(0)\le \ha H_{j,j+1}(L)\le 1$ and $\sigma^*\ge \sigma_j\ge 0$, by  (\ref{Koebe-H-hat})  and the fact that $|v_{j_0}-u|=r\asymp |w-u|e^{-\frac 2\kappa L}$, we have
$$\ha N(L)\le \prod_{j=0}^{j_0-1}\Big( \frac{1}{\ha H_{j,j+1}(0)}\Big)^{ \gamma\sigma_{m-1}} \cdot \prod_{j=j_0}^{m-2}\Big( \frac{\ha H_{j,j+1}(L)}{\ha H_{j,j+1}(0)}\Big)^{ \gamma\sigma_{m-1}} \cdot \prod_{j=0}^{j_0-1}\Big( \frac{\ha H_{j,j+1}(L)}{\ha H_{j,j+1}(0)}\Big)^{ -\gamma\sigma_j}$$
$$\le \ha H_{m-1}(0)^{-\gamma\sigma_{m-1}}\Big(\frac{\ha H_{m-1}(L)}{\ha H_{j_0}(L)}\Big)^{\gamma\sigma_{m-1}} \ha H_{j_0}(L)^{-\gamma\sigma_*}\prod_{j=0}^{j_0-1} \ha H_{j,j+1}(0)^{\gamma\sigma_j}$$
$$\asymp \Big(\frac{r}{|w-u|}\Big)^{-\gamma\sigma_{m-1}} \prod_{j=0}^{j_0-1} \Big(\frac{|v_{j+1}-u|}{|v_j-u|}\Big)^{\gamma\sigma_j}.$$
Here we used that $\ha H_{m-1}(L)/\ha H_{m-1}(0)\asymp |w-u|/r$ and $\ha H_{j_0}(L)\asymp 1$.
Since $\ha R$ under $\til \PP$ up to $\tau_n$ stochastically dominates the process $X$  with transition density $p_t(x,y)$, and $-2\gamma<0$, by (\ref{ptupper}) we get
$$\til \EE_n\Big[\ind_{\{L<\tau_n\wedge \ha \tau_u^*\}}\ha R(L)^{-2\gamma}\Big] \le  \int_0^1 y^{-2\gamma} p_L(1,y) dy\lesssim 1.$$

By the last four displayed formulas  we get
$ \PP[L<\tau_n\wedge \ha \tau_u^*]\lesssim  \prod_{j=0}^{j_0-1} \Big(\frac{|v_{j+1}-u|}{|v_j-u|}\Big)^{\gamma \sigma_j}$, which implies the desired estimate (\ref{est1'''}) as $n$ tends to $\infty$.
\end{proof}

\section{One-point Green's Function} \label{Section 4}
Let $\kappa,\rho_\Sigma,\sigma^*,\sigma_*\in\R$ satisfy that $\kappa>0$, $\rho_\Sigma>(-2)\vee(\frac\kappa 2-4)$ and $\sigma^*\ge 2\ge \sigma_*>0$.
In this section, a constant depends only on $\kappa,\rho_\Sigma,\sigma^*,\sigma_*$.
We will use the following positive constants:
\BGE \gamma=\frac{2\rho_\Sigma+8-\kappa}{2\kappa},\quad  \alpha=(\rho_\Sigma+2)\gamma,\quad \alpha_*=\sigma_*\gamma,\quad \alpha^*=\sigma^*\gamma.\label{alpha}\EDE
We write $X\lesssim Y$ or $Y\gtrsim X$ if there is a constant $C>0$ such that $X\le C Y$; and write $X\asymp Y$ if $X\lesssim Y$ and $Y\lesssim X$. We write $X=O(Y)$ as $R\to 0^+$ if there are   constants  $C,c>0$ such that when $0<R<c$, we have $X\le C Y$.

Let $m\in\N$ and $\ulin\rho=(\rho_1,\dots,\rho_m)\in\R^m$ satisfy that $\sum_{j=1}^m \rho_j=\rho_\Sigma$ and   $\sigma_j:=2+\sum_{k=1}^j\rho_k\in [\sigma_*,\sigma^*]$, $0\le j\le m-1$. We also define $\sigma_m=2+\rho_\Sigma=2+\sum_{j=1}^m\rho_j$, but do not require $\sigma_m\in[\sigma_*,\sigma^*]$. Define
\BGE \alpha_0=2\gamma,\quad \alpha_j=\rho_j \gamma,\quad 1\le j\le m.\label{alphaj}\EDE
Note that   $\alpha=\sum_{j=0}^m \alpha_j$.
Let $w\in\R$ and $v_m\le \cdots\le v_1\in(-\infty,w)\cup\{ w^-\}\subset\R_w$. We write $\ulin v$ for $(v_1,\dots,v_m)$. Let $\eta$ be an SLE$_\kappa(\ulin\rho)$ curve in $\HH$ started from $w$ with force points $\ulin v$. The assumption that $\sum_{k=1}^j\rho_k=\sigma_j-2\ge \sigma_*-2>-2$, $1\le j\le m-1$, and $\sum_{k=1}^m  \rho_k=\rho_\Sigma>-2$ implies that the continuation threshold is never met, and the lifetime  of $\eta$ is $\infty$.

Define  $G:\{(w,\ulin v,u)\in\R^{m+2}:w\ge v_1\ge \cdots\ge v_m>u\}\to (0,\infty)$ by
 \BGE G(w,\ulin v;u)= |w-u|^{-\alpha_0} \prod_{j=1}^m |v_j-u|^{-\alpha_j}\label{G-univ}\EDE
 \BGE  = \Big(\frac{|v_1-u|}{|w-u|}\Big)^{\sigma_0\gamma}\cdot \Big(\frac{|v_2-u|}{|v_1-u|}\Big)^{\sigma_1\gamma} \cdots \Big(\frac{|v_m-u|}{|v_{m-1}-u|}\Big)^{\sigma_{m-1}\gamma}\cdot\Big(\frac 1{|v_m-u|}\Big)^{\sigma_m\gamma} .\label{G-rewrite}\EDE
For  $u\in\R$, $r>0$, and $\FF\in\{\C,\R\}$, let $B_{\FF}(u,r)=\{z\in\FF:|z-u|\le r\}$, and $\tau^u_{r;\FF}=\inf(\{t: \eta(t)\in B_{\FF}(u,r)\}\cup \{\infty\})$. Then $\eta\cap B_{\FF}(u,r)\ne\emptyset$ iff $\tau^u_{r;\FF}<\infty$. We also define $\tau^u_r=\tau^u_{r;\C}$.

The purpose of this section is to prove the following theorem.

\begin{Theorem} [One-point Green's Function]
Suppose either $\rho_\Sigma\in ((-2)\vee (\frac\kappa 2-4),\frac\kappa 2-2)$ and $\FF\in \{\R,\C\}$, or $\rho_\Sigma\ge \frac\kappa 2-2$ and $\FF=\C$.
  There are a constant $C_{\FF}>0$ depending only on $\FF,\kappa,\rho_\Sigma$, and a constant $\beta\in(0,1)$ depending only on $\kappa,\rho_\Sigma,\sigma_*$  such that, for any $u\in(-\infty,v_m)$, as $ r/{|v_m-u|}\to 0^+$,
   \BGE \PP[\eta\cap B_{\FF}(u,r)\ne \emptyset]=C_{\FF} G(w,\ulin v;u) r^{\alpha} + O\Big(\Big(\frac{|v_m-u|}{|w-u|}\Big)^{\alpha_*} \Big(\frac r{|v_m-u|}\Big)^{\alpha+\beta}\Big).\label{Green-1pt}\EDE
Recall that the implicit constants in the $O$ symbol  depend only on $\kappa,\rho_\Sigma,\sigma^*,\sigma_*$.
  \label{main-thm1}
\end{Theorem}

\begin{Remark}
The theorem does not include the case that $\FF=\R$ and $\rho_\Sigma\ge \frac\kappa 2-2$ because in that case $\eta$ a.s.\ does not hit $(-\infty,v_m)$.
  We do not have exact formulas of $C_{\FF}$  in general (unless $\FF=\R$ and $\rho=0$), as happened in \cite{Mink-real,LR}.  We do have an exact formula of $\beta$: $\frac 1\beta=\frac{\rho_\Sigma+4}{\rho_\Sigma+3}+\frac{\alpha_*+1}{\alpha_*}$,  but it is very unlikely to be sharp. {To prove the theorem, we will prove Theorem \ref{m=1} after some lemmas, which is the ``$m=1$'' case of Theorem \ref{main-thm1}. The general case of Theorem \ref{main-thm1} follows from the special case and the fact that the force point functions tend to merge together.}
\end{Remark}

From now on, we fix $u\in(-\infty, v_{  m})$. Let $L_0=|w-v_m|\ge 0$ and $L_1=|v_m-u|>0$.  Using that $\sigma^*\ge \sigma_k\ge \sigma_*>0$, $|w-u|\ge |v_1-u|\ge \cdots \ge |v_m-u|>0$,  and  (\ref{alpha},\ref{G-rewrite}), we get
 \BGE \Big(\frac{L_1}{L_0+L_1}\Big)^{\alpha^*}\Big(\frac {r}{L_1}\Big)^\alpha\le G(w,\ulin v;u) r^\alpha\le\Big(\frac{L_1}{L_0+L_1}\Big)^{\alpha_*}\Big(\frac {r}{L_1}\Big)^\alpha\le \Big(\frac {r}{L_1}\Big)^\alpha. \label{G-upper}\EDE

By treating $u$ as a force point of $\eta$ with force value $0$,  we may apply Theorem \ref{Thm-est1} to conclude that
\BGE \PP[\dist(\eta,u)\le r]\lesssim  ( {L_1}/{(L_0+L_1)} )^{\alpha_*} ( {r}/{L_1} )^{\alpha},\quad\mbox{for any } r\in(0,L_1),\label{est-u}\EDE
which implies that
\BGE \PP[\dist(\eta,u)\le r]\lesssim  ( {r}/{L_1})^{\alpha}\quad\mbox{for any } r>0.\label{est-u*}\EDE
When $v_m<w^-$, i.e., $L_0>0$, by Theorem \ref{Thm-est1},
\BGE \PP[\dist(\eta,v_m)\le r]\lesssim ( {r}/{L_0} )^{\alpha_*},\quad \mbox{for any }r>0.\label{est-vm}\EDE

Let $K_t,g_t,\F,W,\ulin V=(V_1,\dots,V_m)$ be as in the previous section. Let $U$ and $D_u$ be defined by (\ref{U-Du}).
 We have the following  estimate.

\begin{Lemma}
Let  $R>r\in (0,|v_m-u|)$. Let $\tau$ be an $\F$-stopping time.
Then
\BGE\PP[\tau^u_r<\infty |\F_\tau,\{\tau\le \tau^u_R\}\}]\lesssim ( r/R)^{\alpha}.\label{interior-estimate-formula}\EDE
\label{interior-estimate}
\end{Lemma}
\begin{proof}
We may assume that $r<R/2$ since the conditional probability is bounded by $1$. We only need to prove (\ref{interior-estimate-formula}) with $E:=\{ \tau\le \tau^u_R\}$ replaced by  $E_0:=\{ \tau<\tau^*_u\wedge \tau^u_R\}$ since the conditional probability on $E\sem E_0$ is $0$. Suppose $E_0$ occurs. By DMP of SLE$_\kappa(\ulin\rho)$, the conditional probability equals the probability that an SLE$_\kappa(\ulin\rho)$ curve $\eta^\tau$ in $\HH$ started from $W(\tau)$ with force points $\ulin V(\tau)$ visits $g_\tau(B_{\C}(u,r))$. Recall that $g_\tau$ maps $\C\sem (K_\tau^{\doub}\cup [v_m,\infty))$, which contains $u$, conformally onto $\C\sem [V_m(\tau),\infty)$.
By Koebe $1/4$ theorem, $|V_m(\tau)-U(\tau)|\ge D_u(\tau) R/4$. Since $r<R/2$, by Koebe distortion theorem, $\rad_{U(\tau)} (g_\tau(B_{\C}(u,r)))\le 4 D_u(\tau)r$. So $\rad_{U(\tau)} (g_\tau(B_{\C}(u,r)))/|V_m(\tau)-U(\tau)|\le 16 r/R$. By (\ref{est-u*}), the conditional probability that $\eta^\tau$ visits $g_\tau(B_{\C}(u,r))$ is bounded by a positive constant
times $(r/R)^{\alpha}$. So we have (\ref{interior-estimate-formula}).
\end{proof}

For $t\ge 0$, we write $\pa_{\HH}^+ K_t$ for the union of the ``right side''  of $\pa K_t$ and the real interval $[b_{K_t},\infty)$. More precisely, $\pa_{\HH}^+ K_t$ is the $g_{t}^{-1}$ image of the interval $[W(t),\infty)$ as a set of prime ends of $\HH\sem K_t$.
We have the following  estimate.

\begin{Lemma}
Let $\tau$ be an $\F$-stopping time. On the event $\tau<\infty$, let $I$ be an $\F_\tau$-measurable crosscut of ${\HH\sem K_\tau}$, which disconnects $g_\tau^{ {-1}}(V_m(\tau))=v_m\wedge a_{K_\tau}$, viewed as a prime end of ${\HH\sem K_\tau}$,  from $\pa^+_{\HH} K_\tau$ in ${\HH\sem K_\tau}$. Then
\BGE \PP[\eta \mbox{ visits }I|\F_\tau,\{\tau<\infty\}]\lesssim \exp (-\pi \alpha_* d_{\HH\sem K_\tau}(I,\pa^+_{\HH} K_\tau) ).\label{boundary-estimate-formula}\EDE
\label{boundary-estimate}
\end{Lemma}
\begin{proof}
   By DMP   the conditional probability equals the probability that an SLE$_\kappa(\ulin\rho)$ curve $\eta^\tau$ in $\HH$ started from $W(\tau)$ with force points $\ulin V(\tau)$  visits $g_\tau(I)$. By the assumption of $I$, $g_\tau(I)$ is a crosscut of $\HH$ that disconnects $V_m(\tau)$ from $[W(\tau),\infty)$.
   By conformal invariance of extremal length, $d_{\HH}(g_\tau(I), [W(\tau),\infty))=d_{\HH\sem K_\tau}(S,\pa^+_{\HH} K_\tau)$.  By Proposition \ref{extrem-prop}, $1\wedge \frac{\rad_{V_m(\tau)} (g_\tau(I))}{|W(\tau)-V_m(\tau)|}\lesssim \exp({-\pi d_{\HH}(g_\tau(I), [W(\tau),\infty))|})$. So we get the conclusion using (\ref{est-vm}).
\end{proof}

The following lemma is the first two-point estimate in this paper.

\begin{Lemma}
Assume $v_m<w$.  Then for any $r_0,r_1>0$,
   \BGE \PP[\dist(v_m,\eta)\le r_0,\dist(u,\eta)\le r_1]\lesssim ( {r_0}/{L_0} )^{\alpha_*}  ( {r_1}/{L_1} )^{\alpha}.\label{two-pt-ineq}\EDE
\label{two-pt-lem}
\end{Lemma}
\begin{proof}
 If $r_1\ge L_1/6$, then we get  (\ref{two-pt-ineq}) by applying (\ref{est-vm}) to bound $\PP[\dist(v_m,\eta)\le r_0]$. If $r_0\ge (L_0\wedge L_1)/6$ and $r_1<L_1/6$, then we get (\ref{two-pt-ineq}) by applying (\ref{est-u}) to bound $\PP[\dist(u,\eta)\le r_1]$ and using the fact that $\frac{L_1}{L_0+L_1}\asymp \frac{L_0\wedge L_1}{L_0}$.
   So we may now assume that $r_0<(L_0\wedge L_1)/6$  and $r_1< L_1/6$.
   WLOG, we may assume that $r_1=e^{-N}L_1/2$ for some $N\in\N$. Let $\tau^0=\tau^{v_m}_{r_0}$ and $\tau^1_n=\tau^u_{L_1 e^{-n}/2}$
   , $n\in\Z$.  Then (\ref{two-pt-ineq}) is equivalent to
\BGE \PP[\tau^0<\infty,\tau^1_{N}<\infty]\lesssim ( {r_0}/{L_0} )^{\alpha_*}     e^{-\alpha N} .\label{transform*}\EDE

By (\ref{est-vm},\ref{est-u}),
\BGE \PP[\tau^0<\infty]\lesssim (r_0/L_0)^{\alpha_*};\label{tau1<infty}\EDE
\BGE \PP[\tau^1_n<\infty]\lesssim ({L_1}/({L_0}+L_1))^{\alpha_*} e^{-\alpha n},\quad n\in\N.\label{tau2n}\EDE
By (\ref{interior-estimate-formula}),
\BGE \PP[\tau^1_N<\infty|\F_{\tau^1_n},\tau^1_{n}<\infty]\lesssim e^{-\alpha(N-n)},\quad n\le N.\label{tau2n-N}\EDE
Note that, when $\tau^1_n<\tau^0$, every curve in $\HH\sem K_{\tau^1_n}$ connecting $I:=\{z\in\lin\HH:|z-v_m|= r_0\}$ with $\pa_{\HH}^+ K_{\tau^1_n}$ must cross  $A_{\HH}(u,L_1 e^{-n}/2,L_1/2)$  and $A_{\HH}(v_m,r_0,L_1/2)$. By comparison principle of extremal length,  $$d_{ \HH\sem K_{\tau^1_n}}(I,\pa_{\HH}^+ K_{\tau^1_n})\ge \frac 1\pi \ln \Big(\frac{L_1/2}{r_0} \Big)+\frac 1\pi \ln\Big (\frac{L_1/2}{L_1 e^{-n}/2 }\Big).$$
So by (\ref{boundary-estimate-formula}),
\BGE \PP[\tau^0<\infty|\F_{\tau^1_n},\tau^1_n<\tau^0]\lesssim ({r_0}/{L_1  })^{\alpha_*} e^{-\alpha_* n},\quad n\in\N;\label{tau2n-tau1}\EDE

We now use (\ref{tau1<infty}-\ref{tau2n-tau1}) to prove (\ref{transform*}). First, (\ref{tau1<infty}) and (\ref{tau2n-N}) (applied to $n=1$) together imply that
\BGE \PP[ \tau^0<\tau^1_1\le\tau^1_{N}<\infty ]\lesssim ( {r_0}/{L_0} )^{\alpha_*}     e^{-\alpha N}. \label{E0}\EDE
Second, (\ref{tau2n},\ref{tau2n-tau1},\ref{tau2n-N})  together imply that
\BGE \PP[\tau^1_n<\tau^0<\tau^1_{n+1}\le \tau^1_N<\infty]\lesssim ( {r_0}/{L_0  })^{\alpha_*}  e^{-\alpha N}e^{-\alpha_* n} ,\quad 1\le n\le N-1.\label{En}\EDE
Third,  (\ref{tau2n})  and (\ref{tau2n-tau1}) together (applied to $n=N$) imply that
\BGE  \PP[\tau^1_N<\tau^0<\infty]\lesssim ( {r_0}/{L_0  })^{\alpha_*}  e^{-\alpha N}e^{-\alpha_* N} .\label{EN}\EDE
Summing up  (\ref{E0},\ref{En},\ref{EN}), we get the desired inequality (\ref{transform*}).
\end{proof}

We are now able to prove Theorem \ref{main-thm1} in the case $m=1$.

\begin{Theorem}
Suppose $m=1$. Define $ \beta_\Sigma=\frac{\rho_\Sigma +3}{\rho_\Sigma +4}$.  Then there is a constant $C_{\FF}>0 $  depending only on $\kappa,\rho_\Sigma$ such that, as $r/{L_1}\to 0^+$,
  \BGE \PP[\eta\cap B_{\FF}(u,r)\ne\emptyset]=C_{\FF} \Big(\frac{L_1}  { L_0+L_1}\Big)^{\alpha_0} \Big(\frac r{L_1}\Big)^{\alpha}\Big(1+O\Big(\frac r{L_1}\Big)^{\beta_\Sigma}\Big),\label{one-force-point}\EDE
  where the implicit constants also depend only on $\kappa,\rho_\Sigma$.
  \label{m=1}
\end{Theorem}
 {\begin{Remark}
  The idea of the proof is that we reparametrize $\eta$ by some conformal radius, use the driving function and force point process to construct a radial Bessel process of finite lifetime, and apply the results in Section \ref{Section 2.4}. \label{m=1-rem}
\end{Remark}}

\begin{proof}[ {Proof of Theorem \ref{m=1}}]
Since $m=1$, $\rho_\Sigma=\rho_1$. Throughout the proof, a constant depends only on $\kappa,\rho_1$.
By translation and scaling, it suffices to derive (\ref{one-force-point}) in the case that $w=1$ and $u=0$.
For $x\in(0,1]$, let $\eta_x$ be an SLE$_\kappa(\rho_1)$ curve in $\HH$ started from $w=1$ with the force point $v_1=x$. Here if $x=1$, then the force point $v_1$ is set to be $1^-$.
For $s\in(0,1)$, let $p_{\FF}(x,s)$ denote the probability that  $\eta_x$ hits $B_{\FF}(0, x\cdot s)$.    By scaling and translation, we need to show there is a constant  $C_{\FF}> 0$  such that
\BGE p_{\FF}(x,s)=C_{\FF} x^{\alpha_0} s^{\alpha}(1+O(s^{\beta_\Sigma})),\quad \mbox{as }s\to 0.\label{pC0}\EDE

Define  $R$ and $\Phi$   on $[0,\tau_u^*)$ by
$$R= \frac{V_1 -U }{W -U },\quad \Phi = -\frac \kappa2 \ln\Big(\frac{|V_1-U| }{D_u x}\Big);$$
and let  $\ha R(t)= R\circ \Phi^{-1}(t)$, $0\le t<\ha \tau_u^*:=\sup \Phi[0,\tau^*_u)$. These functions differ from the $R,\Phi,\ha R$ in the proof of Theorem \ref{Thm-est1} by that $V_1$ and $v_1$ are used here in place of $V_0$ and $v_0$.
An argument similar to the one used to derive (\ref{dhaR}) shows that  $\ha R$ satisfies the SDE (\ref{Bessel-SDE-X})  with $\delta_+:=\frac 4\kappa (\rho_1+2)>0$ and $\delta_-:=\frac 4\kappa(\kappa-4-\rho_\Sigma)<2$. Here we use the assumption that
$\rho_1>-2$ and $\rho_\Sigma>\frac\kappa 2-4$.

As $t\uparrow \tau_u^*$, $\eta_x(t)$ tends to $\infty$ (when $\tau_u^*=\infty$) or some point on $(-\infty,u)$ (when $\tau_u^*<\infty$). In either case, by comparison principle, the extremal distance between $[u,v_1\wedge a_{K_t}]$ and $\pa_{\HH}^+ K_t$ in $\HH\sem K_t$ tends to $\infty$ as $t\uparrow \tau_u^*$. Thus, by conformal invariance of extremal distance and Proposition \ref{extrem-prop},   $\lim_{t\uparrow \tau^*_u}R(t)=0$, which implies that $\lim_{t\uparrow \ha \tau^*_u}\ha R(t)=0$. So we may apply Proposition \ref{Bessel-transition} (iii) and (\ref{invariant-def},\ref{transition-rate}) to conclude that $\ha R$ has a transition density $p^{\ha R}_t(x,y)$ and a quasi-invariant density $p^{\ha R}_\infty(y)=C_1 (1-y)^{\alpha_+}$ with decay rate $\beta_0$,  such that
\BGE p^{\ha R}_t(x,y)=C_2 x^{\alpha_-} p^{\ha R}_\infty(y) e^{-\beta_0 t}(1+O(e^{-(\beta_1-\beta_0)t})),\quad \mbox{as }t\to\infty,\label{transition-invariant}\EDE
where $C_1,C_2$ are positive constants, and
$$ \alpha_+=\frac{\delta_+}2-1,\quad \alpha_-=1-\frac{\delta_-}2,\quad \beta_n=\frac 12(n+1+\alpha_+)(n+\alpha_-),\quad n=0,1.$$

Let $L>0$, and $\tau_L=\inf(\{t\in [0,\tau_u^*):\Phi(t)\ge L\}\cup \{\infty\}) $.  By Koebe $1/4$ theorem, if $\tau_L<\tau^*_u$,
\BGE x e^{-\frac 2\kappa L}=x e^{-\frac 2\kappa \Phi(\tau_L)}=\frac{|V_1(\tau_L)-U(\tau_L)|}{D_u(\tau_L)} \le \dist(u,\eta_x[0,\tau_L])\wedge x\le 4x e^{-\frac 2\kappa L}.\label{Koebe-dist}\EDE
Let $\til g_t=\frac{g_t-U(t)}{W(t)-U(t)}$, $0\le t<\tau^*_u$. Then $\til g_t$ maps $\HH\sem K_t$ conformally onto $\HH$, fixes $\infty$ and $0$, and sends $\eta_x(t)$ and $v_1\wedge a_{K_t}$ respectively to $1$ and $R(t)$.
By DMP, there is a random curve $\eta^L$ in $\lin\HH$ defined on the event $\{\tau_L<\tau^*_u\}$, whose law conditionally on $\F_{\tau_L}$ and $\{\tau_L<\tau^*_u\}$ is that of $\eta_{\ha R(L)}$, i.e., an SLE$_\kappa(\rho_1)$ curve started from $1$ with the force point $ \ha R(L)$, such that $\eta_x(\tau_L+\cdot)=\til g_{\tau_L}^{-1}(\eta^L)$.
Suppose $s< \frac 14 e^{-\frac 2\kappa L}$. Then $\eta_x$ hits $B_{\FF}(0,xs)$ iff $\tau_L<\tau_u^*$ and  $\eta^L$ hits $\til g_{\tau_L}(B_{\FF}(0,xs))$.
Since $\til g_{\tau_L}'(u)=\frac{D_u(\tau_L)}{W(\tau_L)-U(\tau_L)}=\ha R(L)e^{\frac 2\kappa L} $, by (\ref{Koebe-dist}) and  Koebe distortion theorem,
$$B_{\FF}\Big(0,  \frac{\ha R(L) se^{\frac 2\kappa L} }{(1+4s e^{\frac 2\kappa L})^2}\Big) \subset \til g_{\tau_L}(B_{\FF}(0, xs)) \subset  B_{\FF}\Big(0, \frac{\ha R(L) s e^{\frac 2\kappa L} }{(1-4s e^{\frac 2\kappa L})^2}\Big),$$
which implies that
$$p_{\FF}(x,s)\gtreqless  \EE\Big[\ind_{\{\tau_L<\tau_u^*\}}\cdot p_{\FF}\Big(\ha R(L), \frac{s e^{\frac 2\kappa L} }{(1\pm 4s e^{\frac 2\kappa L})^2}\Big)\Big].$$
Here when we choose $+$ (resp.\ $-$) in $\pm$, the inequality holds with $\ge$ (resp.\ $\le$).
Since $\ha R$ has a transition density $p^{\ha R}_t(x,y)$, the above inequality further implies that
\BGE p_{\FF}(x,s) \gtreqless \int_0^1 p^{\ha R}_L(x,y) p_{\FF} \Big(y,   \frac{s e^{\frac 2\kappa L}}{(1\pm 4s e^{\frac 2\kappa L})^2}\Big) dy.\label{p(x,s)-integral}\EDE
Let $p_{\FF}(s)=\int_0^1 p_{\FF}(x,s) p^{\ha R}_\infty(x) dx$ and  $q_{\FF}(s)=s^{-\frac \kappa 2\beta_0 } p_{\FF}(s)$.  By (\ref{invariant=}), if $0<s< \frac 14 e^{-\frac 2\kappa L}$, then
$$p_{\FF}(s) \gtreqless \int_0^1\!\int_0^1 p^{\ha R}_\infty(x) p^{\ha R}_L(x,y) p_{\FF}  \Big(y, \frac{se^{\frac 2\kappa L} }{(1\pm 4s e^{\frac 2\kappa L})^2}\Big) dy dx$$
$$=e^{-\beta_0L} \int_0^1 p^{\ha R}_\infty(y)  p_{\FF}  \Big(y, \frac{se^{\frac 2\kappa L} }{(1\pm 4s e^{\frac 2\kappa L})^2}\Big) dy
=e^{-\beta_0L} p_{\FF}\Big(\frac{se^{\frac 2\kappa L} }{(1\pm 4s e^{\frac 2\kappa L})^2}\Big),
$$
and so
\BGE q_{\FF}(s)  \gtreqless  (1\pm 4s e^{\frac 2\kappa L})^{-\kappa \beta_0} q_{\FF}\Big(\frac{se^{\frac 2\kappa L} }{(1\pm 4s e^{\frac 2\kappa L})^2}\Big).\label{ineq-q}\EDE
Let $\eps\in(0,1)$, $a_+=\eps$, and $x_0=\frac{a_+/4}{(1+a_+)^2}$. Choose $a_-\in(0,a_+)$ such that $a_++\frac 1{a_+}+2=a_-+\frac 1{a_-}-2$. Then $\frac{a_-/4}{(1+a_-)^2}=x_0$. Let $s\in (0,a_-/4)$. Then there are $L_+,L_->0$ such that
 $4s e^{\frac 2\kappa L_\pm}=a_\pm$, which implies $\frac{se^{\frac 2\kappa L_\pm} }{(1\pm 4s e^{\frac 2\kappa L_\pm})^2}=\frac{a_\pm /4}{(1\pm a_\pm)^2}=x_0$. By (\ref{ineq-q}),
 \BGE q_{\FF}(s)\gtreqless (1\pm a_\pm) ^{-\kappa \beta_0} q_{\FF}(x_0 ) \gtreqless (1\pm \eps)^{-\kappa \beta_0}q_{\FF}(x_0 ) ,\quad \forall s\in (0,a_-). \label{ineq-q2}\EDE
 This inequality implies that $\lim_{s\to 0^+} q_{\FF}(s)$ converges. Let $Q_{\FF}$ denote the limit.
On the other hand, for any $x\in(0,1/16)$, there are $a_+>a_-\in(0,1)$ such that $\frac{a_\pm/4}{(1\pm a_\pm)^2}= x$, and as $x\to 0^+$, $a_\pm=O(x)$. For  $s\in(0,a_-/4)$, there are $L_+>L_->0$ such that $4s e^{\frac 2\kappa L_\pm}=a_\pm$. Using (\ref{ineq-q2}), we get $q_{\FF}(s) \gtreqless (1\pm a_\pm) ^{-\kappa \beta_0} q_{\FF}(x)$. Sending $s$ to $0^+$, we get $q_{\FF}(x)=Q_{\FF} (1+O(x))$ as $x\to 0^+$. Thus,
 \BGE p_{\FF}(s)=Q_{\FF}  s^{\frac \kappa 2\beta_0 } (1+O(s)),\quad \mbox{as } s\to 0^+.\label{p-Q}\EDE

Setting $L =(\frac 2\kappa+(\beta_1-\beta_0))^{-1}\ln(1/s)>0$ and using (\ref{transition-invariant},\ref{p(x,s)-integral}), we find that as $s\to 0^+$,
 $$ p_{\FF}(x,s) \gtreqless  \int_0^1 C_2 x^{\alpha_-} p^{\ha R}_\infty(y) e^{-\beta_0 L} (1+O(e^{-(\beta_1-\beta_0) L})) p_{\FF}  \Big(y,   \frac{s e^{\frac 2\kappa L}}{(1\pm 4s e^{\frac 2\kappa L})^2}\Big) dy.$$
 $$=C_2 x^{\alpha_-} e^{-\beta_0 L} (1+O(e^{-(\beta_1-\beta_0) L})) p_{\FF} \Big(  \frac{s e^{\frac 2\kappa L}}{(1\pm 4s e^{\frac 2\kappa L})^2}\Big)$$
 $$=C_2 Q_{\FF} x^{\alpha_-} s^{\frac \kappa2 \beta_0}(1+O(e^{-(\beta_1-\beta_0)L})+O(s e^{\frac 2\kappa L}))=C_2Q_{\FF} x^{\alpha_-}  s^{\frac \kappa 2 \beta_0} (1+O(s^{\frac{\beta_1-\beta_0}{\frac 2\kappa+\beta_1-\beta_0}})).$$
 Since  ${\frac \kappa2 \beta_0}=\alpha$, $\alpha_-=\alpha_0$, and $\frac{\beta_1-\beta_0}{\frac 2\kappa+\beta_1-\beta_0}=\beta_\Sigma$, we conclude that (\ref{pC0}) holds with $C_{\FF}=C_2 Q_{\FF}$.

Finally, we check that $C_{\FF}>0$, which is equivalent to that $Q_{\FF}>0$ since $C_2>0$.
Since $q_{\FF}\ge 0$, we have $Q_{\FF}\ge 0$.  If $Q_{\FF}=0$, then (\ref{pC0}) implies that there is  $c_0\in(0,1)$  such that  $p_{\FF}(x,c_0)=0$ for all $0<x\le 1$. First, suppose $\FF=\C$. By the properties of the transition density  of $\ha R$,  $\PP[\tau_L<\tau^*_u]>0$ for any $L>0$. By (\ref{Koebe-dist}), $\PP[\dist(u,\eta_x)\le \eps]>0$ for any $\eps>0$, which contradicts that $p_{\C}(x,c_0)=0$. Second, suppose $\FF=\R$.
Then we have $\rho_\Sigma<\frac\kappa 2-2$ by assumption. Let $\eta_{\bf x}$ be an SLE$_\kappa(\rho_\Sigma)$ curve started from $1$ with a random force point ${\bf x}\in(0,1)$, whose law is the quasi-invariant density of $\ha R$.
Since $\rho_\Sigma<\frac\kappa 2-2$, $\eta_{\bf x}$ a.s.\ hits $(-\infty,0)$ (cf.\ \cite{MS1}). So there is $N<0$ such that $\PP[\eta_{\bf x}\cap (N{\bf x},0)\ne \emptyset]>0$.  Choose $L>0$ such that $\frac{c_0 e^{\frac 2\kappa L}}{576}>|N|+1$. Let $\til g_{\tau_L}$ and $\eta^L$ be associated with $\eta_{\bf x}$. Recall that, conditionally on $\F_\tau$, $\eta^L$ is an SLE$_\kappa(\rho_1)$ curve started from $1$ with the force point $\ha R(L)$.  Let $X=\til g_{\tau_L}(-c_0 {\bf x})<0$ and $r=\dist(0,\eta_{\bf x}[0,\tau_L])$. By (\ref{Koebe-dist}), $r\le 4{\bf x} e^{-\frac 2\kappa L}$. Since  $(-\infty,-c_0 {\bf x})$ and $[0,{\bf x}\wedge a_{K_{\tau_L}})$ could be separated by $\A_{\HH}(0,r,c_0{\bf x})$ in $\HH\sem K_{\tau_L}$, by Proposition \ref{extrem-prop} and conformal invariance and comparison principle of extremal length,  $\frac{\ha R(L)}{\ha R(L)-X}\le \frac{144r}{c_0{\bf x}}\le \frac{576 {\bf x} e^{-\frac 2\kappa L}}{c_0{\bf x}}<\frac 1{|N|+1}$, which implies that $X<N \ha R(L)$.  Since $\ha R(L)$ is the force point for $\eta^L$, and conditionally on $\{\tau_L<\tau^*_u\}$ has the same law as ${\bf x}$ (because the law of ${\bf x}$ is the quasi-invariant density of $\ha R$), we know that
 $(\eta^L,\ha R(L))$ conditionally on $\{\tau_L<\tau^*_u\}$ has the same law as $(\eta_{\bf x},{\bf x})$. From $X<N \ha R(L)$ we know that  $\PP[\eta^L\cap (X,0)\ne\emptyset]>0$, which implies that $\PP[\eta_{\bf x}\cap (c_0 {\bf x},0)\ne\emptyset]>0$. But this contradicts that $p_{\R}({\bf x},c_0)=0$.
\end{proof}

\begin{Remark}
  Theorem \ref{m=1} holds with slight modification if $\eta$ has two force points $v_1,u$, and the force values satisfy $\rho_1>-2$ and $\rho_1+\rho_u>\frac\kappa 2-4$. In this case, we redefine $\rho_\Sigma,\alpha,\beta$ using $\rho_\Sigma=\rho_1+\rho_u$, $\alpha=\frac{(\rho_1+2)(2\rho_\Sigma+8-\kappa)}{2\kappa}$, $\beta=\frac{\rho_1+\rho_\Sigma +6}{\rho_1+\rho_\Sigma +8}$, still define $\alpha_0$ using (\ref{alphaj}), and allow all constants to depend on $\kappa,\rho_1,\rho_\Sigma$. Then the proof works without further modification.
\end{Remark}

\begin{Lemma}
Let $\delta=|v_1-v_m|$ and $R\ge 10\delta$. Suppose  $|w-v_m|\ge R$.
Let $\PP_m$ denote the law of the $m$-force-point SLE$_\kappa(\ulin\rho)$ curve $\eta$ in Theorem \ref{main-thm1}.
Let $\PP_1$ denote the law of a one-force-point SLE$_\kappa(\rho_\Sigma)$ curve started from $w$ with the force point $v_m$.  Then $\PP_m$ and $\PP_1$ are mutually absolutely continuous on $\F_{\tau^{v_m}_R}$, and there is a constant $C>0$ depending only on $\kappa,\rho_\Sigma,\sigma^*$ such that
\BGE \Big|\ln\Big(\frac{d\PP_m|\F_{\tau^{v_m}_R}}{d\PP_1|\F_{\tau^{v_m}_R}}\Big)\Big|\le C\cdot \frac {\delta } R.\label{dP1/P2}\EDE
  \label{lemma-merge}
\end{Lemma}
\begin{proof}
Throughout this proof, an explicit or implicit constant depends only on $\kappa,\rho_\Sigma,\sigma^*$.
 First suppose $\eta$ follows the law $\PP_m$. When $\tau^*_{v_m}<\infty$, we have $W=V_1=\cdots=V_m$ at $\tau^*_{v_m}$. The DMP asserts that, under $\PP_m$, conditionally on $\F_{\tau^*_{v_m}}$ and the event $\{\tau^*_{v_m}<\infty\}$, the law of $\eta(\tau^*_{v_m}+\cdot)$ is that of the $f_{\tau^*_{v_m}}$-image of an SLE$_\kappa(\rho_\Sigma)$ curve in $\HH$ started from $W(\tau^*_{v_m})$ with force point $W(\tau^*_{v_m})^-$. Here we use the fact that all force point processes merge together after the time $\tau^*_{v_m}$. The  statement clearly also holds if $\eta$ follows the law $\PP_1$. Thus, it suffices to prove that (\ref{dP1/P2}) holds with $\tau_0:=\tau^*_{v_1}\wedge \tau^{v_m}_R$ in place of $\tau^{v_m}_R$.

Since $W(0)=w>v_1=V_1(0)$, we have $W>V_1$ on $[0,\tau_{v_1}^*)$.
Let   $D_j(t)=g_t'(v_j)$. Then each $D_j$ satisfies the ODE $\frac{d D_j}{D_j}=\frac{-2}{(W-V_j)^2}$.  Define $N_1$ and $N_m$ on $[0,\tau_{v_1}^*)$ by
$$N_1={D_m^{ \frac{\rho_\Sigma(\rho_\Sigma+4-\kappa)}{4\kappa}} |W-V_m|^{ \frac{\rho_\Sigma}\kappa}  },\quad N_m={\prod_{j=1}^m \Big( D_j^{\frac{\rho_j(\rho_j+4-\kappa)}{4\kappa}} |W-V_j|^{\frac{\rho_j}\kappa}  \Big)\prod_{1\le j<k\le m} |V_j-V_k|^{\frac{\rho_j\rho_k}{2\kappa}}},$$
and let $N=N_m/N_1$. Here for $j\in\{1,m\}$, $N_j$ is a local martingale introduced in \cite{SW}, such that $\PP_j$ can be obtained by locally weighting a chordal SLE$_\kappa$  curve (with no force point) from $w$ to $\infty$ by $N_j/N_j(0)$. Thus, if we locally weight $\PP_1$ by $N/N(0)$, then we get $\PP_m$. More precisely, this means that,
if $\tau$ is an $\F$-stopping time with $\tau\le \tau_{v_1}^*$ such that $N/N(0)$ is uniformly bounded on $[0,\tau)$, then
\BGE \frac{d\PP_m|\F_{\tau}}{d\PP_1|\F_{\tau}}=\frac{N(\tau)}{N(0)}.\label{P2/P1-tau-N}\EDE
Here if $\tau=\tau_{v_1}^*$,  then $N(\tau)$ is understood as $\lim_{t\to \tau^-} N(t)$, which exists a.s.\ on the event $\{\tau=\tau_{v_1}^*\}$.
These facts can also be checked directly using It\^o's formula and Girsanov theorem.
We claim that
\BGE \sup_{0\le t<\tau_0} \Big|\ln\Big(\frac{N(t)}{N(0)}\Big)\Big|\lesssim \frac {\delta } R.\label{N/N-bound}\EDE
If the claim is true, then $N/N(0)$ is uniformly bounded on $[0,\tau_0)$, and $\ln( \frac{N(\tau_0)}{N(0)})\lesssim \frac {\delta } R$. So we may apply (\ref{P2/P1-tau-N}) to $\tau=\tau_0$ and get (\ref{dP1/P2}) with $\tau_0$ in place of $\tau^{v_m}_R$.

Now we prove (\ref{N/N-bound}). Let $X_j = \frac{W -V_j }{W-V_m} $. Then $0\le X_1\le \cdots\le X_m=1$ on $[0,\tau_{v_1}^*)$. Define positive functions $D_{j,k}$ on $[0,\tau_{v_1}^*)$, $1\le j,k\le m+1$, such that,
$$D_{j,k}(t)=\left\{\begin{array}{ll}
\frac{V_j(t)-V_k(t)}{v_j-v_k},& \mbox{if }j,k\le m\mbox{ and }v_j\ne v_k;\\
D_j(t),&\mbox{if }j,k\le m\mbox{ and }v_j=v_k;\\
1, & \mbox{if }j\mbox{ or }k=m+1.
\end{array}
\right.$$
Then we have
\BGE N/N(0)=D_{m,m}^{-\frac{\rho_\Sigma(\rho_\Sigma+4-\kappa)}{4\kappa}} \prod_{j=1}^{m-1}( X_j/X_j(0))^{\frac{\rho_j}\kappa} \prod_{j_1=1}^m \prod_{j_2=1}^m D_{j_1,j_2}^{\frac{\rho_{j_1} \rho_{j_2}}{4\kappa}} \prod_{j=1}^m D_{j,j}^{\frac{\rho_j(4-\kappa)}{4\kappa}} \label{M-re1}.\EDE

Define $Y_j$, $1\le j\le m-1$, and $E_{j_1,j_2}$, $1\le j_1,j_2\le m$, on $[0,\tau_{v_1}^*)$ by
$$Y_j=\frac{X_j}{X_{j+1}},\quad E_{j_1,j_2}=\frac{D_{j_1,j_2}D_{j_1+1,j_2+1}}{D_{j_1+1,j_2}D_{j_1,j_2+1}}.$$
Then each $Y_j\in(0,1]$.  By (\ref{dVU}) and the ODE for $D_j$,  $E_{j_1,j_2}$ satisfies the ODE:
$$\frac{d E_{j_1,j_2}}{E_{j_1,j_2}}=-2\prod_{s=1}^2 \Big(\frac 1{W-V_{j_s}}-\frac 1{W-V_{j_s+1}}\Big) dt\le 0,$$
where $\frac{1}{W-V_{m+1}}$ is understood as $0$. So $E_{j_1,j_2}$ is decreasing in $t$. From $E_{j_1,j_2}(0)=1$, we get $0<E_{j_1,j_2}\le 1$. Since $X_m=D_{k_1,m+1}=D_{m+1,k_2}=1$, we get
$$X_j=\prod_{k=j}^{m-1} Y_k,\quad D_{j_1,j_2}=\prod_{k_1=j_1}^m \prod_{k_2=j_2}^m E_{k_1,k_2}.$$
Recall that $ \sum_{j=1}^k \rho_j=\sigma_k-2$, $1\le k\le m$. So
$$ \prod_{j=1}^{m-1} X_j^{ {\rho_j} }= \prod_{j=1}^{m-1} \prod_{k=j}^{m-1} Y_k^{ {\rho_j} }=\prod_{k=1}^{m-1} \prod_{j=1}^k Y_k^{{\rho_j} } =\prod_{k=1}^{m-1} Y_k^{ {\sigma_k-2} };$$
$$\prod_{j_1=1}^m \prod_{j_2=1}^m D_{j_1,j_2}^{ {\rho_{j_1} \rho_{j_2}} } =
\prod_{j_1=1}^m \prod_{j_2=1}^m  \prod_{k_1=j_1}^m \prod_{k_2=j_2}^m E_{k_1,k_2}^{ {\rho_{j_1} \rho_{j_2}} }$$
$$=\prod_{k_1=1}^m \prod_{k_2=1}^m \prod_{j_1=1}^{k_1} \prod_{j_2=1}^{k_2} E_{k_1,k_2}^{ {\rho_{j_1} \rho_{j_2}} }=\prod_{k_1=1}^m \prod_{k_2=1}^m  E_{k_1,k_2}^{ {(\sigma_{k_1}-2)(\sigma_{k_2}-2)} }$$
$$\prod_{j=1}^m D_{j,j}^{\rho_j}=\prod_{j=1}^m \prod_{k_1=j}^m \prod_{k_2=j}^m E_{k_1,k_2}^{ {\rho_{j}} }=\prod_{k_1=1}^m \prod_{k_2=1}^m \prod_{j=1}^{k_1\wedge k_2} E_{k_1,k_2}^{\rho_j} =\prod_{k_1=1}^m \prod_{k_2=1}^m E_{k_1,k_2}^{\sigma_{k_1\wedge k_2}-2}.$$

Let $S=\{(j,k)\in\N^2:1\le j,k\le m\}\sem\{(m,m)\}$.  By (\ref{M-re1}) and the above formulas,
$$N/N(0)= \prod_{j=1}^{m-1} (Y_j/Y_j(0)) ^{\frac{\sigma_j-2}\kappa}   \prod_{(k_1,k_2)\in S} E_{k_1,k_2}^{\frac{(\sigma_{k_1}-2)(\sigma_{k_2}-2)}{4\kappa}+\frac{(4-\kappa)(\sigma_{k_1\wedge k_2}-2)}{4\kappa}}.$$
Thus, it suffices to show that (\ref{N/N-bound}) holds with $N/N(0)$ replaced respectively by
$$\prod_{j=1}^{m-1} Y_j ,\quad \prod_{j=1}^{m-1} Y_j ^{ {\sigma_j } },\quad \prod_{(j,k)\in S} E_{j,k}^{ {\sigma_j\sigma_k}  }, \quad \prod_{(j,k)\in S} E_{j,k}^{ {\sigma_j}  },\quad \prod_{(j,k)\in S} E_{j,k}^{ {\sigma_{j\wedge k}}  }, \quad \prod_{(j,k)\in S} E_{j,k} .$$

 We observe that $\prod_{(j,k)\in S} E_{j,k}={D_1}/{D_m}$. Since $E_{j,k}\le 1$ and $\sigma_j\le \sigma^*\vee (\rho_\Sigma+2)$, we have
$$  \prod_{(j,k)\in S} E_{j,k}^{ {\sigma_j}  },\prod_{(j,k)\in S} E_{j,k}^{ {\sigma_{j\wedge k}}  },
\prod_{(j,k)\in S} E_{j,k}^{ {\sigma_j\sigma_k}  }\in \Big[\Big(\frac{D_1}{D_m}\Big)^{ (\sigma^*\vee (\rho_\Sigma+2))^2 },1\Big].$$
Thus, to prove the statement for $\prod_{(j,k)\in S} E_{j,k}^{ {\sigma_j\sigma_k}  }$, $\prod_{(j,k)\in S} E_{j,k}^{ {\sigma_j}  }$, and  $\prod_{(j,k)\in S} E_{j,k}$, it suffices to show that (\ref{N/N-bound}) holds with $\frac N{N(0)}$ replaced by $\frac{D_1}{D_m}$.

Since $\prod_{j=1}^{m-1} Y_j=X_1=\frac{|W-V_1|}{|W-V_m|}$, $0<Y_j\le 1$ and $\sigma_j\le \sigma^*$,   to prove the statement for $\prod_{j=1}^{m-1} Y_j  $ and $ \prod_{j=1}^{m-1} Y_j  ^{ {\sigma_j } }$, it suffices to show that  (\ref{N/N-bound}) holds with $\frac N{N(0)}$ replaced by $\frac{|W-V_1|}{|W-V_m|}$.

The statements for  $ \frac {D_1}{D_m} $ follow immediately from Koebe distortion theorem since $D_j(t)=g_t'(v_j)$, $j=1,m$, $g_t$ is analytic on $B_{\C}(v_m,R)$, and $|v_1-v_m|=\delta \le R/10$. The statements for  $\frac{|W-V_1|}{|W-V_m|}$  follows from a combination of Koebe distortion theorem with Koebe $1/4$ theorem. To see this, let $t\in [0,\tau_0)$. Since $W(t)=g_t(\eta(t))$, and $|\eta(t)-v_m|\ge R$, by Koebe $1/4$ theorem, $|W(t)-V_m(t)|\ge D_m(t) R/4$. By Koebe distortion theorem, $|V_1(t)-V_m(t)|\le \frac {D_m(t)\delta }{(1-\delta /R)^2}$. Since $\frac{\delta } R\le \frac 1{10}$, we have
$\frac{|V_1(t)-V_m(t)|}{|W(t)-V_m(t)|}\le \frac{4 \delta /R}{(1-\delta /R)^2}\le 5 \frac {\delta } R\le\frac 12$.
Since $|\ln(1\pm x)|\le \ln (4)|x|$ when $|x|\le\frac 12$, we get $\ln(\frac{|W(t)-V_1(t)|}{|W(t)-V_m(t)|})\lesssim \frac {\delta } R$, as desired.
\end{proof}

\begin{Lemma}
 Let $C_{\FF}>0$  be as in Theorem \ref{m=1}. Let $\delta=|v_1-v_m|$. Let $\beta_\Sigma=\frac{\rho_\Sigma+3}{\rho_\Sigma+4}$ and $\beta_*=\frac{\alpha_*}{\alpha_*+1}$. Then  as $\frac r{L_1},\frac {\delta}{L_0\wedge L_1 } \to 0^+$,
\BGE \PP[\eta\cap B_{\FF}(u,r)\ne\emptyset]= C_{\FF}\Big(\frac r{L_1}\Big)^\alpha \Big(\Big(\frac{L_1}{L_0+L_1}\Big)^{\alpha_0} + O\Big (\frac r{L_1}\Big)^{\beta_\Sigma}+O\Big (\frac {\delta}{L_0}\Big)^{\beta_*} \Big)\label{v-close}\EDE
\BGE =C_{\FF} G(w,\ulin v;u) r^{\alpha}+ \Big(\frac r{L_1}\Big)^\alpha\Big( O\Big (\frac r{L_1}\Big)^{\beta_\Sigma}+O\Big (\frac {\delta}{L_0 }\Big)^{\beta_*}+O\Big(\frac {\delta}{L_1 }\Big)\Big) .\label{v-close'}\EDE
\label{v-close-lemma}
\end{Lemma}
\begin{proof}   It suffices to prove (\ref{v-close}) since by (\ref{G-univ},   $G(w,\ulin v;u) =  (\frac{L_1}{L_0+L_1} )^{\alpha_0}(\frac 1{L_1} )^\alpha(1+O(\frac{\delta}{L_1}))$ as $ \delta/{L_1}\to 0^+$.

Suppose $ {\delta}/{L_0 }\in(0,1)$ is very small. Choose $R\in(\delta,L_0)$ to be determined later.
Let $\PP_m$ and $\PP_1$ be as in Lemma \ref{lemma-merge}. Note that $\PP_m=\PP$. By Lemma \ref{two-pt-lem},
$$ |\PP_j[\tau^u_{r;\FF}<\infty]-\PP_j[\tau^u_{r;\FF}< \tau^{v_m}_R]|\lesssim  (  R/{L_0} )^{\alpha_*} (  r/{L_1} )^{\alpha},\quad j=1,m. $$
By Lemma \ref{lemma-merge} and (\ref{est-u*}), if $10\delta<R$, then
$$|\PP_m[\tau^u_{r;\FF}<\tau^{v_m}_R]-\PP_1[\tau^u_{r;\FF}<\tau^{v_m}_R]|\lesssim  ( \delta/ R)  ( r/{L_1})^{\alpha}.$$
By Theorem \ref{m=1},
$$\Big|\PP_1[\tau^u_{r;\FF}<\infty]-C_{\FF}\Big(\frac{L_1}{L_0+L_1}\Big)^{\alpha_0}\Big(\frac{r}{L_1}\Big)^\alpha \Big|\lesssim \Big(\frac{r}{L_1}\Big)^{\alpha+\beta_\Sigma}.$$
Combining the above three displayed formulas, we find that, when $\delta/R<1/10$,
$$\Big|\PP_m[\tau^u_{r;\FF}<\infty]-C_{\FF}\Big(\frac{L_1}{L_0+L_1}\Big)^{\alpha_0}\Big(\frac{r}{L_1}\Big)^\alpha \Big|\lesssim \Big(\frac r{L_1}\Big)^{\alpha}\Big(\Big(\frac R{L_0}\Big)^{\alpha_*}+   \frac \delta R  + \Big(\frac{r}{L_1}\Big)^{ \beta_\Sigma}\Big).$$
Setting $R=\delta^{\frac 1{\alpha_*+1}} L_0^{\frac{\alpha_*}{\alpha_*+1}}$ in the above estimate, we get  (\ref{v-close}).
\end{proof}

\begin{proof}[Proof of Theorem \ref{main-thm1}] Let $C_{\FF}>0$ be as in Theorem \ref{m=1}.
  Let $r<L_1$. Fix $R\in (r,L_1)$ to be determined later.  Define $M$ on $[0,\tau_u^*)$ by
$$ M=D_u^\alpha G(W,\ulin V;U)=D_u^{\alpha} |W-U|^{-\alpha_0} \prod_{j=1}^m |V_j-U|^{-\alpha_j}
  .$$
A straightforward application of It\^o's formula shows that $M$ is a local martingale. By (\ref{G-upper}),
\BGE M\le  ( {|V_m-U|}/{|W-U|} )^{\alpha_*} (  {D_u}/{|V_m-U|} )^\alpha .\label{M-other}\EDE
Let $\tau=\tau^u_R$.
By Koebe $1/4$ theorem, $\frac{D_u}{|V_m-U|}\le \frac 4R $ on $[0,\tau\wedge \tau_u^*)$. So by (\ref{M-other}) $M\le (\frac 4R)^\alpha$ on $[0,\tau\wedge \tau_u^*)$. By a standard extremal length argument, as $t\uparrow \tau_u^*$, $\frac{|V_m(t)-U(t)|}{|W(t)-U(t)|}\to 0$. This implies by (\ref{M-other}) that, on the event $\{\tau_u^*\le \tau^u_R\}$, as $t\uparrow \tau_u^*$, $M(t)\to 0$. Thus, $M(t\wedge \tau\wedge \tau_u^*)$, $0\le t<\infty$, is a uniformly bounded martingale, where when $t\wedge \tau\wedge \tau_u^*\ge \tau_u^*$, $M(t\wedge \tau\wedge \tau_u^*)$ is set to be $0$. In particular, we have $\EE[\ind_{\{\tau<\tau_u^*\}} M(\tau)]=M(0)=G(w,\ulin v;u)$.

By DMP, $\PP[\tau^u_{r;\FF}<\infty|\F_\tau,\tau<\tau^*_u]$ equals the probability that an SLE$_\kappa(\ulin\rho)$ curve $\eta^\tau$ in $\HH$ started from $W(\tau)$ with force points $\ulin V(\tau)$  hits $g_\tau(B_{\FF}(u,r))$. Suppose $\tau<\tau_u^*$.   By Koebe distortion theorem,  there are $r_1,r_2=D_u(\tau) r(1+O(r/R))$ such that
$$  B_{\FF}(U(\tau),r_1)\subset g_\tau(B_{\FF}(u,r))\subset B_{\FF}(U(\tau),r_2).$$
 By Koebe $1/4$ theorem, $|V_m(\tau)-U(\tau)|\ge D_u(\tau) R/4$. Note that the interval $[v_m\wedge a_{K_\tau},v_1\wedge a_{K_\tau}]$ can be disconnected in $\HH\sem K_\tau$ from both $(-\infty,u]$ and $\pa_{\HH}^+ K_\tau$ by  $A_{\HH}(u,R,|v_m-u|)$. Thus, by Proposition \ref{extrem-prop}, $\frac{|V_1(\tau)-V_m(\tau)|}{|W(\tau)-V_m(\tau)|}$ and $\frac{|V_1(\tau)-V_m(\tau)|}{|V_m(\tau)-U(\tau)|}$ both have the order $O(\frac{R}{|v_m-u|})$.
Since $M(\tau)=D_u(\tau)^\alpha G(W(\tau),\ulin V(\tau);U(\tau))$,
by (\ref{v-close'})   and the above estimates, the conditional probability that $\eta^\tau$ hits $g_\tau(B_{\FF}(u,r))$ equals
$$C_{\FF}M(\tau)  r^{\alpha}+\Big(\frac{r}{R}\Big)^\alpha\Big( O\Big(\frac{r}{R}\Big)^{\beta_\Sigma}+ O\Big(\frac{R}{|v_m-u|}\Big)^{\beta_*}\Big).$$
Setting $R$ such that $R^{\beta_\Sigma+\beta_*}=r^{\beta_\Sigma} |v_m-u|^{\beta_*}$, and letting $\beta=\frac{\beta_*\beta_\Sigma}{\beta_*+\beta_\Sigma}$,  we see that the above conditional probability equals
$C_{\FF}M(\tau)  r^{\alpha}+ (\frac{r}{R} )^\alpha O  (\frac{r}{|v_m-u|} )^\beta$. Thus,
$$\PP[\eta\cap B_{\FF}(u,r)\ne\emptyset]=
\EE [\ind_{\{\tau<\tau_u^*\}}C_{\FF}M(\tau) r^{\alpha}] + \PP[\tau<\tau_u^*]  \Big(\frac{r}{R}\Big)^\alpha O \Big(\frac{r}{|v_m-u|}\Big)^\beta $$
$$=C_{\FF} G(w,\ulin v;u) r^\alpha+ \Big(\frac{|v_m-u|}{|w-u|}\Big)^{\alpha_*}\cdot  O\Big(\frac{r}{|v_m-u|}\Big)^{\alpha+\beta},$$
where we used $\EE[\ind_{\{\tau<\tau_u^*\}} M(\tau)]=G(w,\ulin v;u)$ and $\PP[\tau<\tau_u^*]\lesssim (\frac{|v_m-u|}{|w-u|} )^{\alpha_*} (\frac R{|v_m-u|})^\alpha$ by (\ref{est-u}).
\end{proof}

Combining Theorem \ref{main-thm1} with the lower bound of $G(w,\ulin v;u)$ in (\ref{G-upper}), we get the following lower bound of the hitting probability.

\begin{Corollary}
There is a constant $c\in(0,1)$  such that when
${r}/{L_1}\le c ( {L_1}/({L_0+L_1}))^{\frac{\alpha^*-\alpha_*}\beta}$,
  we have
  $\PP[ \tau^u_{r;\FF}<\infty]\ge C_{\FF} G(w,\ulin v;u) r^\alpha/2$.
  \label{Cor-lower-1pt}
\end{Corollary}

Define $G_{t,\FF}(u)$ for $t\in[0,\infty)$ by
\BGE
G_{t;\FF}(u)=\left\{\begin{array}{ll}
C_{\FF} D_u(t)^\alpha G(W(t),V(t);U(t)), & \mbox{if } t<\tau_u^*;\\
0, &\mbox{if }t\ge \tau_u^*.
\end{array}
\right.
\label{Gt}
\EDE
By Koebe $1/4$ theorem and (\ref{G-upper}),
\BGE G_{t;\FF}(u)\lesssim \dist(u, \eta[0,t])^{-\alpha}.\label{Gt-upper}\EDE

The following lemmas will be used in the next section.

\begin{Lemma}
  For any
  $\F$-stopping time $\tau$ and $R>r\in (0,|v_m-u|)$, as $r/R\to 0^+$,
  \BGE \PP[ \tau^u_{r;\FF}<\infty|\F_\tau, \tau<\tau^u_R ]=   G_{\tau;\FF}(u) r^\alpha  +O(r/R)^{\alpha+\beta}.\label{Green-r-R}\EDE
  \label{Cor-main-thm1}
\end{Lemma}
\begin{proof}
If $\tau^*_u\le \tau<\tau^u_R$, the part of $\eta$ after $\tau$ is disconnected from $u$ by $\eta[0,\tau^*_u]$, which has distance greater than $R$ from $u$. Thus, $\PP[\ \tau^u_{r;\FF}<\infty|\F_\tau, \tau^*_u\le \tau<\tau^u_R ]=0$. So it suffices to prove (\ref{Green-r-R}) with $\tau<\tau^*_u\wedge \tau^u_R$ in place of $\tau<\tau^u_R$.
   By DMP, $\PP[  \tau^u_{r;\FF}<\infty|\F_\tau, \tau<\tau^*_u\wedge \tau^u_R ]$ equals the probability that an SLE$_\kappa(\ulin\rho)$ curve $\eta^\tau$ started from $W(\tau)$ with force points $\ulin V(\tau)$ visits $g_\tau(B_{\FF}(u,r))$. Assume that $\tau<\tau^*_u\wedge \tau^u_R$.   By Koebe's $1/4$ theorem and distortion theorem, $|U(\tau)-V_m(\tau)|\ge D_u(\tau) R/4$, and  there are $r_1,r_2=D_u(\tau) r(1+O(r/R))$ such that
   $ B_{\FF}(U(\tau),r_1)\subset g_\tau(B_{\FF}(u,r))\subset B_{\FF}(U(\tau),r_2)$. So we get the estimate of the conditional probability by applying Theorem \ref{main-thm1} and (\ref{Gt-upper}) .
 \end{proof}

\begin{Lemma}
  Let $t_1<t_2\in [0,\tau_u^*)$. Suppose $I$ is a crosscut of $\HH\sem K_{t_1}$ connecting a point on the real interval $(u,v_m\wedge a_{K_{t_1}}]$ with a point on $\pa_{\HH}^+ K_{t_1}$ such that $\eta[t_1,t_2]$ lies in the closure of the bounded connected component of $(\HH\sem K_{t_1})\sem I$, and $d_{\HH\sem K_{t_1}}(I,(-\infty,u])>\ln(1440)/\pi$. Then
  \BGE|\ln(G_{t_2;\FF}(u))-\ln(G_{t_1;\FF}(u))|\lesssim \exp({-\pi d_{\HH\sem K_{t_1}}(I,(-\infty,u])}).\label{Gt1t2-eqn}\EDE
  \label{Gt1t2}
\end{Lemma}
\begin{proof}
Let $\til I=g_{t_1}(I)$. Then $\til I$ is a crosscut in $\HH$ connecting $(U(t_1),V_m(t_1)]$ with $[W(t_1),\infty)$, and so it encloses $W(t_1),V_1(t_1),\dots,V_m(t_1)$. Moreover, since $\eta[t_1,t_2]$ does not intersect the unbounded component of $(\HH\sem K_{t_1})\sem I$, we see that $\til I$ also encloses $\til K:=g_{t_1}(K_{t_2}\sem K_{t_1})$. Let $X=d_{\HH\sem K_{t_1}}(I,(-\infty,u])$. By conformal invariance of extremal length, $d_{\HH}(\til I,(-\infty,g_t(u)])=X$. Since  $X\ge \ln(1440)/\pi$, by Proposition \ref{extrem-prop},
\BGE  {\rad_{W(t_1)}([V_m(t_1),W(t_1)]\cup \til K)} \le
 {\rad_{W(t_1)}(\til I)}\le 144 e^{-\pi X}{|W(t_1)-U(t_1)|}.\label{rad}\EDE

In order to prove (\ref{Gt1t2-eqn}), it suffices to prove that (\ref{Gt1t2-eqn}) holds with $D_u({t_\iota})$, $|W(t_\iota)-U(t_\iota)|$, and $|V_j(t_\iota)-U(t_\iota)|$, $1\le j\le m$, respectively, in place of $G_{t_\iota;\FF}(u)$, $\iota=1,2$. Since $\til K=g_{t_1}(K_{t_2}\sem K_{t_1})$, we have $g_{t_2}=g_{\til K}\circ g_{t_1}$. So $U(t_2)=g_{\til K}(U(t_1))$ and $D_u(t_2)=g_{\til K}'(U(t_1))D_u(t_1)$.

Note that $g_{t_1}(K_{t_1+s}\sem K_{t_1})$ and $g_{t_1+s}\circ g_{t_1}^{-1}$, $s\ge 0$, are chordal Loewner hulls driven by $W(t_1+s)$, $s\ge 0$, and $\til K$ is the hull in the process at the time $t_2-t_1$. Let $r= {\rad_{W(t_1)}(\til I)} $ and $R=|U(t_1)-W(t_1)|$. Then $\frac rR\le 144 e^{-\pi X}<\frac 1{10}$ by  (\ref{rad}) and that $X>\ln(1440)/\pi$.
By Proposition \ref{basic-chordal} (i,ii) and (\ref{rad}), $|W(t_2)-W(t_1)|\le 2 r$ and $|U(t_2)-U(t_1)|\le 3r$. Since $|U(t_1)-W(t_1)|=R>10 r$, these estimates imply $|\frac{|U(t_2)-W(t_2)|}{|U(t_1)-W(t_1)|}-1|\le \frac {5r}R\le \frac 12$, and so
$$\Big|\ln\Big(\frac{|U({t_2})-W(t_2)|}{|U({t_1})-W(t_1)|}\Big)\Big|\le 2  \Big|\frac{|U({t_2})-W(t_2)|}{|U({t_1})-W(t_1)|}-1\Big|\le \frac{10 r}R\le 1440 e^{-\pi X}.$$
This means that (\ref{Gt1t2-eqn}) holds with $|U({t_\iota})-W(t_\iota)|$ in place of $G_{t_\iota;\FF}(u)$, $\iota=1,2$.
Fix $j\in\{1,\dots,m\}$. By (\ref{rad}), we have $|V_j(t_1)-W(t_1)|\le 2r$. So $|U({t_1})-V_j(t_1)|\ge R-2r\ge \frac 4 5 R\ge 8r$. We have $V_j(t_2)=\lim_{x\uparrow V_j(t_1)} g_{\til K}(x)$. By Proposition \ref{basic-chordal} (i), $|V_j(t_1)-V_j(t_2)|\le 3r$. Thus,
$|\frac{|U({t_2})-V_j(t_2)|}{|U({t_1})-V_j(t_1)|}-1|\le \frac {6r}{4 R/5}\le \frac 34$, which implies
$$\Big|\ln\Big(\frac{|U({t_2})-V_j(t_2)|}{|U({t_1})-V_j(t_1)|}\Big)\Big|\le 4  \Big|\frac{|U({t_2})-V_j(t_2)|}{|U({t_1})-V_j(t_1)|}-1\Big|\le \frac{30 r}R\le 4320 e^{-\pi X}.$$
This means that (\ref{Gt1t2-eqn}) holds with $|U({t_\iota})-V_j(t_\iota)|$ in place of $G_{t_\iota;\FF}(u)$, $\iota=1,2$. Finally,
by Proposition \ref{basic-chordal} (iii), $|\ln ( {D_u(t_2)}/{D_u(t_1)})|\le 5 (r/R)^2 \le 72 e^{-\pi X}$. So (\ref{Gt1t2-eqn}) also  holds with $D_u(t_\iota)$ in place of $G_{t_\iota;\FF}(u)$, $\iota=1,2$.  The proof is now complete.
\end{proof}

\section{Two-point Green's Function} \label{Section 5}
 We use the setup of the previous section. Now we fix $u_1>u_2\in (-\infty,v_{  m})$, and let $L_0=|w-v|$, $L_1=|v-u_1|$, $L_2=|u_1-u_2|$, and $ L_\wedge =L_1\wedge L_2$.
We are going to prove the following theorem.

\begin{Theorem} [Two-point Green's function]
There is a number $G_{\FF}(w,\ulin v;u_1,u_2)\in (0,\infty)$ (depending also on $\kappa,\ulin\rho$), which  satisfies
\BGE G(w,\ulin v;u_1) L_2^{-\alpha} \lesssim G_{\FF}(w,\ulin v;u_1,u_2)\lesssim \Big(\frac{L_1}{L_0+L_1}\Big)^{\alpha_*} L_1^{-\alpha}L_2^{-\alpha}\label{G(u1,u2)-upper}\EDE
such that, for some constants $\zeta_1,\zeta_2\in(0,1)$, as $\frac{r_1}{L_\wedge},\frac{r_2}{L_2}\to 0^+$,
$$ \PP[\tau^{u_j}_{r_j;\FF}<\infty, j=1,2]=G_{\FF}(w,\ulin v; u_1,u_2) r_1^\alpha r_2^\alpha  $$
 \BGE+ \Big(\frac{L_1}{L_0+L_1}\Big)^{{\alpha_*}}  \Big(\frac{r_1}{L_1}\Big)^\alpha \Big(\frac{r_2}{L_2}\Big)^\alpha \Big[O \Big(\frac{r_1}{L_\wedge}\Big)^{\zeta_1}+O\Big(\frac{r_2}{L_2}\Big)^{\zeta_2}\Big]. \label{G(u1,u2)-est}\EDE
In particular, we have
\BGE G_{\FF}(w,\ulin v;u_1,u_2)=\lim_{r_1,r_2\to 0^+} r_1^{-\alpha}r_2^{-\alpha} \PP[\eta\cap B_{\FF}(u_j,r_j)\ne \emptyset,j=1,2].\label{Green-2pt}\EDE
\label{main-thm2}
\end{Theorem}

We now write  $\tau^j_{r;\FF}$ for $\tau^{u_j}_{r;\FF}$ and $\tau^j_r$ for $\tau^{u_j}_r=\tau^j_{r;\C}$.  Let $U_j(t)=g_t(u_j)$ and $D_j(t)=g_t'(u_j)$.

\begin{Lemma}
 Let $r_1\in (0,L_1)$ and $r_2>0$.
  Then
  \BGE \PP[\tau^j_{r_j}<\infty,j=1,2]\lesssim  ( {L_1}/({L_0+L_1}) )^{{\alpha_*}}  ( {r_1}/{L_1} )^{\alpha} ( {r_2}/{L_2} )^{\alpha}.\label{transform-2}\EDE
\label{two-pt-lem-u1u2}
\end{Lemma}

\begin{Remark} This lemma is another two-point estimate. Its proof is more complicated than that of Lemma \ref{two-pt-lem}.
 The main reason is that the boundary exponent $\alpha_*$ in Lemma \ref{boundary-estimate} agrees with the exponent $\alpha_*$ for $v_m$ in Lemma \ref{two-pt-lem}, but may be strictly less than the exponent $\alpha$ for $u_1$ in this lemma. The idea of the proof, originated in \cite{LW}, is to exponentially bound the probability that $\eta$ travels back and forth between small discs centered at $u_1$ and $u_2$.
\end{Remark}

\begin{proof}
  If $r_2\ge L_2/6$, then we get (\ref{transform-2}) by applying (\ref{est-u}) to bound $\PP[\tau^1_{r_1}<\infty]$. Suppose $r_2<L_2/6$ and $r_1\ge L_\wedge /6$. From $\frac 1{L_1+L_2}\asymp \frac{L_\wedge}{L_1 L_2}$, we get $\frac{r_2}{L_1+L_2} \lesssim \frac{r_1r_2}{L_1L_2}$.
  There are two cases. Case 1. $L_1\ge L_0$. In this case, since $\frac{L_1}{L_0+L_1}\gtrsim 1$, we get (\ref{transform-2}) by applying (\ref{est-u*}) to bound $\PP[\tau^2_{r_2}<\infty]$.
 Case 2. $L_1\le L_0$. In this case, $ (\frac{L_1+r_1}{L_0} )^{\alpha_*}\lesssim   (\frac{L_1}{L_0+L_1} )^{\alpha_*}$, and we get (\ref{transform-2}) by applying  (\ref{two-pt-ineq}) to bound $\PP[\tau^{v_m}_{L_1+r_1}<\infty,\tau^{u_2}_{r_2}<\infty]$.
  So we always have (\ref{transform-2}) when $r_2<L_2/6$ and $r_1\ge L_\wedge/6$. From now on, we assume that $r_1 <L_\wedge /6$ and $r_2<L_2/6$.

Let $L_1'=L_\wedge $ and $L_2'=L_2$. WLOG, we assume $r_j= e^{-N_j}L_j'/2$, $j=1,2$, for some $N_1,N_2\in\N$. With a slight abuse of notation, we write $\tau^j_n$ for $\tau^j_{L_j' e^{-n}/2}$.
It suffices to  prove
\BGE \PP[\tau^j_{N_j}<\infty,j=1,2]\lesssim  ( {L_1}/({L_0+L_1}) )^{\alpha_*} ( {L_\wedge}/{L_1} )^{\alpha} e^{-\alpha(N_1+N_2)}.\label{weaker}\EDE

By (\ref{est-u}),
\BGE \PP[\tau^1_n<\infty]\lesssim ( {L_1}/({L_0+L_1}) )^{\alpha_*} ( {L_\wedge}/{L_1} )^{\alpha} e^{-\alpha n} ;\label{tau1n'}\EDE
\BGE \PP[\tau^2_n<\infty]\lesssim ( ({L_1+L_2})/({L_0+L_1+L_2}))^{\alpha_*} e^{-\alpha n}.\label{tau2n'}\EDE
By (\ref{interior-estimate-formula}),
\BGE \PP[\tau^j_{n_2}<\infty|\F_{\tau^j_{n_1}},\tau^j_{n_1}<\infty]\lesssim e^{-\alpha(n_2-n_1)},\quad n_2>n_1.\label{tau1-tau2}\EDE
Note that if $\tau^1_{n_1-1}<\tau^2_{n_2}<\tau^1_{n_1}$, at the time $\tau^2_{n_2}$, $\{|z-u_1|=L_1' e^{1-n_1}/2\}\cap (\HH\sem K_{\tau^2_{n_2}})$ has a connected component $I$, which disconnects $v_m\wedge a_{K_{\tau^2_{n_2}}}$ from $\infty$; and in order for $\eta$ to visit $\{|z-u_1|\le L_1' e^{-n}/2\}$, it must visit $I$ after $\tau^2_{n_2}$. Since any curve in $\HH\sem K_{\tau^2_{n_2}}$ connecting $I$ and $\pa_{\HH}^+ K_{\tau^2_{n_2}}$ must cross   $A_{\HH}(u_1,L_\wedge e^{1-n_1}/2,L_2/2)$ and $A_{\HH}(u_2,   L_2 e^{-n_2}/2,L_2/2)$, we have
$$d_{\HH\sem K_{\tau^2_{n_2}}}(I,\pa_{\HH}^+ K_{\tau^2_{n_2}})\ge \frac 1\pi \ln\Big(\frac{L_2/2}{L_\wedge e^{1-n_1}/2}\Big) +\frac 1\pi \ln\Big(\frac{L_2/2}{L_2 e^{-n_2}/2}\Big).$$
So by (\ref{boundary-estimate-formula}),
\BGE \PP[\tau^1_{n_1}<\infty|\F_{\tau^2_{n_2}},\tau^1_{n_1-1}<\tau^2_{n_2}<\tau^1_{n_1}]\lesssim  ( {L_\wedge}/{L_2} )^{\alpha_*} e^{-\alpha_* (n_1+n_2)}.\label{tau2-tau1}\EDE

Now we prove (\ref{weaker}) using (\ref{tau1n'}-\ref{tau2-tau1}).
We break the event $E:=\{\tau^j_{N_j}<\infty,j=1,2\}$ into a disjoint union of events according to the order of those hitting times $\tau^j_{n}$, $n\le N_j$, $j=1,2$. For  $j=1,2$, $\tau^j_n$ is increasing in $n$. But the order between any $\tau^1_{n_1}$ and $\tau^2_{n_2}$ is uncertain.
For  $\ii,\oo\in\{1,2\}$ and $s\in\N$, let $\Xi^{(\ii,\oo)}_s$ be the collection of $(S_1,S_2)$ such that
\begin{itemize}
  \item $S_j\subset \{n\in\N:n\le N_j\}$, $j=1,2$;
  \item $|S_{\ii}|=s-|\ii-\oo|$ and $|S_{3-\ii}|=s-1$.
\end{itemize}

For $\ii,\oo\in\{1,2\}$ and $\xi=(S_1,S_2)\in\Xi^{(\ii,\oo)}_s$, we order the elements in $S_j$ as $n^{(j)}_1<\cdots <n^{(j)}_{s_j}$, where $s_j=|S_j|$, $j=1,2$.  We also set $n^{(j)}_0=0$ and $n^{(j)}_{s_j+1}=N_j+1$. Define
$$E_{(\ii,\oo);\xi}=\Big(\bigcap_{k=1}^{s} \Big \{\tau^{\ii}_{n^{(\ii)}_{k}-1}< \tau^{3-\ii}_{n^{(3-\ii)}_{k-1}}\Big\}\Big)\bigcap\Big( \bigcap_{k=1}^{s-|\ii-\oo|}\Big \{\tau^{3-\ii}_{n^{(3-\ii)}_k-1}<\tau^{\ii}_{n^{(\ii)}_k}\Big\}\Big)\bigcap  \{\tau^{\oo}_{N_{\oo}}<\infty\}$$
$$=\{\tau^{\ii}_{n^{(\ii)}_0}\le \tau^{\ii}_{n^{(\ii)}_1-1}<\tau^{3-\ii}_{n^{(3-\ii)}_0}\le \tau^{3-\ii}_{n^{(3-\ii)}_1-1}<\tau^{\ii}_{n^{(\ii)}_1}\le \cdots \le \tau^{3-\oo}_{N_{3-\oo}}<\tau^{\oo}_{n^{\oo}_{s_{\oo}}}\le \tau^{\oo}_{N_{\oo}}<\infty\}. $$

The symbols $\ii$ and $\oo$ respectively stand for ``in'' and ``out'': on the event $E_{(\ii,\oo);\xi}$, $\tau^{\ii}_0$ happens before $\tau^{3-\ii}_0$, and $\tau^{\oo}_{N_{\oo}}$ happens after $\tau^{3-\oo}_{N_{3-\oo}}$.
It is straightforward to check that
\BGE E=\bigcup_{\ii\in\{1,2\}} \bigcup_{\oo\in\{1,2\}}  \bigcup_{s\in\N} \bigcup_{\xi\in \Xi^{(\ii,\oo)}_s} E_{(\ii,\oo);\xi} .\label{unionE}\EDE

Using (\ref{tau1n'}-\ref{tau2-tau1} we find that, for some constant $C\ge 1$,
 \BGE \PP[E_{\ii,\oo;\xi}]\le \Big(\frac{L_1}{L_0+L_1}\Big)^{\alpha_*} \Big(\frac{L_\wedge}{L_1}\Big)^{\alpha} \prod_{j\in\{1,2\}} \prod_{k=1}^{s_j+1} C e^{-\alpha(n^{(j)}_k-n^{(j)}_{k-1} )} \prod_{k=1}^{s_2+|\oo-2|} C e^{-\alpha_*(n^{(2)}_k+n^{(1)}_{k-|\ii-1|} )}\label{PE-pre}\EDE
\BGE \le C^{3s+3} \Big(\frac{L_1}{L_0+L_1}\Big)^{\alpha_*} \Big(\frac{L_\wedge}{L_1}\Big)^{\alpha}   e^{-\alpha N_1} e^{-\alpha N_2} \prod_{n\in S_1} e^{-\alpha_* n} \prod_{n\in S_2} e^{-\alpha_* n}.\label{PE}\EDE
To see that (\ref{PE-pre}) holds, consider two cases. Case 1. $\ii=1$. By (\ref{tau1n'},\ref{tau1-tau2},\ref{tau2-tau1}), $$\PP[\tau^1_{n^{(1)}_1}<\infty]\le C\Big (\frac{L_1}{L_0+L_1}\Big)^{\alpha_*} \Big (\frac{L_\wedge }{L_1} \Big)^{\alpha} e^{-n_1^{(1)}\alpha};$$ \BGE \PP[\tau^j_{n^{(j)}_{k}-1}<\infty|\F_{\tau^j_{n^{(j)}_{k-1}}},\tau^j_{n^{(j)}_{k-1}}<\infty]\le  C e^{-\alpha(n^{(j)}_k-n^{(j)}_{k-1})}; \label{same1}\EDE \BGE \PP[\tau^1_{n^{(1)}_{k-|\ii-1|}}<\infty | \F_{\tau^2_{n^{(2)}_{k}-1}},\tau^2_{n^{(2)}_{k}-1}<\infty]\le C \Big (\frac {L_\wedge}{L_2}\Big)^{\alpha_*}
e^{-\alpha_*(n^{(2)}_k+n^{(1)}_{k-|\ii-1|})}.\label{same2}\EDE  Note that the product of upper bounds is bounded by the RHS of (\ref{PE-pre})  since $\frac{L_\wedge}{ L_2}\le 1$. Also note that $s_2+|\oo-2|=s-|\ii-\oo|$.

Case 2. $\ii=2$. Then $\PP[\tau^2_{n^{(2)}_1}<\infty]\le C(\frac{L_1+L_2}{L_0+L_1+L_2})^{\alpha_*} (\frac{L_2}{L_1+L_2} )^\alpha e^{-n^{(2)}_1\alpha }$ by \ref{tau2n'}. We still have (\ref{same1},\ref{same2}).   The product of these upper bounds is also bounded by the RHS of (\ref{PE-pre})  because (i) $ \frac{L_2}{L_1+L_2} \asymp \frac{L_\wedge}{L_1} $; (ii) there is at least one factor $(\frac{L_\wedge}{L_2})^{\alpha_*}$ coming from (\ref{same2}); (iii) $ \frac{L_1+L_2}{L_0+L_1+L_2} \cdot \frac {L_\wedge}{L_2} \lesssim \frac{L_1}{L_0+L_1} $; and (iv) $s_2+|o-2|=s$.

To get (\ref{PE}), we note that  $s_1,s_2\in \{s,s-1\}$, $\sum_{k=1}^{s_j+1} (n^{(j)}_k-n^{(j)}_{k-1})=N_j+1$,   $j=1,2$, and $n^{(2)}_k$ and $n^{(1)}_{k-|\ii-1|}$, $1\le k\le s_2+|\oo-2|$, exhaust the sets $S_1$ and $S_2$.

Combining (\ref{unionE},\ref{PE}), we get (\ref{weaker}) because
$$ \sum_{\ii,\oo\in\{1,2\}} \sum_{s=1}^\infty \sum_{\xi=(S_1,S_2)\in \Xi^{(\ii,\oo)}_s} C^{3s+3}  \prod_{k=1}^{s_1} e^{-\alpha_* n^{(1)}_k} \prod_{k=1}^{s_2} e^{-\alpha_* n^{(2)}_k}$$
$$\le   \sum_{\ii,\oo\in\{1,2\}} \sum_{s=1}^\infty C^{3s+3} \sum_{n^{(1)}_1=1}^\infty e^{-\alpha_* n^{(1)}_1}\cdots \sum_{n^{(1)}_{s_1}=s_1}^\infty e^{-\alpha_* n^{(1)}_{s_1}} \sum_{n^{(2)}_1=1}^\infty e^{-\alpha_* n^{(2)}_1}\cdots \sum_{n^{(2)}_{s_2}=s_2}^\infty e^{-\alpha_* n^{(2)}_{s_2}}    $$
$$= \sum_{\ii,\oo\in\{1,2\}} \sum_{s=1}^\infty C^{3s+3} \prod_{k=1}^{s_1} \frac{e^{-k\alpha_*}}{1-e^{-\alpha_*}} \prod_{k=1}^{s_2} \frac{e^{-k\alpha_*}}{1-e^{-\alpha_*}}\le \sum_{\ii,\oo\in\{1,2\}} \sum_{s=1}^\infty \frac{C^{3s+3} e^{-\alpha_*(s_1^2+s_2^2)/2} }{(1-e^{-\alpha_*})^{s_1+s_2}}=: C_*.$$
Here in the second line we used the fact that $n^{(j)}_k\ge k$. We have $C_*<\infty$ because $s_1$ and $s_2$ are determined by  $s_{\ii}=s-|\ii-\oo|$ and $s_{3-\ii}=s-1$.
\end{proof}

The following lemma improves the two-point estimate in the case that $\eta$ first visits a small disc centered at $u_2$ and then visits a small disc centered at $u_1$.

\begin{Lemma}
For $r_1\in(0,L_\wedge /6)$, $R_2>r_2>0$, we have
  \BGE \PP[\tau^2_{R_2}<\tau^1_{r_1;\FF}<\infty]\lesssim \Big(\frac{L_1}{L_0+L_1}\Big)^{\alpha_*}  \Big(\frac{r_1}{L_1}\Big)^{\alpha} \Big(\frac{R_2}{L_2}\Big)^{\alpha+\alpha_*};\label{21-ineq}\EDE
    \BGE \PP[\tau^2_{R_2}<\tau^1_{r_1;\FF}<\tau^2_{r_2}<\infty]\lesssim \Big(\frac{L_1}{L_0+L_1}\Big)^{\alpha_*} \Big(\frac{r_1}{L_1}\Big)^{\alpha} \Big(\frac{r_2}{L_2}\Big)^{\alpha} \Big(\frac{R_2}{L_2}\Big)^{\alpha_*}.\label{212-ineq}\EDE
\label{two-pt-cor}
\end{Lemma}
\begin{proof}
  If $R_2\ge L_2/6$, these formulas clearly follow from (\ref{est-u},\ref{transform-2}). Suppose now $R_2< L_2/6$. If $r_2\ge R_2/6$, then (\ref{212-ineq}) follows from (\ref{21-ineq}). Suppose now $r_2< R_2/6$. Let $E_{r_1}$ and $E_{r_1,r_2}$ be respectively the events in (\ref{21-ineq}) and (\ref{212-ineq}), and let    $\tau=\tau^2_{R_2}$.

Let $E_0=\{\tau<\tau^1_{L_\wedge/3}\}$ and $E_n=\{\tau^1_{e^{1-n} L_\wedge /3}<\tau<\tau^1_{e^{-n}L_\wedge/3 }\}$, $n\in\N$. By (\ref{est-u},\ref{transform-2}),
  \BGE \PP[E_0]\lesssim \Big(\frac{L_1+L_2}{L_0+L_1+L_2}\Big)^{\alpha_*} \Big(\frac{R_2}{L_1+L_2}\Big)^{\alpha}\lesssim  \Big(\frac{L_1}{L_0+L_1}\Big)^{\alpha_*}  \Big(\frac{L_2}{L_\wedge}\Big)^{\alpha_*}   \Big(\frac{L_\wedge }{L_1}\Big)^{\alpha} \Big(\frac{R_2}{L_2}\Big)^{\alpha} ;\label{tauR20}\EDE
  \BGE \PP[E_n]\lesssim \Big(\frac{L_1}{L_0+L_1}\Big)^{\alpha_*} \Big(\frac{e^{-n}L_\wedge}{L_1}\Big)^{\alpha} \Big(\frac{R_2}{L_2}\Big)^{\alpha},\quad n\ge 1. \label{tauR2n}\EDE
  Here we used that $ {L_1L_2}\asymp (L_1+L_2)L_\wedge$  in the second inequality of (\ref{tauR20}).

By DMP, $\PP[E_{r_1}|\F_\tau,\tau<\infty]$ equals the probability that an SLE$_\kappa(\ulin\rho)$ curve $\eta^\tau$ in $\HH$ started from $W(\tau)$ with force points $\ulin V(\tau)$ visits $g_\tau(B_{\FF}(u_1,r_1))$, and  $\PP[E_{r_1,r_2}|\F_\tau,\tau<\infty]$ equals the probability that such $\eta^\tau$ visits $g_\tau(B_{\FF}(u_1,r_1))$ and $g_\tau(B_{\C}(u_2,r_2))$ in order. Let $L_0'=|W(\tau)-V_m(\tau)|$, $L_1'=|V_m(\tau)-U_1(\tau)|$, $L_2'=|U_1(\tau)-U_2(\tau)|$, and $r_j'= \rad_{U_2(\tau)}(g_\tau(B_{\C}(u_j,r_j)))$, $j=1,2$.    Since $r_2\le R_2/6$, by Koebe's $1/4$ theorem and distortion theorem, $L_2'\ge D_2(\tau) R_2/4$ and  $r_2'\le \frac {36}{25} D_2(\tau) r_2$, which imply
\BGE r_2'/L_2'\le 6(r_2/R_2)<1.\label{L2'r2'}\EDE

Suppose some $E_n$, $n\ge 0$, happens. Consider two cases. Case 1. $r_1< e^{-n} L_\wedge /6$. By Koebe's $1/4$ theorem and distortion theorem,
$L_1'\ge D_1(\tau) e^{-n} L_\wedge /4$, $r_1'\le   \frac {36}{25} D_1(\tau) r_1$, and so
\BGE r_1'/L_1'\le 6  r_1/(e^{-n}L_\wedge)<1 .\label{L1'r1'0}\EDE

Note that any curve in $\HH\sem K_\tau$ connecting $[u_1,v_m\wedge a_{K_\tau}]$ and $\pa_{\HH}^+ K_\tau$ crosses
$A_{\HH}(u_2,R_2,L_2/2)$, and when $e^{-n}L_\wedge<L_2/2$, also crosses
 $ A_{\HH}(u_1,e^{-n}L_\wedge,L_2/2)$. So by Proposition \ref{extrem-prop} and  the comparison principle and conformal invariance of extremal length,
\BGE  {L_1'}/({L_0'+L_1'})\lesssim  ({e^{-n}L_\wedge} /{L_2}) \cdot ({R_2} /{L_2}) .\label{L01/L1}\EDE

Combining (\ref{L2'r2'},\ref{L1'r1'0},\ref{L01/L1}) with (\ref{est-u},\ref{transform-2})   we get
\BGE \EE[E_{r_1}|\F_\tau,E_n]\lesssim \Big(\frac{e^{-n}L_\wedge} {L_2}\Big)^{\alpha_*} \Big(\frac{R_2} {L_2}\Big)^{\alpha_*} \Big(\frac{r_1}{e^{-n} L_\wedge }\Big)^\alpha;\label{condition-Er1}\EDE
\BGE \EE[E_{r_1,r_2}|\F_\tau,E_n]\lesssim \Big(\frac{e^{-n}L_\wedge} {L_2}\Big)^{\alpha_*} \Big(\frac{R_2} {L_2}\Big)^{\alpha_*} \Big(\frac{r_1}{e^{-n} L_\wedge }\Big)^\alpha\Big(\frac{r_2}{R_2}\Big)^\alpha.
\label{condition-Er1r2}\EDE
Combining these inequalities with (\ref{tauR20},\ref{tauR2n}) we get that, for all $n\ge 0$,
\BGE \PP[E_{r_1}\cap E_n] \lesssim e^{-n\alpha_*} \Big(\frac{L_1}{L_0+L_1}\Big)^{\alpha_*} \Big(\frac{r_1}{L_1}\Big)^{\alpha}  \Big(\frac{R_2}{L_2}\Big)^{\alpha+\alpha_*} ;\label{Er1En}\EDE
\BGE \PP[E_{r_1,r_2}\cap E_n] \lesssim e^{-n\alpha_*} \Big(\frac{L_1}{L_0+L_1}\Big)^{\alpha_*} \Big(\frac{r_1}{L_1}\Big)^{\alpha} \Big(\frac{r_2}{L_2}\Big)^{\alpha}  \Big(\frac{R_2}{L_2}\Big)^{\alpha_*} .\label{Er1r2En}\EDE

 Case 2. $r_1\ge  e^{-n} L_\wedge /6$.
 Then $n\ge 1$ since it is assumed that $r_1< L_\wedge /6$.
 Suppose  $E_n\cap \{\tau<\tau^1_{r_1;\FF}\}$ happens.
 Using Proposition \ref{extrem-prop}  and
  two separation semi-annuli $A_{\HH}(u_1,  e^{1-n} L_\wedge/3, L_2/2)$ and $A_{\HH}(u_2,R_2,L_2/2)$,  we find that
\BGE 1\wedge \frac{\rad_{V_m(\tau)}(g_\tau(B_{\FF}(u_1,r_1)\cap\lin\HH))}{L_0'}\lesssim \frac{ e^{-n}L_\wedge}{L_2}\cdot \frac{R_2}{L_2}
.\label{1wedge vm}\EDE
Combining (\ref{1wedge vm}) with (\ref{est-vm}), we find that $ \EE[E_{r_1}|\F_\tau,E_n]\lesssim  e^{-n \alpha_*}(\frac{L_\wedge}{L_2})^{\alpha_*} (\frac{R_2} {L_2} )^{\alpha_*}  $, which together with  (\ref{tauR2n}) and that $e^{-n}L_\wedge \le 6r_1$ implies (\ref{Er1En}).   Similarly, combining (\ref{1wedge vm}) with (\ref{L2'r2'}) and (\ref{two-pt-ineq}), we find that  $ \EE[E_{r_1,r_2}|\F_\tau,E_n]\lesssim   e^{-n \alpha_*}   (\frac{R_2}{L_2})^{\alpha_*}(\frac{r_2}{R_2})^{\alpha}$, which together with  (\ref{tauR2n}) and that $e^{-n}L_\wedge \le 6 r_1$ implies  (\ref{Er1r2En}). So (\ref{Er1En}) and (\ref{Er1r2En})  hold in both cases. Finally, summing up (\ref{Er1En}) and (\ref{Er1r2En})  over $n\in\N\cup\{0\}$, we get (\ref{21-ineq}) and (\ref{212-ineq}).
\end{proof}

Combining Lemma \ref{two-pt-cor} with Corollary \ref{Cor-lower-1pt}, we get the following corollary.

\begin{Corollary}
    There is a constant $c\in(0,1)$ such that when
  \BGE \frac{r_1}{L_\wedge}\le c \Big(\frac{L_1}{L_0+L_1}\Big)^{\frac{\alpha^*-\alpha_*}\beta}\quad \mbox{and}\quad \frac{r_2}{L_2}\le c \Big(\frac{L_1}{L_0+L_1}\Big)^{\frac{\alpha^*-\alpha_*}{\alpha+\alpha_*}}, \label{2pt-condition}\EDE
  we have
  \BGE \PP[\tau^j_{r_j;\FF}<\infty,j=1,2]\gtrsim G(w,\ulin v;u_1) L_2^{-\alpha} r_1^\alpha r_2^\alpha.\label{lower-prob}\EDE
  \label{Cor-lower-2pt}
\end{Corollary}
\begin{proof}
  By  Corollary \ref{Cor-lower-1pt}  and (\ref{21-ineq},\ref{G-upper}), there is a constant $c_1\in(0,1)$  such that when (\ref{2pt-condition}) holds with $c_1$ in place of $c$, we have
  \BGE\PP  [\tau^1_{r_1;\FF}<\tau^2_{r_2}] \gtrsim   G(w,\ulin v;u_1) r_1^\alpha.\label{lower-prob-tau}\EDE
  Suppose that $\tau^1_{r_1;\FF}<\tau^2_{r_2}$ and (\ref{2pt-condition}) holds with $c_1$ in place of $c$. If $\eta(\tau^1_{r_1;\FF})\not\in\R$, the union of $[v_m\wedge a_{K_{\tau^1_{r_1;\FF}}},a_{K_{\tau^1_{r_1;\FF}}}]$ and the boundary arc of $\pa K_{\tau^1_{r_1;\FF}}$ from $a_{K_{\tau^1_{r_1;\FF}}}$ to $\eta(\tau^1_{r_1;\FF})$  can be disconnected from $(-\infty,u_2]$ in $\HH\sem K_{\tau^1_{r_1;\FF}}$ by  $A_{\HH}(u_1,r_1,L_2)$. By Proposition \ref{extrem-prop} and conformal invariance and comparison principle of extremal length, if $\frac{r_1}{L_2}<\frac 1{288}$, then $\frac{|W(\tau^1_{r_1;\FF})-V_m(\tau^1_{r_1;\FF})|}{|W(\tau^1_{r_1;\FF})-U_2(\tau^1_{r_1;\FF})|}\le 144 \frac{r_1}{L_2}< \frac 12$. If $\eta(\tau^1_{r_1;\FF})\in\R$, then from $r_1<L_1=|v_m-u_1|$ we see that $\eta(\tau^1_{r_1;\FF})<v_m$, which implies $W(\tau^1_{r_1;\FF})=V_m(\tau^1_{r_1;\FF})$. So in any case we have
\BGE \frac{|V_m(\tau^1_{r_1;\FF})-U_2(\tau^1_{r_1;\FF})|}{|W(\tau^1_{r_1;\FF})-U_2(\tau^1_{r_1;\FF})|}=1- \frac{|W(\tau^1_{r_1;\FF})-V_m(\tau^1_{r_1;\FF})|}{|W(\tau^1_{r_1;\FF})-U_2(\tau^1_{r_1;\FF})|} > \frac 12. \label{12}\EDE
By Koebe  $1/4$ theorem, on the event $\{\tau^1_{r_1;\FF}<\tau^2_{r_2}\}$,
\BGE g_{\tau^1_{r_1;\FF}}(B_{\FF}(u_2,r_2))\supset B_{\FF}(U_2(\tau^1_{r_1;\FF}),D_2(\tau^1_{r_1;\FF}) r_2/4) ;\label{contains}\EDE
\BGE |V_m(\tau^1_{r_1;\FF})-U_2(\tau^1_{r_1;\FF})|\le D_2(\tau^1_{r_1;\FF}) \dist(u_2,\eta[0,\tau^1_{r_1;\FF}]\cup [v_m,\infty))\le  D_2(\tau^1_{r_1;\FF}) L_2 .\label{Vm-U}\EDE
By DMP, Corollary \ref{Cor-lower-1pt}, and (\ref{G-upper},\ref{12},\ref{contains},\ref{Vm-U}), there is a constant $c_2\in(0,1)$  such that when $\frac{r_1}{L_2}<\frac 1{288}$ and $\frac{r_2}{L_2}<c_2$,
\BGE \PP[\tau^2_{r_2;\FF}<\infty|\F_{\tau^1_{r_1;\FF}},\tau^1_{r_1;\FF}<\tau^2_{r_2}]\gtrsim  ( {r_2}/{L_2})^\alpha.\label{tau2r2}\EDE
Thus,  if (\ref{2pt-condition}) holds for $c:=c_1\wedge c_2\wedge \frac 1{288}$, then (\ref{lower-prob-tau}) and (\ref{tau2r2}) both hold, which together imply (\ref{lower-prob}).
\end{proof}

For  $0\le t<\tau^*_u$ and $r>0$, let $I_t(u,r)$ denote the connected component of $\{|z-u|=r\}\cap (\HH\sem K_t)$ which disconnects $u$ from $\infty$. Note that one endpoint of $I_t(u,r)$ is $u-r$.

\begin{Lemma}
Let $S_1>R_1>r_1\in (0, L_\wedge /6)$ and $r_2\in(0,L_2/2)$.
  Then
  \BGE \PP[\tau^1_{r_1;\FF},\tau^2_{r_2}<\infty;\eta[\tau^1_{R_1;\FF},\tau^1_{r_1;\FF}]\cap I_{\tau^1_{R_1;\FF}}(u_1,S_1)\ne \emptyset]\lesssim \Big(\frac{L_1}{L_0+L_1}\Big)^{{\alpha_*}}\Big (\frac{R_1}{S_1}\Big)^{{\alpha_*}} \Big(\frac{r_1r_2}{L_1L_2}\Big)^{\alpha} ;\label{E-P}\EDE
    \BGE \PP[\tau^1_{r_1;\FF}<\infty;\eta[\tau^1_{R_1;\FF},\tau^1_{r_1;\FF}]\cap I_{\tau^1_{R_1;\FF}}(u_1,S_1)\ne \emptyset]\lesssim \Big(\frac{L_1}{L_0+L_1}\Big)^{{\alpha_*}}\Big (\frac{R_1}{S_1}\Big)^{{\alpha_*}} \Big(\frac{r_1}{L_1}\Big)^{\alpha} ;\label{E-P'}\EDE
  \label{comeout-2}
\end{Lemma}
\begin{proof} We only prove (\ref{E-P}). The proof of (\ref{E-P'}) is  simpler. Let $E$ denote the event in (\ref{E-P}).
 Let $\tau_*=\inf(\{t\ge \tau^1_{R_1;\FF}:\eta(t)\in {I_{\tau^1_{R_1;\FF}}(u_1,S_1)}\}\cup\{\infty\})$. Then $E=\{\tau_*<\tau^1_{r_1;\FF}<\infty;\tau^2_{r_2}<\infty\}$. We may assume $r_2<L_2/6$ for otherwise (\ref{E-P}) reduces to (\ref{E-P'}).
 Let $R_2=L_2/2$.
 Let $N_1,N_2\in\N$ be such that $e^{-N_j} R_j\le r_j<e^{1-N_j} R_j$, $j=1,2$.
Define
$$ E^1_{n_1}=\{\tau^1_{e^{1-n_1} R_1;\FF}\le \tau_*<\tau^1_{e^{-n_1} R_1;\FF}\},\quad 1\le n_1\le N_1;$$
$$E^2_{n_2}=\{\tau^2_{e^{1-n_2} R_2}< \tau^1_{R_1;\FF}<\tau^2_{e^{-n_2}R_2}\},\quad 1\le n_2\le N_2;$$
$$E^2_0=\{\tau^1_{R_1;\FF}<\tau^2_{R_2}\},\quad E^2_{N_2+1}=\{\tau^2_{e^{-N_2} R_2}< \tau^1_{R_1;\FF}<\infty\}.$$
When $E$ happens, we have $\tau^1_{R_1;\FF}<\infty$ and $\tau^1_{R_1;\FF}<\tau_*<\tau^1_{r_1;\FF}\le \tau^1_{e^{-N_1}R_1;\FF}$ because $r_1\ge e^{-N_1} R_1$. So we have
\BGE E\subset\bigcup_{n_1=1}^{N_1} \bigcup_{n_2=0}^{N_2+1} (E^1_{n_1}\cap E^2_{n_2}).\label{E-union}\EDE
Note that all $E^1_{n_1}\cap E^2_{n_2}\in\F_{\tau_*}$.
By (\ref{est-u},\ref{21-ineq}), for all $n_1,n_2$,
\BGE \PP[E^1_{n_1}\cap E^2_{n_2}]
\lesssim  \Big(\frac{L_1}{L_0+L_1}\Big)^{{\alpha_*}}\Big(\frac{e^{1-n_1}R_1}{L_1}\Big)^\alpha \Big(\frac{ e^{1-n_2}R_2 }{L_2}\Big)^{\alpha+{\alpha_*}} .\label{n2>0}\EDE

  Suppose for some $1\le n_1\le N_1$ and $0\le n_2\le N_2$, $E^1_{n_1}\cap E^2_{n_2}$ happens.
  Let $I= I_{\tau_*}(u_1,R_1 e^{1-n_1})$. Then $I$ is a crosscut of $\HH\sem K_{\tau_*}$, which disconnects $[u_1,v_m\wedge a_{K_{\tau_*}}]\cup (B_{\FF}(u_1,r_1)\cap \lin\HH)$ from $\infty$, and can be disconnected from $\pa_{\HH}^+ K_{\tau_*}$ in $\HH\sem K_{\tau_*}$ by $A_{\HH}(u,  R_1 e^{1-n_1},S_1)$.
   By Proposition \ref{extrem-prop} and conformal invariance and comparison principle of extremal length, \BGE 1\wedge \frac{\rad_{U_1(\tau_*)}([U_1(\tau_*),V_m(\tau_*)]\cup g_{\tau_*}(B_{\FF}(u_1,r_1)\cap \lin\HH))}{|W(\tau_*)-U_1({\tau_*})|}\lesssim \frac{R_1 e^{1-n_1}}{S_1}.\label{extrem-est-WV}\EDE
  By Koebe $1/4$ theorem,
  $$|V_m(\tau_*)-U_1({\tau_*})|\ge D_1({\tau_*}) R_1 e^{-n_1}/4,\quad |U_1(\tau_*)-U_2({\tau_*})|\ge D_2({\tau_*}) R_2 e^{-n_2}/4.$$
  By Koebe distortion theorem, for $j=1,2$, if $R_je^{-n_j}>6r_j$, then
  $$\rad_{U_j({\tau_*})}  (g_{\tau_*}(B_{\C}(u_j,r_j)))\le \frac 32 D_j({\tau_*}) r_j <|V_m(\tau_*)-U_j({\tau_*})|.$$
  Combining these three estimates with DMP  and (\ref{transform-2}), we see that, for $1\le n_1\le N_1$ and $1\le n_2\le N_2$, when $R_j e^{-n_j}>6r_j$, $j=1,2$,
  \BGE \PP [E|\F_{\tau_*},E^1_{n_1}\cap E^2_{n_2}]\lesssim \Big(\frac{R_1 e^{1-n_1}}{S_1}\Big)^{{\alpha_*}} \Big(\frac{r_1}{R_1 e^{-n_1}}\Big)^\alpha  \Big(\frac{r_2}{R_2  e^{-n_2}}\Big)^\alpha.\label{two-pt-bdry-temp1}\EDE
  This estimate also holds without assuming $R_j e^{-n_j}>6r_j$, $j=1,2$. In fact, if $R_j e^{-n_j}>6r_j$ holds  for only one of $j\in\{1,2\}$, then we get (\ref{two-pt-bdry-temp1}) using (\ref{est-u}) instead of (\ref{transform-2}). If  $R_j e^{-n_j}\le 6r_j$ for $j\in\{1,2\}$, then (\ref{two-pt-bdry-temp1}) follows from (\ref{est-vm},\ref{extrem-est-WV}). A similar argument shows that
   \BGE \EE[E|\F_{\tau_*},E^1_{n_1}\cap E^2_{N_2+1} ]\le \EE[\tau^1_{r_1;\FF}<\infty |\F_{\tau_*},E^1_{n_1}\cap E^2_{N_2+1} ]\lesssim \Big(\frac{R_1 e^{1-n_1}}{S_1}\Big)^{{\alpha_*}} \Big(\frac{r_1}{R_1 e^{-n_1}}\Big)^\alpha.\label{two-pt-bdry-temp2-N}\EDE
   Here when $R_1 e^{-n_1}>6r_1$ we use  (\ref{est-u}); and when $R_1 e^{-n_1}\le 6 r_1$, we use (\ref{est-vm}).

  Combining (\ref{n2>0}) with (\ref{two-pt-bdry-temp1},\ref{two-pt-bdry-temp2-N}), we see that, for all $n_1,n_2$,
  $$\PP[E\cap E^1_{n_1}\cap E^2_{n_2}]\lesssim \Big(\frac{L_1}{L_0+L_1}\Big)^{{\alpha_*}} \Big(\frac{R_1 }{S_1}\Big)^{{\alpha_*}} \Big(\frac{r_1}{L_1}\Big)^\alpha \Big(\frac{r_2}{L_2}\Big)^\alpha  e^{-(n_1+n_2){\alpha_*}}.$$
  Summing up these inequalities and using (\ref{E-union}), we get the desired upper bound of $\PP[E]$.
\end{proof}

\begin{proof}[Proof of Theorem \ref{main-thm2}]
Suppose $r_1\in(0,L_\wedge/6 )$ and $r_2\in (0,L_2 /6)$. Choose
$S_1>R_1\in (r_1, L_\wedge/6  )$ and $R_2\in (r_2,L_2/6 )$, whose values are to  be determined later.  Recall the $I_t(u,r)$ defined before Lemma \ref{comeout-2}.
Let $\ha\tau^1_{R_1,S_1}=\inf(\{t\ge \tau^1_{R_1;\FF}:\eta(t)\in I_{\tau^1_{R_1}}(u_1, S_1)\}\cup\{\infty\})$. Define the   events:
$$ E=\{\tau^j_{r_j;\FF}<\infty,j=1,2\}; \quad   \ha E_1=\{\tau^1_{r_1;\FF}<\ha \tau^1_{R_1,S_1}\};\quad E_{ 2}^{x}=\{\tau^1_{x;\FF}<\tau^2_{R_2}\},\quad x\in\{r_1,R_1\}.$$
In the following, we use $X\st{e}{\approx}Y$ to denote the approximation relation $|X-Y|=e$, and call $e$ the error term.  We are going to use the following approximation relations:
\BGE
\begin{aligned}
 &\PP[E] \st{e_1}{\approx} \PP[E\cap \ha E_1]\st{e_2}{\approx} \PP[E\cap \ha E_1\cap E_2^{r_1}]
=\EE[\ind_{\ha E_{1}\cap E_2^{r_1}} \PP[\tau^2_{r_2;\FF}<\infty|\F_{\tau^1_{r_1;\FF}},E_2^{r_1}]]\\ \st{e_3}{\approx}& \EE[\ind_{ \ha E_{1}\cap E_2^{r_1}} G_{\tau^1_{r_1;\FF};\FF}(u_2) ]r_2^\alpha
\st{e_4}{\approx} \EE[\ind_{\ha E_{1}\cap E_2^{r_1}} G_{\tau^1_{R_1;\FF};\FF}(u_2) ]r_2^\alpha\\
\st{e_5}{\approx}& \EE[\ind_{\ha E_{1}\cap E_2^{R_1} } G_{\tau^1_{R_1;\FF};\FF}(u_2)] r_2^\alpha\st{e_6}{\approx} \EE[\ind_{\{\tau^1_{r_1;\FF}<\infty\}\cap E_2^{R_1}} G_{\tau^1_{R_1;\FF};\FF}(u_2)] r_2^\alpha\\
=&\EE[\ind_{E_2^{R_1}} G _{\tau^1_{R_1;\FF}}(u_2) \PP[\tau^1_{r_1;\FF}<\infty|\F_{\tau^1_{R_1;\FF}}]] r_2^\alpha \st{e_7}{\approx} \EE[\ind_{E_2^{R_1}} G _{\tau^1_{R_1;\FF}}(u_1) G _{\tau^1_{R_1;\FF}}(u_2) ] r_1^\alpha r_2^\alpha.
\end{aligned}
\label{approximation}
\EDE

By (\ref{est-u}),
\BGE \PP[E_2^{x}]\le \PP[\tau^1_{x;\FF}<\infty] \lesssim   ( {L_1}/({L_0+L_1}) )^{{\alpha_*}}  ( {x}/{L_1} )^{\alpha},\quad x\in\{r_1,R_1\}.\label{Ex-upper}\EDE
By Lemma \ref{Cor-main-thm1}, there is a constant $c_1\in(0,1)$ such that if $r_1/R_1<c_1$,
\BGE |\PP[\tau^1_{r_1;\FF}<\infty|\F_{\tau^1_{R_1}},\tau^1_{R_1}<\infty]-  G_{\tau^1_{R_1};\FF}(u_1)r_1^\alpha|\lesssim  (r_1/R_1)^{\alpha+\beta};\label{G1r1}\EDE
and if $r_2/R_2<c_1$, then for $x=r_1$ or $R_1$,
\BGE |\PP[\tau^2_{r_2;\FF}<\infty|\F_{\tau^1_{x}},E_2^x]-  G_{\tau^1_{x};\FF}(u_2)r_2^\alpha|\lesssim (r_2/R_2)^{\alpha+\beta}, \label{G2r2}\EDE
and so
\BGE  G_{\tau^1_{x};\FF}(u_2)r_2^\alpha \lesssim \PP[\tau^2_{r_2;\FF}<\infty|\F_{\tau^1_{x}} ]+ (r_2/R_2)^{\alpha+\beta}\quad \mbox{on }E_2^{x}.\label{G<}\EDE

We now bound the error terms in (\ref{approximation}). By (\ref{E-P}),
$$e_1=\PP[E\sem \ha E_1]\lesssim \Big(\frac{L_1}{L_0+L_1}\Big)^{{\alpha_*}}\Big (\frac{R_1}{S_1}\Big)^{{\alpha_*}} \Big(\frac{r_1}{L_1}\Big)^{\alpha} \Big(\frac{r_2}{L_2}\Big)^{\alpha}.$$
Since $E\sem E_2^{r_1}=\{\tau^2_{R_2}<\tau^1_{r_1}<\tau^2_{r_2}<\infty\}\cup \{\tau^2_{r_2}<\tau^1_{r_1}<\infty\}$, by (\ref{21-ineq},\ref{212-ineq}),
$$e_2\le \PP[E\sem E_{2}^{r_1}]\lesssim  \Big(\frac{L_1}{L_0+L_1}\Big)^{{\alpha_*}} \Big(\frac{R_2}{L_2}\Big)^{{\alpha_*}} \Big(\frac{r_1}{L_1}\Big)^\alpha \Big(\frac{r_2}{L_2}\Big)^\alpha.$$
Combining (\ref{Ex-upper}) and (\ref{G2r2}) with $x=r_1$, we find that,   when $r_2/R_2<c_1$,
$$e_3\lesssim \Big(\frac{L_1}{L_0+L_1}\Big)^{{\alpha_*}} \Big(\frac{r_1}{L_1}\Big)^\alpha \Big(\frac{r_2}{R_2}\Big)^{\alpha+\beta}.$$
By (\ref{21-ineq}),
$\PP[\ha E_1\cap(  E_2^{R_1} \sem  E_2^{r_1})]\le \PP[\tau^2_{R_2}<\tau^1_{r_1;\FF}<\infty] \lesssim (\frac{L_1}{L_0+L_1} )^{{\alpha_*}} (\frac{r_1}{L_1} )^\alpha (\frac{R_2}{L_2} )^{\alpha+{\alpha_*}}$.
By (\ref{G-upper}) and Koebe's $1/4$ theorem,    $\ind_{E_2^{R_1}}G_{\tau^1_{R_1;\FF}}(u_2)\lesssim R_2^{-\alpha}$. These estimates together imply that
$$e_5\lesssim \Big(\frac{L_1}{L_0+L_1}\Big)^{{\alpha_*}} \Big(\frac{R_2}{L_2}\Big)^{ {\alpha_*}} \Big(\frac{r_1}{L_1}\Big)^\alpha \Big(\frac{r_2}{L_2}\Big)^{\alpha}.$$
By (\ref{G<},\ref{E-P},\ref{E-P'}), when $r_2/R_2<c_1$,
$$e_6\lesssim \EE[\ind_{E_2^{R_1}\cap (\{\tau^1_{r_1;\FF}<\infty\}\sem \ha E_1)} (\PP[\tau^2_{r_2;\FF}<\infty|\F_{\tau^1_{R_1}}]+(r_2/R_2)^{\alpha+\beta})]$$
$$\le \PP[(\{\tau^1_{r_1;\FF}<\infty\}\sem \ha E_1)\cap \{\tau^2_{r_2;\FF}<\infty\}]+(r_2/R_2)^{\alpha+\beta} \PP[ \{\tau^1_{r_1;\FF}<\infty\}\sem \ha E_1]$$
$$\lesssim \Big(\frac{L_1}{L_0+L_1}\Big)^{{\alpha_*}}\Big (\frac{R_1}{S_1}\Big)^{{\alpha_*}} \Big(\frac{r_1}{L_1}\Big)^{\alpha}\Big[ \Big(\frac{r_2}{L_2}\Big)^{\alpha}+ \Big(\frac{r_2}{R_2}\Big)^{\alpha+\beta} \Big]  .$$
By (\ref{G1r1},\ref{G<}), Koebe $1/4$ theorem, and (\ref{est-u},\ref{transform-2}), when $r_j/R_j<c_1$, $j=1,2$,
$$e_7\lesssim (r_1/R_1)^{\alpha+\beta}\EE[\ind_{E_2^{R_1}} \PP[\tau^2_{r_2;\FF}<\infty|\F_{\tau^1_{R_1}}]+(r_2/R_2)^{\alpha+\beta}] $$
$$\le (r_1/R_1)^{\alpha+\beta}( \PP[\tau^1_{R_1;\FF}<\infty,\tau^2_{r_2;\FF}<\infty]+(r_2/R_2)^{\alpha+\beta} \PP[\tau^1_{R_1;\FF}<\infty])$$
$$\lesssim \Big(\frac{L_1}{L_0+L_1}\Big)^{{\alpha_*}} \Big(\frac{r_1}{R_1}\Big)^\beta \Big(\frac{r_1}{L_1}\Big)^\alpha\Big[\Big(\frac{r_2}{L_2}\Big)^\alpha+\Big(\frac{r_2}{R_2} \Big)^{\alpha+\beta}\Big].$$

When $\tau^1_{R_1}<\infty$, $I_{\tau^1_{R_1}}(u_1, S_1)$ can be disconnected in $\HH\sem K_{\tau^1_{R_1}}$ from $(-\infty,u_2]$ by  $ A_{\HH}(u,S_1,L_2)$. So $d_{\HH\sem K_{\tau^1_{R_1}}}(I_{\tau^1_{R_1}}(u_1, S_1), (-\infty,u_2])\ge \frac 1\pi \ln(\frac{L_2}{S_1})$. On the event $\ha E_1$, by Lemma \ref{Gt1t2}, if $L_2/S_1> 1440$, $|G_{\tau^1_{R_1};\FF}(u_2)-G_{\tau^1_{r_1};\FF}(u_2)|\lesssim \frac{S_1}{L_2} G_{\tau^1_{r_1};\FF}(u_2)  $. Combining this with (\ref{G<},\ref{est-u},\ref{transform-2}), we get, when $S_1/L_2<1/ 1440$ and $r_2/R_2<c_{1}$,
$$e_4 \lesssim  ({S_1}/{L_2}) ( \EE[\ind_{  E_2^{r_1}} \PP[\tau^2_{r_2;\FF}<\infty|\F_{\tau^1_{r_1}}] ]+ ( {r_2}/{R_2})^{\alpha+\beta} \PP[ E_2^{r_1}])$$
$$\le  ({S_1}/{L_2}) ( \PP[\tau^j_{r_j;\FF}<\infty,j=1,2]+  ( {r_2}/{R_2} )^{\alpha+\beta} \PP[\tau^1_{r_1;\FF}<\infty] )$$
$$\lesssim \frac{S_1}{L_2} \Big(\frac{L_1}{L_0+L_1}\Big )^{{\alpha_*}} \Big(\frac{r_1}{L_1}\Big)^\alpha \Big[ \Big(\frac{r_2}{L_2}\Big)^\alpha+\Big(\frac{r_2}{R_2}\Big)^{\alpha+\beta}\Big].$$

Now we determine $R_1,S_1,R_2$. Let $\ha r_1\in [r_1,L_\wedge/6)$ and $\ha r_2\in [r_2,L_2/6)$. Define
$$R_1=(L_\wedge /6)^{\frac{\alpha_*}{\alpha_*\beta+{\alpha_*}+\beta}} \cdot \ha r_1^{\frac{\alpha_*\beta+\beta}{\alpha_*\beta+{\alpha_*}+\beta}}, \quad S_1=(L_\wedge /6)^{\frac{ {\alpha_*}+\beta }{\alpha_*\beta+{\alpha_*}+\beta}}\cdot \ha r_1^{\frac{\alpha_*\beta }{\alpha_*\beta +{\alpha_*}+\beta}},$$ $$R_2=(L_2/6)^{\frac{\alpha+{\alpha_*}}{\alpha+{\alpha_*}+\beta}} \cdot \ha r_2^{\frac{\beta}{\alpha+{\alpha_*}+\beta}}.$$
Then $ r_1\le \ha r_1<R_1<S_2<L_\wedge /6$ and $r_2\le \ha r_2<R_2<L_2/6$.
Straightforward computation shows that  there is a constant $c\in(0,1)$   such that if $
\ha r_1/L_\wedge<c$ and $\ha r_2/L_2<c$,
then $r_j/R_j<c_1$, $j=1,2$, and $S_1/L_2<1/1440$, and so the estimates for $e_j$, $1\le j\le 7$, all hold, and we have
$$|r_1^{-\alpha}r_2^{-\alpha} \PP[\tau^j_{r_j;\FF}<\infty,j=1,2]-\EE[\ind_{E_2^{R_1}}G _{\tau^1_{R_1}}(u_1) G _{\tau^1_{R_1}}(u_2) ]| $$ $$\le r_1^{-\alpha} r_2^{-\alpha} \sum_{j=1}^7 e_j\lesssim \Big(\frac{L_1}{L_0+L_1}\Big )^{{\alpha_*}} \Big(\frac{ 1}{L_1}\Big)^\alpha \Big(\frac{1}{L_2}\Big)^\alpha\Big[ \Big(\frac{\ha r_1}{L_\wedge}\Big)^{\zeta_1}+ \Big(\frac{\ha r_2}{L_2}\Big)^{\zeta_2}\Big],$$
where $\zeta_1:=\frac{\alpha_*\beta}{\alpha_*\beta+\alpha_*+\beta}$ and $\zeta_2:=\frac{\alpha_*\beta}{\alpha+\alpha_*+\beta}$.
Now we treat $\ha r_1,\ha r_2$ as fixed. Since the RHS of the above inequality and $\EE[\ind_{E_2^{R_1}}G _{\tau^1_{R_1}}(u_1) G _{\tau^1_{R_1}}(u_2) ]$ do not depend on $r_1,r_2$, we see that the family $r_1^{-\alpha}r_2^{-\alpha} \PP[\tau^j_{r_j;\FF}<\infty,j=1,2]$ is Cauchy as $r_1,r_2\to 0^+$.
So $G_{\FF}(w,v ;u_1,u_2)\in [0,\infty)$ could be well defined by (\ref{Green-2pt}), and satisfies (\ref{G(u1,u2)-est}) by the above inequality. Formula (\ref{G(u1,u2)-upper}) follows immediately from (\ref{transform-2}) and (\ref{lower-prob}), which implies that $G_{\FF}(w,v ;u_1,u_2)\in(0,\infty)$.
\end{proof}

\begin{Remark}
With some more work, one can prove that $G_{\FF}(w,\ulin v;u_1,u_2)$ has an expression that resembles the ordered two-interior-point  Green's function for chordal SLE$_\kappa$ in \cite{LW}:
$$G_{\FF}(w,\ulin v;u_1,u_2)=  G_{0;\FF}( u_1) \EE_*[ G_{\tau^*_{u_1};\FF}(u_2)],$$
where $\EE_*$ stands for the expectation w.r.t.\ the SLE$_\kappa(\ulin\rho,\kappa-8-2\rho_\Sigma)$ curve started from $w$ with force points $(\ulin v,u_1)$. Such curve almost surely hits $u_1$ at the time $\tau^*_{u_1}$, and can be understood as the SLE$_\kappa(\ulin\rho)$ curve $\eta$ conditioned to pass through $u_1$.
\end{Remark}

\section{Minkowski Content Measure} \label{Section 6}
Let $\eta$ be as in Theorems \ref{main-thm1} and \ref{main-thm2}. The purpose of this section is to use those theorems to construct a covariant Borel measure, called the Minkowski content measure, on $\eta\cap(-\infty,v_m]$, which is closely related to the Minkowski contents of subsets of $\eta\cap(-\infty,v_m]$.

\subsection{General theory} \label{Section 6.1}
In this subsection, we first review the  Minkowski contents, and then define the Minkowski content measures, and derive some basic properties.

Let $n\in\N$. Let $\lambda^n$ denote the Lebesgue measure on $\R^n$. Fix $d\in(0,n)$. Let $S\subset\R^n$. For $r>0$, let $B(S,r)=\{x\in\R^n: \dist(x,S)\le r\}$, and $\Cont(S;r)=r^{d-n} \lambda^n(B(S,r))$. The upper and lower ($d$-dimensional) Minkowski content of $S$ are respectively defined by
$$\lin{\Cont}(S):=\limsup_{r\to 0^+} \Cont(S;r),\quad \ulin{\Cont}(S):=\liminf_{r\to 0^+} \Cont(S;r).$$
When $\lin{\Cont}(S)=\ulin{\Cont}(S)$, the common value, denoted by $\Cont(S)$ or $\Cont_d(S)$, is called the ($d$-dimensional) Minkowski content of $S$.

\begin{Remark}
 The definition here  differs from that in \cite[3.2.37]{MC} in two aspects. First, we allow $d$ to be not an integer; second, we omit a multiplicative constant: $\frac{\Gamma(\frac{n-d}2+1)}{\pi^{\frac{n-d}2}}$  for simplicity.
\end{Remark}

 For $S_1,S_2\subset \R^n$,   from $\lambda^n(B(S_1\cup S_2,r))\le \lambda^n (B(S_1,r))+\lambda^n(B(S_2,r))$, we get
\BGE \lin{\Cont}(S_1\cup S_2)\le   \lin{\Cont} (S_1)+\lin{\Cont} (S_2).\label{Cont-union-upper-n}\EDE
\BGE \ulin{\Cont}(S_1\cup S_2)\le \ulin{\Cont}(S_1)+\lin{\Cont}(S_2).\label{Cont-union-lower-n}\EDE

\begin{Lemma}
  For any compact sets $S_1,S_2\subset\R^n$,
  \BGE \ulin\Cont(S_1)+\ulin\Cont(S_2)\le \ulin\Cont(S_1\cup S_2)+\lim_{\eps\to 0^+} \lin\Cont(S_1\cap B(S_1\cap S_2,\eps)). \label{Cont-union-bigger}\EDE
\end{Lemma}
\begin{proof}
  Fix $\eps>0$. Since $ S_1\cap B(S_2,1/k)\downarrow S_1\cap S_2$,  there is $k_0\in\N$ such that $S_1\cap B(S_2,1/k_0)\subset B(S_1\cap S_2,\eps)$. Let $r\in(0,1/(2k_0))$. Suppose $x\in B(S_1,r)\cap B(S_2,r)$. Then there exist $y_j\in S_j$ such that $|x-y_j|\le r$, $j=1,2$. Then $|y_1-y_2|\le 2r\le 1/k_0$. So $y_1\in S_1\cap B(S_2,1/k_0)\subset S_1\cap B(S_1\cap S_2,\eps)$, which implies that $x\in B(S_1\cap B(S_1\cap S_2,\eps),r)$. So we get
  $B(S_1,r)\cap B(S_2,r)\subset B(S_1\cap B(S_1\cap S_2,\eps),r)$. Thus,
  $$\Cont(S_1;r)+\Cont(S_2;r)\le \Cont(S_1\cup S_2;r)+\Cont(S_1\cap B(S_1\cap S_2,\eps);r).$$
 We then get (\ref{Cont-union-bigger}) by first sending $r$ to $0^+$ and then sending $\eps$ to $0^+$.
\end{proof}

For $T\subset S\subset\R$, we write $\pa_S T$ for the relative boundary of $T $ in $S$, i.e.,  $S\cap \lin { T}\cap \lin{S\sem T}$.

\begin{Definition}
Let ${\cal M}$ denote the family of all sets in $\R^n$ that can be expressed as the intersection of an open set with a closed set.
 {Fix} $S\in\cal M$.  A $d$-dimensional Minkowski content measure on $S$  is a Borel measure $\mu$ on $S$ which satisfies that, for any compact set $K\subset S$, (i) $\mu(K)<\infty$, and (ii) whenever $\mu(\pa_S K)=0$, $\Cont_d(K)$ exists and equals $\mu(K)$. \label{Def-Mink}
\end{Definition}

\begin{Remark}
  A Minkowski content measure $\mu$  {on any $S\in\cal M$} is $\sigma$-finite since  {$S$} can be expressed as a countable union of compact sets.  {Given such $\mu$}, for $1\le k\le n$,  {let} $E_{\mu,k}$  {denote the set of} points $c\in\R$ such that $\mu\{x\in\R^n:x_k=c\}>0$.  {Then $E_{\mu,k}$} is countable. Let $\cal R_\mu$ denote the family of  {$n$-dimensional rectangles} $R=\prod_{k=1}^n I_k\subset \R^n$ such that each $I_k$ is a compact interval with $\pa I_k\cap E_{\mu,k}=\emptyset$. Then $\mu(\pa R)=0$ for $R\in \cal R_{\mu}$, and every compact  {$n$-dimensional}  rectangle in $\R^n$ could be approximated by elements in $\cal R_\mu$. \label{sigma-finite}
\end{Remark}

\begin{Lemma}
  The Minkowski content measure on any $S\in\cal M$ is unique {, when it exists}. \label{unique}
\end{Lemma}
\begin{proof}
  Let $S=F\cap G\in \cal M$, where $F$ is closed and $G$ is open. Let $\mu_1,\mu_2$ be Minkowski content measures on $S$. Let ${\cal R}_{\mu_1,\mu_2}^G$ be the family of $R\in {\cal R}_{\mu_1}\cap {\cal R}_{\mu_2}$ that are contained in $G$. For any $R\in{\cal R}_{\mu_1,\mu_2}^G$, from $\mu_j(\pa R)=0$ and $\pa_S(R\cap F)\subset \pa R$, we get $\mu_j(R)=\mu_j(R\cap F)=\Cont(R\cap F)$, $j=1,2$, which implies that $\mu_1(R)=\mu_2(R)$. Since  every compact rectangle in $G$ can be approximated by elements in ${\cal R}_{\mu_1,\mu_2}^G$, we get $\mu_1=\mu_2$.
\end{proof}

\begin{Lemma}
 (i) If $\mu$ is the Minkowski content measure on $S\in\cal M$, then for any compact set $K\subset S$, $\lin\Cont(K)\le \mu(K)$. (ii) If $S\subset\R^n$ is compact, and $\mu$ is a   measure on  $S$ such that $\Cont(S)=\mu(S)<\infty$, and for any compact set $K\subset S$, $\lin\Cont(K)\le \mu(K)$, then $\mu$ is the Minkowski content measure on $S$. \label{Lemma-upper-cont}
\end{Lemma}
\begin{proof}
(i) Let $S=F\cap G$, where $F$ is closed and $G$ is open. Let $K$ be a compact subset of $S$. Then for any $k>0$, $K$ can be covered by finitely many rectangles in $\cal R_\mu$ that are contained in $B(K,1/k)\cap G$. Let $K_k$ be the union of these rectangles. Then $K_k$ is a compact subset of $ B(K,1/k)\cap G$, and $\mu(\pa K_k)=0$. Now $K_k\cap S=K_k\cap F$ is a compact subset of $S$, and $\pa_S(K_k\cap S)\subset \pa K_k$. So $\mu(\pa_S (K_k\cap S))=0$. Thus, $\lin\Cont(K)\le \lin\Cont(K_k\cap S)=\mu(K_k\cap S)\le \mu(B(K,1/k))$. Sending $k$ to $\infty$, we get  $\lin\Cont(K)\le \mu(K)$.

(ii) Let $K$ be a compact subset of $S$ such that $\mu(\pa_S K)=0$. Let $L=\lin{S\sem K}$. Then $L$ is also a compact subset of $S$, $S=K\cup L$, and $K\cap L=\pa_S K$. From $\mu(\pa_S K)=0$ we get $\Cont(S)=\mu(S)=\mu(K)+\mu(L)$.
Thus, by the assumption and (\ref{Cont-union-lower-n}),
$$\mu(S)=\Cont(S)\le \ulin\Cont(K)+\lin\Cont(L)\le \lin\Cont(K)+\lin\Cont(L)\le \mu(K)+\mu(L)=\mu(S).$$
So we get $\Cont(K)=\mu(K)$, and conclude that $\mu$ is the Minkowski content measure on $S$.
\end{proof}
 {
\begin{Corollary}
  Let $S\in\cal M$ and $\mu$ be a measure on $S$. Suppose there exists a family of compact subsets $\{F_\iota\}_{\iota\in I}$ of $S$ such that (i) for every $\iota\in I$, $\Cont(F_\iota)=\mu(F_\iota)<\infty$;   (ii) for every compact $K\subset S$, there exists $\iota\in I$ such that $K\subset F_\iota$; and (iii) every compact $K\subset S$ can be expressed as $\bigcap_{n=1}^\infty J_n$, where $(J_n)$ is a decreasing sequence of sets such that each $J_n$ is a finite disjoint union of elements in $\{F_\iota\}$. Then $\mu$ is the Minkowski content measure on $S$. \label{cor-Mink}
\end{Corollary}
\begin{proof}
Let $K_0\subset S$ be compact. By (ii) there is $\iota\in I$ such that $K_0\subset F_\iota$. So we have $\mu(K_0)\le \mu(F_\iota)<\infty$ by (i). Suppose further that $\mu(\pa_S K_0)=0$. Let $\mu_{F_\iota}=\mu(\cdot\cap F_\iota)$ be a measure on $F_\iota$. Then $\Cont(F_\iota)=\mu_{F_\iota}(F_\iota)<\infty$. Let $K\subset F_\iota$ be compact. By (iii) we express $K=\bigcap J_n$,  where $(J_n)$ is a decreasing sequence of sets each of which is a finite disjoint union of elements in $\{F_\iota\}$. By (\ref{Cont-union-upper-n}) and (i), each $J_n$ satisfies $\lin\Cont(J_n)\le \mu(J_n)$. Since $\mu(J_n)\downarrow \mu(K)$, we get $\lin\Cont(K)\le \mu(K)$.
 By Lemma \ref{Lemma-upper-cont} (ii), $\mu_{F_\iota}$ is the Minkowski content measure on ${F_\iota}$. Since $K_0$ is a compact subset of ${F_\iota}$, and $\mu_{F_\iota}(\pa_{F_\iota} K_0)\le\mu(\pa_S K_0)=0$, we get $\Cont(K_0)=\mu_{F_\iota}(K_0)=\mu(K_0)$. Thus, $\mu$ is the Minkowski content measure on $S$.
\end{proof}}

\begin{Example}
  Let $n=1$. Let $B$ be a $1$-dimensional Brownian motion, and let $L_t$ be the usual local time of $B$ at $0$ (cf.\ \cite{RY}). Let $Z$ be the zero set $\{t:B_t=0\}$. There is a  {deterministic} constant $C_0\in(0,\infty)$ such that for any $t_2\ge t_1\ge 0$, $\Cont_{1/2}(Z\cap [t_1{  ,}t_2])=C_0 (L_{t_2}-L_{t_1})$ (cf.\ \cite{Law-BM}). {By applying Corollary \ref{cor-Mink} with $\{F_\iota\}=\{Z\cap [t_1,t_2]:t_2\ge t_1\ge 0\}$, we can conclude that $C_0 dL$ is the $1/2$-dimensional Minkowski content measure on $Z$.}
  \label{Example-1}
\end{Example}

\begin{Example}
Let $n=2$. Let $\kappa\in(0,8)$ and $d=1+\frac \kappa 8$.  Let $\eta$ be an SLE$_\kappa$ curve in $\HH$ from $0$ to $\infty$. The natural parametrization of $\eta$ is a random continuously strictly increasing function $\theta$ that satisfies that   for any $ {t_2\ge t_1\ge 0}$, $\Cont_d(\eta[ {t_1,t_2}])=\theta( {t_2})-\theta( {t_1})$ (\cite{LR}).  {By applying Corollary \ref{cor-Mink} with $\{F_\iota\}=\{\eta([t_1,t_2]):t_2\ge t_1\ge 0\}$, we can conclude that $\mu_\eta:=\eta_*(d\theta)$ is the $d$-dimensional Minkowski content measure on the image of $\eta$. Here we use the fact that $\eta^{-1}(K)$ is compact for any compact set $K$, which follows from the transience of $\eta$, i.e., $\lim_{t\to\infty} \eta(t)=\infty$.}
\end{Example}

\begin{Example}
  Let $n\ge 3$. Let $B$ be an $n$-dimensional Brownian motion. There is a constant $C_n\in(0,\infty)$ such that for any $t_2\ge t_1\ge 0$, $\Cont_{2}(B[t_1,t_2])=C_n (t_2-t_1)$ (cf.\ \cite{Law-BM}). Then $C_n B_* {(} \lambda^1|_{[0,\infty)} {)}$ is the $2$-dimensional Minkowski content measure on the image of $B$.  {The argument is similar to that in the previous example.}
\end{Example}

\begin{Theorem}
  Let $S\subset\R^n$ be compact. For $r>0$, let $\mu_r$ be the measure such that $d\mu_r=r^{d-n} \ind_{B(S,r)} d\lambda^n$. Then $\mu$ is the Minkowski content measure on $S$ iff $\mu_r\to \mu$ weakly as $r\to 0^+$. \label{weak}
\end{Theorem}
\begin{proof}
  First, suppose $\mu_r\to \mu$ weakly as $r\to 0^+$. Since $\mu_r$ is supported by $B(S,r)$, $\mu$ is supported by $S$.  Since  $|\mu_r|=\Cont(S;r)$ by definition, we get $|\mu|=\lim |\mu_r|=\Cont(S)$.
  Let $K$ be a compact subset of $S$. Fix $r_0>0$. Then for any $r\in(0,r_0)$, $\Cont(K;r)\le \mu_r(B(K;r_0))$. Sending $r$ to $0^+$, we get $\lin\Cont(K)\le \limsup_{r\to 0^+} \mu_r(B(K,r_0))\le \mu(B(K,r_0))$ since $B(K,r_0)$ is closed. Sending $r_0$ to $0^+$, we get $\lin\Cont(K)\le \mu(K)$. By Lemma \ref{Lemma-upper-cont} (ii), $\mu$ is the Minkowski content measure on $S$.

  Second, suppose $\mu$ is the Minkowski content measure on $S$. Then $|\mu|=\Cont(S)=\lim_{r\to 0^+} |\mu_r|$. Thus, the family $(\mu_r)_{r\in(0,1]}$ is pre-compact w.r.t.\ the weak convergence.  Suppose $\mu_{r_n}\to \mu'$ weakly for some sequence $r_n\to 0$. Then $\mu'$ is supported by $S$, and $|\mu'|=\lim|\mu_{r_n}|=\lim\Cont(S;r_n)=\Cont(S)$.
Let $K$ be a compact subset of $S$ such that $\mu(\pa_S K)=0$. Then $\Cont(K)=\mu(K)$. Fix $r_0\in(0,1]$. Choose $n$ big enough such that $r_n<r_0$. Then $\mu_{r_{n}}(B(K,r_0))\ge \Cont(K,r_{n})$. First sending $n$ to $\infty$ and then sending $r_0$ to $0^+$, we get $\mu'(K)\ge \Cont(K)=\mu(K)$. Now let $K$ be any compact subset of $S$. Then we could find a decreasing sequence of compact subsets $(K_n)$ of $S$ such that $K=\bigcap_n K_n$ and $\mu(\pa_S K_n)=0$ for each $n$. Then $\mu'(K)=\lim\mu'(K_n)\ge \lim\mu(K_n)=\mu(K)$. Since this holds for any compact $K\subset S$ and $|\mu'|=|\mu|$, we get $\mu'=\mu$. This means that any subsequential weak limit of $(\mu_r)$ as $r\to 0^+$ is $\mu$. Thus, $\mu_r\to \mu$ weakly as $r\to 0^+$.
\end{proof}

\begin{Theorem} [Restriction Property]
  Let $\mu$ be the Minkowski content measure on $S\in \cal M$. Let $F\subset\R^n$ be closed and $G\subset\R^n$ be open. Suppose $\mu(\pa_{S} (S\cap F))=0$. Let $S'=S\cap F\cap G$ and $\mu'=\mu|_{S'}$. Then $\mu'$ is the Minkowski content measure on $S'$. \label{restriction}
\end{Theorem}
\begin{proof}
  We have $S'\in \cal M$ since $\cal M$ is closed under finite intersections. Let $K\subset S'$ be compact. Then $K$ is also a compact subset of $S$, and so $\mu'(K)=\mu(K)<\infty$. Suppose $\mu'(\pa_{S'} K)=0$. Then $\mu(\pa_{S'} K)=0$. Let $S_F=S\cap F$.  Since $S'=S_F\cap G$ is relatively open in $S_F$, and $K\subset S'$ is compact, we get $\pa_{S'} K=\pa_{S_F} K$. From  $K\subset S_F\subset S$, we get $\pa_{S}K\subset \pa_{S_F} K\cup \pa_{S} S_F$. Thus, $\mu(\pa_{S} K)\le \mu(\pa_{S'} K)+\mu(\pa_{S} S_F)=0$.    Since $\mu$ is the Minkowski content measure on $S$, we get $\Cont(K)=\mu(K)=\mu'(K)$. So $\mu'$ is the Minkowski content measure on $S'$.
\end{proof}

The following theorem asserts that the Minkowski content measures are preserved under covariant transforms \cite[Definition 1.1]{cov}.  {For an open set $G\subset\R^n$, we say that a map $\phi:G\to \R^n$ is conformal if it is $C^1$ and the Jacobian matrix at every $x\in G$ is some positive number $s(x)$ times an orthogonal matrix, and the $|\phi'(x)|$ is understood as such (unique) number $s(x)$. When $n=1$, this means that $\phi$ is $C^1$ and $\phi'\ne 0$. When $n=2$,   $\phi$ is  holomorphic or anti-holomorphic on each component of $G$ with nonzero complex derivatives. When $n\ge 3$, by Liouville's theorem, $\phi$ is a M\"obius transformation on each component of $G$.}

\begin{Theorem}[Conformal Covariance]
Let $\mu$ be the Minkowski content measure on  $S\in\cal M$, which is relatively closed in an open set $G\subset \R^n$. Let $\phi:G\to\R^n$ be conformal and injective.
Then $\nu:=\phi_*(|\phi'|^d\cdot \mu)$ is the Minkowski content measure on $\phi(S)$, where
 $|\phi'|^d\cdot \mu$ is the measure  {$\mu$ weighted by} $|\phi'|^d$, and $\nu$ is the pushforward of $|\phi'|^d\cdot \mu$ under $\phi$. \label{phi(K)-Cont}
\end{Theorem}
\begin{proof}
 Since $S$ is relatively closed in $G$ and $\phi$ is a homeomorphism on $G$, both $\phi(G)$ and $\phi(G\sem S)$ are open. So $\phi(S)=\phi(G)\sem \phi(G\sem S)\in \cal M$.

Let $K$ be any compact subset of $G$. Let $m_K=\min_{x\in K} |\phi'(x)|$ and $M_K=\max_{x\in K} |\phi'(x)|$. Let $c>1$. Then there is $\delta>0$ such that $|\phi'(x)|<cM_K$ for $x\in B(K,\delta)$, and $|((\phi)^{-1})'(y)|\le c/m_K$ for $y\in B(\phi(K),\delta)$. Thus, for $r>0$ small enough, we have
$$\phi(B(K,\frac r{ cM_K}))\subset B(\phi(K),r)\subset \phi(B(K,\frac {r}{m_K/c})),$$
which then implies that
$$\frac{(m_K/c)^n}{(cM_K)^{n-d}} \Cont(K;\frac {r}{cM_K})\le \Cont(\phi(K );r)\le \frac {(cM_K)^n}{(m_K/c)^{n-d}} \Cont(K ;\frac {r}{m_K/c}).$$
First sending $r$ to $0^+$ and then sending $c$ to $1^+$, we get
\BGE  {m_K^n}{M_K^{d-n}}\cdot \ulin\Cont(K)\le \ulin\Cont(\phi(K))\le \lin\Cont(\phi(K))\le  {M_K^n}{m_K^{d-n}} \cdot \lin\Cont(K).\label{phi(K)-ineq}\EDE
Now assume $K\subset S$ and $\mu(\pa_S K)=0$. Since $\Cont(K)=\mu(K)$ and $m_K\le |\phi'|\le M_K$ on $K$, we get
\BGE \frac {m_k^n}{M_K^n} \int_K |\phi'(x)|^d \mu (dx)\le \ulin{\Cont}(\phi (K ))\le \lin{\Cont}(\phi(K )) \le  \frac {M_K^n}{m_K^n} \int_K |\phi'(x)|^d \mu (dx). \label{phimM}\EDE

Let $K_\phi\subset \phi(S)$ be compact. Let $K=\phi^{-1}(K_\phi)$. Then $K\subset S$ is compact. So $\mu(K)<\infty$. Since $|\phi'|$ is bounded on $K$,  $\nu(K_\phi)<\infty$. Suppose $\nu(\pa_{\phi(S)} K_\phi)=0$. We want to show that $\Cont(K_\phi)=\nu(K_\phi)$.  Since $\pa_S K=\phi^{-1}(\pa_{\phi(S)} K_\phi)$, we get $\mu(\pa_S K)=0$.

Fix $c>1$. We may find   $R_1,\dots,R_N\in {\cal R}_\mu$ which intersect each other only at their boundaries such that $K\subset \bigcup_{j=1}^N R_j\subset G$, and for each $1\le j\le N$, $ \max_{x\in R_k}|\phi'(x)|\le c \min_{x\in R_k}|\phi'(x)|$. Let $K_j=K\cap R_j$, $1\le j\le N$. Then $\mu(\pa_S K_j)\le \mu(\pa_S K)+\mu(\pa R_j)=0$. By (\ref{phimM}),
\BGE c^{-n} \int_{K_j} |\phi'(x)|^d \mu (dx)\le \ulin{\Cont}(\phi (K_j ))\le \lin{\Cont}(\phi(K_j )) \le  c^n \int_{K_j} |\phi'(x)|^d \mu (dx). \label{phimM-j}\EDE
Summing up the last inequality of (\ref{phimM-j}) over $1\le j\le N$, and using (\ref{Cont-union-upper-n}) and that $\mu(K_j\cap K_k)\le \mu(R_j\cap R_k)=0$ when $j\ne k$, we get
$\lin{\Cont}(K_\phi)   \le  c^n \int_{K} |\phi'(x)|^d \mu (dx)$.

Let $1\le k<j\le N$. Since $\mu(K_j\cap K_k)=0$, we get $\lim_{\delta\to 0^+} \lin\Cont( K_j\cap B(K_j\cap K_k,\delta))=0$ by Lemma \ref{Lemma-upper-cont} (i). By (\ref{phi(K)-ineq}),   $\lim_{\eps\to 0^+} \lin\Cont( \phi(K_j)\cap B(\phi(K_j)\cap \phi(K_k),\eps))=0$. By (\ref{Cont-union-upper-n}), we then get
$\lim_{\eps\to 0^+} \lin\Cont( \phi(K_j)\cap B(\phi(K_j)\cap \bigcup_{k=1}^{j-1} \phi(K_k),\eps))=0$ for $2\le j\le n$. By (\ref{Cont-union-bigger}) and induction, we get $\ulin\Cont(K_\phi)\ge \sum_{j=1}^N \ulin\Cont( \phi(K_j))$, which together with the first inequality of (\ref{phimM-j})  implies that  $\ulin\Cont(K_\phi)\ge c^{-n} \int_{K} |\phi'(x)|^d \mu (dx)$. Combining this inequality with $\lin{\Cont}(K_\phi)   \le  c^n \int_{K} |\phi'(x)|^d \mu (dx)$   and sending $c$ to $1^+$, we get $\Cont(K_\phi)=\int_{K} |\phi'(x)|^d \mu (dx)=\nu(K_\phi)$. Thus, $\nu$ is the Minkowski content measure on $\phi(S)$.
\end{proof}

\begin{Remark}
  It's easy to extend the notions of Minkowski content and Minkowski content measure  to Riemannian manifolds. The propositions in this subsection will still hold.
\end{Remark}

\begin{Remark} Most well-known deterministic fractal sets do not have non-trivial Minkowski contents. However, we may define the average Minkowski content by
$$\Cont_d^{\aaa}(S)=\lim_{t\to \infty} \frac 1t \int_0^t \Cont_d(S;e^{-r})dr,$$
and use this to define the average Minkowski content measure.  {All} results in this subsection  hold   {with the word ``average'' added before ``Minkowski'' and $\Cont$ replaced by $\Cont^{\aaa}$ except that Theorem \ref{weak} requires moderate modification. The correct statement is: for any compact set $S\subset\R^n$, $\mu$ is the average Minkowski content measure on $S$ iff $\mu$ is the weak limit of $\mu_t$ as $t\to \infty$, where $\mu_t$ is the measure on $\R^n$ that satisfies $d\mu_t/d\lambda^n=\frac 1t\int_0^t e^{(n-d)s} \ind_{B(S,e^{-s})} ds$.}
\end{Remark}

\begin{Example}
   Let $S\subset [0,1]$ be the $1/3$-Cantor set. Let $d=\frac{\ln 2}{\ln 3}$. Then $\Cont_d(S,e^{-t})$, $t\ge 0$, has period $\ln(3)$, and is not constant. So the Minkowski content $\Cont_d(S)$ does not exist, but the average Minkowski content $\Cont_d^{\aaa}(S)=C_0$ for some $C_0\in(0,\infty)$.  {Then the average Minkowski content measure on $S$ exists and equals $C_0\mu$, where $\mu$ is the Haar probability measure on $S$.}
\end{Example}

\subsection{Boundary Minkowski content of SLE$_\kappa({\protect\ulin\rho})$} \label{Section 6.2}
Now we come back to the SLE$_\kappa(\ulin\rho)$ curve $\eta$ in Theorem \ref{main-thm1} with the additional assumption that $\rho_\Sigma\in (\frac\kappa 2-4,\frac\kappa 2-2)$. This means that $\eta$ intersects the interval $(-\infty,v_m)$.
Let
$$d=1-\alpha=\frac{(\rho_\Sigma+4)(\kappa-4-2\rho_\Sigma)}{2\kappa}.$$
It is known that $d$ is the Hausdorff dimension of $\eta\cap (-\infty,v_m)$ (cf.\ \cite{MW}). We are going to prove the existence of the $d$-dimensional { {Minkowski}} content measure on $\eta\cap (-\infty,v_m)$.

We write $G_1(u)$ for the function $C_{\R} G(w,\ulin v;u)$ in Theorem \ref{main-thm1}, and $G_2(u_1,u_2)$ for the function $G_{\R}(w,\ulin v;u_1,u_2)$ in Theorem \ref{main-thm2}. We also define $G_2(u_1,u_2)=G_2(u_2,u_1)$ if $u_1<u_2\in (-\infty, v_m)$.  {We use $\lambda$ to denote the Lebesgue measure on $\R$.}

\begin{Theorem}
Almost surely  the (random) $d$-dimensional Minkowski content measure $\mu_\eta$ on $\eta\cap (-\infty,v_m]$ in $\R$ exists, is atomless, and satisfies the following properties.
  \begin{itemize}
    \item [(i)] If $\phi$ is a deterministic nonnegative measurable function on $(-\infty,v_m]$, then
  \BGE\EE\Big[\int \phi(u) \mu_\eta(du)\Big]=\int \phi(u) G_1(u) \lambda(du).\label{EC1-phi}\EDE
  \item [(ii)]  If $\psi$ is a deterministic nonnegative measurable function on $(-\infty,v_m]^2$, then
  \BGE \EE\Big[\int\!\!\int \psi(u_1,u_2) \mu_\eta^2(du_1\otimes du_2)\Big]=\int \!\!\int \psi(u_1,u_2) G_2(u_1,u_2) \lambda^2(du_1\otimes du_2).\label{EC2-psi}\EDE
  \item [(iii)] For any deterministic compact set $S\subset (-\infty,v_m]$ with $\Cont(\pa_{ {\R}} S)=0$, almost surely
  \BGE  \Cont(\eta\cap S)=\mu_\eta(S).\label{Cont=muL}\EDE
   \item [(iv)] The (minimal) support of $\mu_{ {\eta}}$ (i.e., the smallest closed set $F\subset \R$ such that $\mu_{ {\eta}}(F^c)=0$) is a.s.\ $\eta\cap (-\infty,v_m]$.
  \end{itemize}
 \label{Thm-Mink}
\end{Theorem}

\begin{Remark}
  We emphasize here that the a.s.\ existence of the  Minkowski content measure $\mu_\eta$ means that, on an event with probability $1$, $\mu_\eta$ satisfies properties (i) and (ii) in Definition \ref{Def-Mink}  simultaneously for all compact subsets of $\eta\cap[-\infty,v_m]$.
\end{Remark}

\begin{proof}
For $r>0$, define random functions $I_r,J_r$ on $(-\infty,v_m)$ by $I_r(u)= \ind_{B_{\R}(\eta\cap \R,r)}(u)$ and $J_r(u)=r^{-\alpha} I_r(u)$.  By Theorems \ref{main-thm1} and \ref{main-thm2}, for $u\in (-\infty,v_m)$, $\EE[J_r(u)]\to G_1(u)$ as $r\to 0^+$, and for $u_1\ne u_2\in (-\infty,v_m)$, $\EE[J_{r_1}(u_1)J_{r_2}(u_2)]\to G_2(u_1,u_2)$ as $r_1,r_2\to 0^+$.

By (\ref{est-u*},\ref{G-upper}), for any $u\in (-\infty,v_m)$ and $r>0$,
\BGE  \EE[J_r(u)],G_1(u) \lesssim |v_m-u|^{-\alpha}.\label{JG1'}\EDE
By (\ref{transform-2},\ref{G(u1,u2)-upper}), for $u_1\ne u_2\in (-\infty,v_m)$ and $r_1,r_2>0$,
\BGE \EE[J_{r_1}(u_1)J_{r_2}(u_2)],G_2(u_1,u_2)\lesssim |v_m-u_1|^{-\alpha} |u_1-u_2|^{-\alpha}.\label{JG2'}\EDE
Combining (\ref{JG2'}) with (\ref{G(u1,u2)-est}) and setting $\zeta:=\zeta_1\wedge \zeta_2\wedge \frac{1-\alpha}2\in (0,1-\alpha)$, we get
\begin{align}
  &|\EE[J_{r_1}(u_1)J_{r_2}(u_2)]-G_2(u_1,u_2)| \nonumber \\
  \le & |v_m-u_1|^{-\alpha}|u_1-u_2|^{-\alpha}\Big(1\wedge \Big(\frac{r_1\vee r_2}{|v_m-u_1|\wedge |u_1-u_2|}\Big)^\zeta \Big). \label{J-G2'}
\end{align}

Fix a bounded measurable set $S\subset (-\infty,v_m]$. Let $A=\sup\{|v_m-u|:u\in S\}$. Define
\BGE Y_r =\int_S J_r(u) \lambda(du) =r^{-\alpha} \lambda(B_{\R}(\eta\cap\R,r)\cap S).\label{YrS'}\EDE
From (\ref{JG2'}) and that $\alpha\in(0,1)$ we get
$$\int_S\!\int_S G_2(u_1,u_2)\lambda(du_1)\lambda(du_2)\lesssim \int_0^A L_1^{-\alpha}\int_0^{A-L_1} L_2^{-\alpha} \lambda(dL_2)\lambda(dL_1)<\infty.$$
The finiteness of this integral together with (\ref{J-G2'}) implies that, for any $r_1,r_2>0$,
\begin{align*}
 & \EE[(Y_{r_1}-Y_{r_2})^2]=\EE\Big[\Big(\int_S J_{r_1}(u)\lambda(du)-\int_S J_{r_2}(u)\lambda(du)\Big)^2\Big]\\
  \le &2 \sum_{j=1}^2 \sum_{k=1}^2 \int\!\!\int_{\{(u_1,u_2)\in S^2:u_1>u_2\}} |\EE[J_{r_j}(u_1) J_{r_k}(u_2)]-G(u_1,u_2)| \lambda(du_1) \lambda(du_2)\\
  \lesssim
 & \int_{0}^{A}\!\!\int_{0}^{A }  L_1^{-\alpha} L_2^{-\alpha} \Big(1\wedge \Big(\frac {r_1\vee r_2}{L_1\wedge L_2}\Big)^\zeta\Big)\lambda(dL_1) \lambda(dL_2),
\end{align*}
where in the last line we used a change of variables: $L_1=v_m-u_1$ and $L_2=u_1-u_2$. Since $\alpha\in(0,1)$ and $\zeta\in (0,1-\alpha)$,
we can easily bound the last integral and get
\BGE \|Y_{r_1}-Y_{r_2}\|_{L^2}^2 \lesssim A^{2-2\alpha-\zeta} (r_1\vee r_2)^{\zeta}.
\label{L2diff'}\EDE
This implies that $(Y_r)$ is Cauchy in $L^2$ as $r\to 0^+$. Let $\xi_S\in L^2$ be the limit.

By (\ref{L2diff'}) and Chebyshev's inequality,
$$\sum_{n=1}^\infty \PP[|Y_{Ae^{-n}}-Y_{Ae^{1-n}}|\ge  A^{1-\alpha} e^{-n\zeta/3}]\lesssim \sum_{n=1}^\infty e^{-n\zeta/3}<\infty.$$
By Borel-Cantelli lemma, with probability one, $|Y_{Ae^{-n}}-Y_{Ae^{1-n}}|\le A^{1-\alpha} e^{-n\zeta/3}$ for sufficiently large $n$, which implies that a.s.\ $\lim_{n\to\infty} Y_{e^{-n}}$ converges. The limit is $\xi_S$ since $Y_{Ae^{-n}}\to \xi_S$ in $L^2$. A similar argument shows that, for any $k\in\N$, a.s.\ $Y_{Ae^{-n/2^k}}\to \zeta$ as $n\to\infty$.

Let $E_0$ be the event that $\lim_{n\to\infty} Y_{e^{-n/2^k}}= \zeta$ for every $k\in\N$. Then $\PP[E_0]=1$. By the definition of $I_r$ and $J_r$, for $r_1\ge r_2>0$, we have $I_{r_1}(u)\ge I_{r_2}(u)$ and $J_{r_1}(u)\ge (r_2/r_1)^\alpha J_{r_2}(u)$, which implies that $Y_{r_1}\ge (r_2/r_1)^\alpha Y_{r_2}$. Let $k\in\N$. For $r\in(0,1]$, there is a unique $n_k(r)\in\N$ such that $e^{-n_k(r)/2^k}< r\le e^{(1-n_k(r))/2^k}$. Then we get
$ e^{-\alpha /2^k} Y_{e^{-n_k(r)/2^k}}\le Y_r\le  e^{\alpha/2^k}  Y_{e^{(1-n_k(r))/2^k}}$, which implies that
$e^{-\alpha/2^k} \zeta\le \liminf_{r\to 0^+} Y_r\le \limsup_{r\to 0^+} Y_r\le e^{\alpha/2^k}\zeta$ on the event $E_0$.
Since this holds for every $k\in\N$, and $\PP[E_0]=1$, we get a.s.\ $Y_r\to\xi_S$ as $r\to 0^+$.

So far we have proved that, for any bounded measurable set $S\subset (-\infty,v_m]$, there is a nonnegtive random variable $\xi_S\in L^2$ such that, as $r\to 0^+$, $\int_S J_r(u) \lambda(du)\to \xi_S$  almost surely and in $L^2$. It is clear that a.s.\ $\xi_{S_1\cup S_2}=\xi_{S_1}+\xi_{S_2}$ if $S_1\cap S_2=\emptyset.$ By Fubini Theorem,
$$\int_S \EE[J_r(u)]\lambda(du)\to \EE[\xi_S],\quad \int_S\!\int_S \EE[J_r(u_1)J_r(u_2)] \lambda(du_1) \lambda(du_2)\to \EE[\xi_S^2].$$
By polarization, the second limit further implies that for any bounded measurable sets $S_1,S_2\subset (-\infty,v_m]$,
$$\int_{S_1}\!\int_{S_2} \EE[J_r(u_1)J_r(u_2)] \lambda(du_1) \lambda(du_2)\to \EE[\xi_{S_1}\xi_{S_2}].$$
Since $\EE[J_r(u)]\to G_1(u)$ and $\EE[J_r(u_1)J_r(u_2)]\to G_2(u_1,u_2)$ as $r\to 0^+$, by (\ref{JG1'},\ref{JG2'}) and dominated convergence theorem, for any bounded measurable sets $S,S_1,S_2\subset (-\infty,v_m]$,
\BGE \EE[\xi_S]=\int_S G_1(u) \lambda(du),\quad \EE[\xi_{S_1 }\xi_{S_2}]=\int_{S_1}\!\int_{S_2} G_2(u_1,u_2) \lambda(du_1) \lambda(du_2).\label{moment-zeta}\EDE

We now define $\theta(u)=\xi_{[u,v_m]}$, $u\in (-\infty,v_m]$. Then for $a<b\in (-\infty,v_m)$,  by (\ref{moment-zeta},\ref{JG2'})
$$\EE[(\theta(a)-\theta(b))^2]=\EE[\xi_{[a,b)}^2]=\int_{a}^{b}\int_{a}^{b} G_2(u_1,u_2) du_1 du_2$$
$$\lesssim \int_{v_m-b}^{v_m-a} \int_0^{b-a} L_1^{-\alpha} L_2^{-\alpha} \lambda(dL_2)\lambda(dL_1)\lesssim |v_m-b|^{-\alpha} |b-a|^{ 2-\alpha}.$$
By Kolmogorov continuity theorem, for any $b\in (-\infty,v_m)\cap \Q$, $\theta$ has a locally H\"older continuous version on $(-\infty,b]$. Thus, $\theta$ has a continuous version on $(-\infty,v_m)$. Such continuous version is nonnegative and decreasing since for any $a<b\in(-\infty,v_m]$, a.s.\ $\theta(a)-\theta(b)=\xi_{[a,b)}\ge 0$. By (\ref{moment-zeta},\ref{JG1'}), $\EE[\theta(b)]\to 0$ as $b\to v_m^-$. So $\theta$ has a continuous decreasing version on $(-\infty,v_m]$ with $\theta(v_m)=0$. Such version of $\theta$ determines an atomless locally finite measure $\mu_\eta$ on $(-\infty,v_m]$.
Then for any set $S\subset (-\infty,v_m]$, which is a finite disjoint union of intervals of the form $(a_j,b_j]$, we have a.s.\ $\mu_\eta(S)=\xi_S$. By monotone class theorem, for any bounded measurable set $S\subset (-\infty,v_m]$, a.s.\ $\mu_\eta(S)=\xi_S$. By (\ref{moment-zeta}), we get
  \BGE\EE[\mu_\eta(S)]=\int_S G_1(u) \lambda(du),\label{EC1}\EDE
  and if   $T=S_1\times S_2$ for some bounded measurable sets $S_1,S_2\subset (-\infty,v_m]$, then
  \BGE \EE[\mu_\eta^2(T)]=\int \!\!\int_{T} G_2(u_1,u_2) \lambda^2(du_1\otimes du_2).\label{EC2}\EDE
By monotone convergence theorem and monotone class theorem, (\ref{EC1}) holds for any measurable set $S\subset (-\infty,v_m]$, and
(\ref{EC2}) holds for any measurable set $T\subset (-\infty,v_m]^2$. By monotone convergence theorem, these results further imply (\ref{EC1-phi},\ref{EC2-psi}).

Let $S$ be a compact subset of $(-\infty,v_m]$ with $\Cont(\pa S)=0$. For $r>0$,
since the symmetric difference between $B_{\R}(\eta\cap \R,r)\cap S$ and $B_{\R}(\eta\cap S,r)$ is contained in $B_{\R}(\pa S,r)$, by (\ref{YrS'}) we get $| \int_S J_r(u) \lambda(du)-\Cont(\eta\cap S;r)|\le \Cont(\pa S;r)$. Sending $r$ to $ 0^+$ and using that a.s.\ $\int_S J_r(u) \lambda(du)\to \xi_S=\mu_\eta(S)$, we get (\ref{Cont=muL}).

 Applying (\ref{Cont=muL}) to $S$ of the form $[a,b]$, we find that  there is an event $E$ with $\PP[E]=1$ such that on the event $E$, for any $a<b\in (-\infty,v_m]\cap \Q$, $\Cont(\eta\cap [a,b])=\mu_\eta[a,b]=\theta(a)-\theta(b)$.  {Since $\PP[E]=1$, by applying Corollary \ref{cor-Mink} with $\{F_\iota\}=\{\eta\cap [a,b]:a<b\in (-\infty,v_m]\cap \Q\}$, we conclude that $\mu_\eta$ is a.s.\ the Minkowski content measure on $\eta\cap (-\infty,v_m]$.}

To prove Part (iv), we need the lemma below.

\begin{Lemma}
  Let $\tau$ be a finite $\F$-stopping time. Suppose $\eta^\tau$ is a continuous curve in $\lin\HH$, which satisfies that a.s.\ $f_\tau(\eta^\tau)=\eta(\tau+\cdot)$.   Then
  $\mu^\tau:=(g_\tau)_*((g_\tau')^d\cdot \mu_\eta|_{(-\infty,v_m\wedge a_{K_\tau})})$
  is a.s.\ the Minkowski content measure on $\eta^\tau\cap (-\infty,V_m(\tau)]$.  {Recall that $a_{K_\tau}=\min(\lin{K_\tau}\cap\R)$.}
  \label{lemma-mu-tau}
\end{Lemma}
\begin{proof}
By DMP, conditionally on $\F_\tau$, $\eta^\tau$ is an SLE$_\kappa(\ulin\rho)$ curve started from $W(\tau)$ with force points $\ulin V(\tau)$. Thus, the Minkowski content measure $\mu_{\eta^\tau}$ on $\eta^\tau\cap (-\infty,V_m(\tau)]$ a.s.\ exists.
By Theorem \ref{restriction}, $\mu_\eta|_{(-\infty,v_m\wedge a_{K_\tau})}$ and $\mu_{\eta^\tau}|_{ (-\infty,V_m(\tau))}$ are respectively the Minkowski content measure on $\eta\cap (-\infty,v_m\wedge a_{K_\tau})$ and $\eta^\tau\cap (-\infty,V_m(\tau))$. Since $g_\tau$ is injective and $C^1$ on $(-\infty,v_m\wedge a_{K_\tau})$, and sends $\eta\cap (-\infty,v_m\wedge a_{K_\tau})$ to $\eta^\tau\cap (-\infty,V_m(\tau))$, by Theorem \ref{phi(K)-Cont},
  $\mu_{\eta^\tau}|_{ (-\infty,V_m(\tau))}=(g_\tau)_*((g_\tau')^d\cdot \mu_\eta|_{(-\infty,v_m\wedge a_{K_\tau})})=\mu^\tau$. Since $\mu_{\eta^\tau}$ is supported by $(-\infty,V_m(\tau)]$, and is atomless, we get $\mu^\tau=\mu_{\eta^\tau}$. So we get the conclusion.
\end{proof}

Now we prove Part (iv). Let $\supp(\nu)$ denote the (minimal) support of a measure $\nu$.
First consider the case that $w=0$ and $v_m=0^-$. Then $\eta$ is an SLE$_\kappa(\rho_\Sigma)$ curve started from $0$ with the  force point $0^-$. It satisfies the scaling property: for any $a>0$, $a\eta$ has the same law as $\eta$ up to a linear time-change. For $L\in(-\infty,0)$, let $p(L)=\PP[\mu_\eta[L,0]=0]$. Then $p(L)$ is increasing with $\lim_{L\to -\infty} p(L)= \PP[\mu_\eta=0]$ and $\lim_{L\to 0^-} p(L)=\PP[0\not\in \supp(\mu_\eta)]$. By Theorem \ref{phi(K)-Cont}, $\mu_\eta$ has the same law as $a^d \mu_\eta(\cdot/a)$, which implies that $p(L)=p(aL)$ for any $a>0$. So $p$ is constant on $(-\infty,0)$. Denote the constant value by $p_0\in[0,1]$. Then $\PP[\mu_\eta=0]=\PP[0\not\in \supp(\mu_\eta)]=p_0$.

Let $\tau=\tau^*_{-1}$. Then $\tau$ is a finite $\F$-stopping time, and $\eta(\tau)\in(-\infty,-1)$. By DMP, $g_\tau\circ \eta(\tau+\cdot)-W(\tau)$ has the same law as $\eta$, and is independent of $\F_\tau$. By  Lemma \ref{lemma-mu-tau}, we have $\PP[\eta|_{(-\infty,\eta(\tau)]}=0|\F_\tau]=p_0$. Since $\mu_\eta[\eta(\tau),0]=\Cont(\eta[0,\tau]\cap (-\infty,0])$ is $\F_\tau$-measurable, we have
$$p_0=\PP[\mu_\eta=0]=\PP[\{\mu_\eta|_{[\eta(\tau),0]}=0\}\cap \{\mu|_{(-\infty,\eta(\tau)]}=0\}]=p_0 \PP[\mu_\eta|_{[\eta(\tau),0]}=0]\le p_0^2.$$  If $p_0=1$, then by $\PP[\mu_\eta=0]=p_0$ and (\ref{EC1-phi}), we get $0=\EE[\mu_\eta(-\infty,0]]=\int_{-\infty}^0 G_1(u) \lambda(du)>0$, which is a contradiction. So  $p_0=0$, and we get $\PP[0\in\supp(\mu_\eta)]=1$.

Now we consider the general case. To prove that a.s.\ $\supp(\mu_\eta)=\eta\cap (-\infty,v_m]$, it suffices to prove that, for any $a<b\in(-\infty,v_m]\cap \Q$, on the event $\eta\cap (a,b)\ne\emptyset$, a.s.\ $\mu_\eta(a,b)>0$. Fix $a<b\in(-\infty,v_m]\cap \Q$. Let $\tau=\tau^*_b$. Then $\tau$ is a finite $\F$-stopping time, and $\eta(\tau)\in (-\infty,b]$. By DMP, $g_\tau(\eta(\tau+\cdot))-W(\tau) $, denoted by $\eta^\tau$, has the law of an SLE$_\kappa(\rho_\Sigma)$ curve started from $0$ with the force point $0^-$. Let $\mu^\tau$ be the Minkowski content measure on $\eta^\tau\cap(-\infty,0]$. By the last paragraph, a.s.\ $0\in \supp(\mu^\tau)$. By Lemma \ref{lemma-mu-tau}, a.s.\ $\eta(\tau)\in \supp(\mu_\eta|_{(-\infty,\eta(\tau)]})$. On the event $\eta\cap (a,b)\ne\emptyset$, we have $\eta(\tau)\in (a,b]$, which together with a.s.\  $\eta(\tau)\in \supp(\mu_\eta|_{(-\infty,\eta(\tau)]})$ implies that a.s.\ $\mu_\eta(a,b)>0$ on this event. The proof is now complete.
\end{proof}

\begin{Remark}
Let $S=[v_m-A,v_m]$ for some $A>0$. By (\ref{EC2-psi}) and (\ref{JG2'}), for any  $\eps\in(0,d)$,
$$\EE\Big[\int\!\!\int_{S^2} \frac{\mu_\eta^2 (du_1\otimes du_2)}{|u_1-u_2|^{d-\eps}}\Big]
=\int\!\!\int_{S^2} |u_1-u_2|^{-c} G_2(u_1,u_2) \lambda^2(du_1\otimes du_2)$$
$$\lesssim \int_0^{A}\!\! \int_0^{A } L_1^{-\alpha}  L_2^{-\alpha-(d-\eps)}\lambda(dL_1) \lambda( dL_2)\le\frac{A^{1-\alpha}}{1-\alpha}\cdot \frac{A^{\eps}}{\eps} <\infty,$$
which implies that a.s.\ $\int\!\!\int_{S^2} \frac{\mu_\eta^2 (du_1\otimes du_2)}{|u_1-u_2|^{d-\eps}}<\infty$. This shows that $\mu_\eta$ has finite $(d-\eps)$-energy. So, on the event $\{\eta\cap S\ne \emptyset\}$, $\mu_\eta|_S$ could serve as a Frostman measure on $\eta\cap S$.
\label{Frostman}
\end{Remark}

\begin{Theorem}
For  $u\in (-\infty,v_m)$, let
 $$M_u(t)=\ind_{\{t< \tau^*_u\}}  g_t'(u)^\alpha C_{\R}  G (W(t),\ulin V(t);g_t(u)),\quad t\ge 0.$$
Let $S\subset (-\infty,v_m)$ be compact, and define the supermartingale
 $M_{S}(t)=\int_{S} M_u(t) \lambda(du)$.
  Then we have the Doob-Meyer decomposition:
 $$M_{S}(t)=\EE[\mu_\eta(S)|\F_t]-\mu_\eta(\eta[0,t]\cap S),\quad t\ge 0.$$
\end{Theorem}
\begin{proof}
Since $ \mu_\eta(S) -\mu_\eta(\eta[0,t]\cap S)=\mu_\eta(\eta(t,\infty)\cap S) $ and  $\mu_\eta(\eta[0,t]\cap S)=\Cont(\eta[0,t]\cap S)\in\F_t$, it suffices to show $\EE[\mu_\eta(\eta(t,\infty)\cap S)|\F_t]=M_S(t)$. Since  $\eta(t,\infty)\subset (-\infty,a_{K_t})$ and $\mu_\eta$ is supported by $(-\infty,v_m)$, $\mu_\eta(\eta(t,\infty)\cap S)=\mu_\eta((-\infty,a_{K_t}\wedge v_m)\cap S)$.

  Let $\eta^{t}$ be the SLE$_\kappa(\ulin\rho)$ curve conditionally on $\F_{t}$  started from $W(t)$ with force points $\ulin V(t)$ such that $\eta(t+\cdot)=f_{t} \circ \eta^{t}$. Let $\mu^{t}$ be the the Minkowski content measure on $\eta^{t}\cap (-\infty, V_m(t)]$. By (\ref{EC1-phi}), if $\phi$ is an $\F_{t}$-measurable random nonnegative measurable function on $(-\infty,V_m(t))$\footnote{More precisely, $\phi$ is an $\F\times {\cal B}(\R)$-measurable function defined on $\{(\omega,x)\in \Omega\times \R: x<V_m(\omega,t)\}$. Here we use $V_m(\omega,t)$ to emphasize the dependence of $V_m(t)$ on $\omega\in\Omega$.}, then
  \BGE \EE\Big[\int \phi(x) \mu^{t}(du)\Big|\F_{t}\Big] =\int \phi(x) C_{\R}  G(W(t),\ulin V(t);x) \lambda(dx).\label{eta-S*}\EDE

By Lemma \ref{lemma-mu-tau}, $\mu^{t}=(g_{t})_*((g_{t}')^d\cdot \mu_\eta|_{(-\infty, a_{K_{t}}\wedge v_m)})$. Thus,
\begin{align*}
  &\EE[\mu_\eta((-\infty,a_{K_t}\wedge v_m)\cap S)|\F_{t_0}]=\EE\Big[\int_{(-\infty,a_{K_t}\wedge v_m)\cap S} g_{t}'(g_{t}^{-1}(x))^{-d} \mu^{t}(dx)\Big|\F_{t}\Big]\\
  =& \int_{(-\infty,V_m(t))\cap g_t(S)}  g_{t}'(g_{t}^{-1}(x))^{-d}C_{\R}  G(W(t),\ulin V(t);x) \lambda(dx)\qquad (\mbox{by }(\ref{eta-S*}))\\
  =& \int_{(-\infty,a_{K_t}\wedge v_m)\cap S}  g_{t}'(u)^{\alpha} C_{\R} G(W(t),\ulin V(t);g_{t}(u))  \lambda(du)\qquad (x=g_{t}(u))\\
  = &\int_{(-\infty,a_{K_t}\wedge v_m)\cap S} M_u(t ) \lambda(du)=  \int_{S} M_u(t ) \lambda(du)= M_{S}(t),
\end{align*}
where the second last ``$=$'' holds because $S\subset (-\infty,v_m]$ and $M_u(t)=0$ for $u\ge a_{K_t}$.
 So we get $\EE[\mu_\eta(\eta(t,\infty)\cap S)|\F_t]=M_S(t)$, as desired.
\end{proof}

\begin{Remark}
Note that the above proof does not use the Doob-Meyer decomposition theorem.
When all $\rho_j$ equal $0$, the theorem shows that, if $\eta$ is a chordal SLE$_\kappa$ curve ($\kappa\in(4,8)$) with no force point, the measure $\mu_\eta$ agrees with the covariant measure derived in \cite{cov} up to a multiplicative constant. So we proved the conjecture in \cite[Section 7.3]{cov}.
\end{Remark}

\end{document}